\documentclass[leqno, 10pt, a4paper]{amsart}
\usepackage{amsmath}
\usepackage{amssymb, latexsym, mathrsfs, mathabx, dsfont, a4wide, bbm, color, import, xifthen, pdfpages, transparent}
\usepackage{graphicx}
\usepackage{caption}
\usepackage{color}

\theoremstyle{definition}
\newtheorem{definition}{Definition}[section]

\theoremstyle{theorem}
 \newtheorem{theorem}[definition]{Theorem}
 \newtheorem{lemma}[definition]{Lemma}
 \newtheorem{proposition}[definition]{Proposition} 
 \newtheorem{corollary}[definition]{Corollary}

 \newtheorem*{theorem*}{Theorem}
\newtheorem*{proposition*}{Proposition}
\newtheorem*{lemma*}{Lemma}

 \theoremstyle{remark}
 \newtheorem{example}[definition]{Example}
 \newtheorem{remark}[definition]{Remark}

   \newtheorem*{claim*}{Claim}

\newcommand{\op}[1]{\operatorname{#1}}

\newcommand{\norm}[1]{\ensuremath{\|{#1}\|}}

\newcommand{\acou}[2]{\ensuremath{\left\langle #1 , #2 \right\rangle}}

\newcommand{\brak}[1]{\ensuremath{\langle #1\rangle}}
\newcommand{\acoup}[2]{\ensuremath{\left(#1|#2\right)}}

\newcommand{\Tr}{\ensuremath{\op{Tr}}}

\newcommand{\Tra}{\ensuremath{\op{Tr}}}
\newcommand{\Trace}{\ensuremath{\op{Tr}}}

\def\XXint#1#2#3{{\setbox0=\hbox{$#1{#2#3}{\int}$}
\vcenter{\hbox{$#2#3$}}\kern-.5\wd0}}

\newcommand{\Hol}{\op{Hol}}

\newcommand{\C}{\ensuremath{\mathbb{C}}} 
 
\newcommand{\N}{\ensuremath{\mathbb{N}}} 
 
\newcommand{\R}{\ensuremath{\mathbb{R}}} 
\newcommand{\bS}{\ensuremath{\mathbb{S}}} 
\newcommand{\T}{\ensuremath{\mathbb{T}}} 
\newcommand{\Z}{\ensuremath{\mathbb{Z}}}

\newcommand{\Rn}{\ensuremath{\R^{n}}}

\newcommand{\fp}{\ensuremath{\mathfrak{p}}}

\newcommand{\Ca}[1]{\ensuremath{\mathcal{#1}}}
\newcommand{\cA}{\mathscr{A}}

\newcommand{\cE}{\Ca{E}}
\newcommand{\cF}{\ensuremath{\mathscr{F}}}

\newcommand{\cH}{\ensuremath{\mathscr{H}}}

\newcommand{\cK}{\ensuremath{\mathscr{K}}}
\newcommand{\cL}{\ensuremath{\mathscr{L}}}

\newcommand{\cS}{\ensuremath{\mathscr{S}}}

\newcommand{\cU}{\ensuremath{\mathscr{U}}}

\newcommand{\sB}{\mathscr{B}}
\newcommand{\sE}{\mathscr{E}}
\newcommand{\sF}{\mathscr{F}}

\newcommand{\sU}{\mathscr{U}}

\newcommand{\psido}{$\Psi$DO} 
\newcommand{\psidos}{$\Psi$DOs}

\newcommand{\stS}{\mathbb{S}}

\def\dba{{\mathchar'26\mkern-12mu d}}
\newcommand{\dbar}{{\, \dba}}

\newcommand{\dist}{\op{dist}}

\newcommand{\Sp}{\op{Sp}}

\newcommand{\subsubset}{\subset\!\subset}

\numberwithin{equation}{section}

\begin{document}

\title{Functional Calculus for Elliptic Operators on Noncommutative Tori, I}

\author{Gihyun Lee}
 \address{Department of Mathematical Sciences, Seoul National University, Seoul, South Korea}
 \email{gihyun.math@gmail.com}

\author{Rapha\"el Ponge}
 \address{School of Mathematics, Sichuan University, Chengdu, China}
 \email{ponge.math@icloud.com}

 \thanks{The research for this article was partially supported by 
  NRF grants 2013R1A1A2008802 and 2016R1D1A1B01015971 (South Korea).}
 
\begin{abstract}
 In this paper, we introduce a parametric pseudodifferential calculus on noncommutative $n$-tori which is a natural nest for resolvents of elliptic pseudodifferential operators. Unlike in some previous approaches to parametric pseudodifferential calculi, our parametric pseudodifferential calculus contains resolvents of elliptic pseudodifferential operators that need not be differential operators. As an application we show that complex powers of positive elliptic pseudodifferential operators on noncommutative $n$-tori are pseudodifferential operators. This confirms a claim of Fathi-Ghorbanpour-Khalkhali.
\end{abstract}

\maketitle 

\section{Introduction}
Noncommutative tori are among the most well known examples of noncommutative spaces. For instance, noncommutative 2-tori naturally arise from actions on circles by irrational rotations. In general, the (smooth) noncommutative torus associated with a given anti-symmetric real $n\times n$-matrix $\theta=(\theta_{jk})$ is a (Fr\'echet) $*$-algebra $\cA_\theta$ that is generated by unitaries $U_1,\ldots, U_n$ subject to the relations, 
\begin{equation*}
 U_kU_j = e^{2i\pi \theta_{jk}} U_jU_k, \qquad j,k=1,\ldots, n. 
\end{equation*}
 When $\theta=0$ we recover the algebra $C^\infty(\T^n)$ of smooth functions on the ordinary torus $\T^n=(\R\slash \Z)^n$. (We refer to Section~\ref{sec:NCtori} for a summary of the main definitions and properties regarding noncommutative tori.) The recent work of Connes-Tretkoff~\cite{CT:Baltimore11} and Connes-Moscovici~\cite{CM:JAMS14} is currently spurring on an intensive research activity on the construction of a full differential geometric apparatus on noncommutative tori (see, e.g., \cite{CF:MJM19, CM:JAMS14, CT:Baltimore11, DS:SIGMA15, FK:JNCG12,  FK:JNCG15, FGK:arXiv16, LM:GAFA16, LNP:TAMS16, Liu:arXiv18b, Ro:SIGMA13}). One specific issue is the taking into account of the non-triviality of the Takesaki-Tomita modular automorphism group. This is a purely noncommutative phenomenon. 

An important tool in this line of research is the (local) short time asymptotic for heat semi-group traces $\Tr[ae^{-tP}]$, where $P$ is some Laplace-type operator and $a\in \cA_\theta$. In particular, the second coefficient in this asymptotic is naturally interpreted as the so-called \emph{modular} scalar curvature (see~\cite{CM:JAMS14, CT:Baltimore11}). The derivation of this asymptotic heavily relies on the pseudodifferential calculus for $C^*$-dynamical systems announced in the notes of Connes~\cite{Co:CRAS80} and Baaj~\cite{Ba:CRAS88}. It was only recently that detailed accounts on this \psido\ calculus became available~(see~\cite{HLP:Part1, HLP:Part2, LM:GAFA16}; see also~\cite{Ta:JPCS18}). Nevertheless, there is a general understanding that the needed asymptotics can be obtained by implementing in the setting of noncommutative tori the approach to the heat kernel asymptotic of Gilkey~\cite{Gi:CRC95} (see~\cite{ CM:JAMS14, CT:Baltimore11, LM:GAFA16}). This actually requires a version with parameter of the pseudodifferential calculus on noncommutative tori that is similar to the parametric \psido\ calculus of Gilkey~\cite{Gi:CRC95} and Shubin~\cite{Sh:Springer01} (see~\cite{LM:GAFA16}). 

The parametric \psido\ calculus of~\cite{Gi:CRC95, Sh:Springer01} is a natural nest for the resolvent $(P-\lambda)^{-1}$ when $P$ is an elliptic \emph{differential} operator. However, when $P$ is more generally a non-differential elliptic \psido\ the resolvent   $(P-\lambda)^{-1}$ falls out this parametric calculus. In fact, the calculus does not contain non-differential \psidos, and so in this case even $P$ and $P-\lambda$ are not in this parametric calculus. As a result, the approach of~\cite{Gi:CRC95} does not allow us to derive short-time asymptotics for traces $\Tr[Ae^{-tP}]$ when $P$ or $A$ are not differential operators. Similarly, the approach of~\cite{Sh:Springer01} does not allow us to construct complex powers $P^s$, $s\in \C$, when $P$ is a ``non-differential'' elliptic \psido\ (compare~\cite{Se:PSPM67}). One approach to remedy is to use the weakly parametric pseudodifferential calculus of Grubb-Seeley~\cite{Gr:Birkhauser96, GS:IM95}.

The aim of this paper is to construct a parametric \psido\ calculus on noncommutative tori which is a natural receptacle for the resolvent $(P-\lambda)^{-1}$, where $P$ is an elliptic \psido\ that need not be a differential operator. This parametric \psido\ calculus is inspired by the parametric Heisenberg calculus of~\cite{Po:PhD, Po:CRAS01}. Its construction avoids some of the technical difficulties of the construction of the weakly parametric pseudodifferential calculus in~\cite{Gr:Birkhauser96, GS:IM95}. In addition, it contains as a sub-calculus the version for noncommutative tori of the parametric pseudodifferential calculus of~\cite{Gi:CRC95, Sh:Springer01} (see~\cite{Po:Heat}). This paper can considered as a first step toward a general functional calculus for elliptic pseudodifferential operators on noncommutative tori. In forthcoming papers we are planning on presenting applications of the parametric \psido\ calculus of this paper toward the following objectives:
\begin{enumerate}
 \item[(i)] Short-time asymptotics for traces $\Tr[a e^{-tP}]$, where $P$ is an elliptic \emph{differential} operator and $a\in \cA_\theta$. 
  
 \item[(ii)] Construction of complex powers $P^s$, $s\in \C$, of elliptic pseudodifferential operators $P$ as holomorphic families of pseudodifferential operators. 
 
 \item[(iii)] Precise study of the singularities of zeta functions $\zeta(P;s)=\Tr [P^{-s}]$ of elliptic pseudodifferential operators $P$ together with corresponding variation formulas when we let $P$ vary. 
 
 \item[(iv)] Extension of (i) to short-time asymptotics for traces $\Tr[A e^{-tP}]$, where $A$ and $P$ are allowed to be \emph{non-differential} \psidos.  
 \end{enumerate}

In~\cite{Gi:CRC95, Sh:Springer01} the parametric \psidos\ are associated with symbols with parameter, 
\begin{equation}
 p(x,\xi;\lambda)\sim \sum_{j\geq 0} p_{m-j}(x,\xi;\lambda), \qquad  p_{m-j}(x,t\xi;t^w\lambda)=t^{m-j} p_{m-j}(x,\xi;\lambda),
 \label{eq:Intro.asymptotic-expansion-parametric-symbol-GS}  
\end{equation}
where $w$ is a given positive number and the asymptotic expansion is meant in the sense that the remainder terms are $\op{O}((|\xi|+|\lambda|^{\frac1{w}})^{-N})$ as $|\xi|+|\lambda|\rightarrow \infty$ with $N$ arbitrarily large. 
Symbols of differential operators $p(x,\xi)= \sum a_\alpha(x) \xi^\alpha$ fit into this framework, but the symbols of classical \psidos\ $p(x,\xi)\sim \sum_{j\geq 0} p_{m-j}(x,\xi)$ do not fit in as soon as there are infinitely many non-zero homogeneous components $p_{m-j}(x,\xi)$. 

The observation in~\cite{Po:PhD, Po:CRAS01} is that we may afford to be somewhat loose with $\lambda$-decay. 
Namely, it is enough to require the remainder terms in the asymptotic expansion~(\ref{eq:Intro.asymptotic-expansion-parametric-symbol-GS}) to be $\op{O}(|\lambda|^d|\xi|^{-N})$ where $N$ is still arbitrarily large, but $d$ is some fixed real number. This allows us to construct complex powers $P^s$, $s\in \C$, of a given elliptic \psido\ $P$ as holomorphic families of \psidos, including when $P$ is not a differential operator. With some additional work this allows us to study the singularities of the zeta functions $\zeta(P; s)=\Tr[P^{-s}]$ and $\zeta(P, A; s)=\Tr[AP^{-s}]$ (where $A$ is a \psido) together with establishing uniform boundedness along vertical lines (at least when $P$ is strongly elliptic). 
A standard Mellin transform argument (see, e.g., \cite{GS:JGA96}) then provides with short time asymptotics for $\Tr[Ae^{-tP}]$. Therefore, this approach provides us with a  sensible alternative to using the weakly parametric pseudodifferential calculus of~\cite{Gr:Birkhauser96, GS:IM95}. 

The parametric pseudodifferential calculus presented in this paper implements this approach on noncommutative tori. We refer to Section~\ref{sec:PsiDOs} for a review of the main definitions and properties of the pseudodifferential calculus on noncommutative tori. Given any $m\in \R$, the space of $m$-th order (standard) symbols $\stS^m(\R^n; \cA_\theta)$ consists of maps $\rho(\xi)\in C^\infty(\R^n;\cA_\theta)$ such that $\partial_\xi^\beta \rho(\xi)$ is $\op{O}(|\xi|^{m-|\beta|})$ in $\cA_\theta$. The \psido\ associated with a symbol $\rho(\xi)\in \stS^m(\R^n; \cA_\theta)$ is the (continuous) linear operator $P_\rho:\cA_\theta\rightarrow \cA_\theta$ given by
\begin{equation}
 P_\rho u = (2\pi)^{-n} \iint e^{is\cdot\xi} \rho(\xi) \alpha_{-s}(u) dsd\xi, \qquad u \in \cA_\theta,
 \label{eq:Intro.quantization}
\end{equation}
where the above integral is meant as an oscillating integral over $\R^n\times \R^n$ and $\R^n\times \cA_\theta \ni  (s,u)\rightarrow \alpha_s(u)$ is the 
$*$-action of $\R^n$ on $\cA_\theta$ such that $\alpha_s(U^m)=e^{is\cdot m}U^m$, $m\in \Z^n$. We are more especially interested in \psidos\ associated with classical symbols. They are symbols with an asymptotic expansion $\rho(\xi)\sim \sum_{j\geq 0} \rho_{m-j}(\xi)$, where $ \rho_{m-j}(\xi)$ is homogeneous of degree~$m-j$. 

A first issue to address in the construction of the parametric pseudodifferential calculus is the general shape of the domains of the parameters $\lambda$. For deriving short-time heat trace asymptotics we may take these domains to be cones (where by a cone we shall always mean a cone with vertex at the origin). However, the construction of complex powers of elliptic operators involves integrating over contours in the complex plane  that wind around the origin.  Such contours cannot be inside a cone with vertex at the origin. We take the domains to be (open) \emph{pseudo-cones}, where by a pseudo-cone we mean a connected set that agrees with a cone  off some open disk about the origin (see Section~\ref{sec:Parametric-symbols} for the precise definition). For instance, we obtain pseudo-cones by pasting/cutting disks or annuli to/from cones. 

Given a pseudo-cone $\Lambda$ and $d\in \R$, a $\Hol^d(\Lambda)$-family in a locally convex space $\sE$ is a holomophic family that is $\op{O}(|\Lambda|^d)$ in $\sE$ as $\lambda$ goes to $\infty$ (the precise definition requires some uniformity condition; see Definition~\ref{def:symbols.Hold-sE}). Intuitively speaking, the parametric \psido\ calculus of this paper is meant to be a $\Hol^d$-family version of the \psido\ calculus on noncommutative tori. 
The spaces of standard symbols $\stS^m(\R^n;\cA_\theta)$ mentioned above are Fr\'echet spaces. 
Given a pseudo-cone $\Lambda$ and $m,d\in \R$, the space of standard parametric symbols $\stS^{m,d}(\R^n\times \Lambda;\cA_\theta)$ is just the space of $\Hol^d(\Lambda)$-families in $\stS^m(\R^n;\cA_\theta)$. 
The space of (classical) parametric symbols $S^{m,d}(\R^n\times \Lambda; \cA_\theta)$ consists of parametric symbols that have an asymptotic expansion, 
\begin{equation*}
 \rho(\xi;\lambda)\sim \sum_{j\geq 0} \rho_{m-j}(\xi;\lambda), \qquad \rho_{m-j}(t\xi;t^w\lambda)=t^{m-j} \rho_{m-j}(\xi;\lambda), \quad t>0, 
\end{equation*}
where $w$ is some positive number and the asymptotic expansion is meant in the sense that the remainder terms can be made to be $\op{O}(|\lambda|^d)\op{O}(|\xi|^{-N})$ in $\cA_\theta$ with $N$ arbitrary large (see Section~\ref{sec:Parametric-symbols} for the precise meaning). 

 If we denote by $\Theta$ the conical part of $\Lambda$ (i.e., $\Lambda$ agrees with $\Theta$ off some open disk about the origin), then the homogeneous symbols $\rho_{m-j}(\xi;\lambda)$ are defined on open sets of the form, 
\begin{equation*}
 \Omega_c(\Theta)=\big\{ (\xi,\lambda)\in(\Rn\setminus 0)\times\C; \ \text{$\lambda\in\Theta$ or $|\lambda|<c|\xi|^w$} \big\} .
\end{equation*}
Suitably cut-off homogeneous symbols give rise to standard parametric symbols (Lemma~\ref{lem:Parameter.homogeneous-symbol-estimate}). 
This implies that classical  parametric symbols are standard parametric symbols (see Proposition~\ref{symbols:inclusion-classical-standard} for the precise relationship). 
There is also a version of Borel's lemma for parametric symbols (see Proposition~\ref{prop:Parameter.Borel-for-classical-symbol}). 

A parametric \psido\ is a family of \psidos\ $P_\rho(\lambda)$ associated via~(\ref{eq:Intro.quantization}) with a parametric symbol $\rho(\xi;\lambda)$. We denote by $\Psi^{m,d}(\cA_\theta; \Lambda)$ the class of parametric \psidos\ that have a classical parametric symbol in $S^{m,d}(\R^n\times \Lambda; \cA_\theta)$. The $\Hol^d$-approach makes it simple to extend to parametric \psidos\ the main properties of (non-parametric) \psidos\ via continuity arguments. 
This includes composition (Proposition~\ref{prop:Parameter.composition-PsiDOs}), Sobolev regularity properties (Proposition~\ref{prop:PsiDOs-parameter.classical-PsiDO-Sobolev-mapping-properties}), and Schatten-class properties, including trace-class properties (see Proposition~\ref{prop:PsiDOs-parameter.Schatten-classical} and Proposition~\ref{prop:PsiDOs-parameter.trace-class}). 
For sake of completeness we also introduce a class of  parametric toroidal \psidos\ and show it agrees with our original class of parametric \psidos\ (Proposition~\ref{prop:Parametric-PsiDOs.equality-toroidal}). This leads us to a characterization of parametric smoothing operators (Proposition~\ref{prop:PsiDOs-parameter.smoothing-operators-with-parameter-characterization}). 

The heart of the matter of the paper is Section~\ref{sec:Resolvent}. Given an elliptic operator $P$ of order $w>0$ with (classical) symbol $\rho(\xi)\sim \sum_{j\geq 0} \rho_{w-j}(\xi)$ we introduce its \emph{elliptic parameter cone}, 
\[
      \Theta(P)= \bigcup_{\xi \in \R^n\setminus 0} [\C^*\setminus \Sp(\rho_w(\xi))].
\] 
We shall say that $P$ is \emph{elliptic with parameter} when $\Theta(P)\neq \emptyset$ (which we assume thereon). We may regard $P-\lambda$ as an element of $\Psi^{w,1}(\cA_\theta;\Lambda)$, even when $P$ is not a differential operator. When $P$ is elliptic with parameter the parametric symbolic calculus of this paper allows us to get an explicit construction of a parametrix for $P-\lambda$ as an element of $\Psi^{-w,-1}(\cA_\theta; \Lambda)$, where $\Lambda$ is any pseudo-cone that is obtained by glueing to $\Theta(P)$ a disk about the origin (Theorem~\ref{thm:Resolvent.parametrix-with-parameter}).  

The existence of such a parametrix allows to give a somewhat precise localization of the spectrum of $P$ (see Theorem~\ref{thm:Resolvent.P-has-discrete-spectrum-resolvent-estimate}). More precisely, given any cone $\Theta$ such that $\overline{\Theta}\setminus 0\subset \Theta(P)$, there are at most finitely many eigenvalues of $P$ that can be contained in $\Theta$, and, furthermore, $\|(P-\lambda)^{-1}\|=\op{O}(|\lambda|^{-1})$ as $\lambda\rightarrow \infty$ within $\Theta$. This implies that any ray contained in $\Theta(P)$ that does not cross the spectrum of $P$ is automatically a ray of minimal growth (see Corollary~\ref{cor:Resolvent.ray-with-no-eigenvalue-is-a-ray-of-minimal-growth}). The existence of such a ray is an important ingredient in the complex powers of elliptic operators (see, e.g., \cite{Se:PSPM67}). 

If we let $\check{\Theta}(P)$ be the cone obtained from $\Theta(P)$ by deleting rays through eigenvalues of $P$, then $\check{\Theta}(P)$ is still a non-empty open cone (see Lemma~\ref{lem:Resolvent.checkThetaP-open}). The main result of Section~\ref{sec:Resolvent} states that the resolvent $(P-\lambda)^{-1}$ is an element of $\Psi^{-w,-1}(\cA_\theta; \Lambda(P))$, where $\Lambda(P)$ is obtained by glueing to $\check{\Theta}(P)$ any disk or marked disk about the origin that does not contain any eigenvalue of $P$ (Theorem~\ref{thm:Resolvent.resolvent-is-psido-with-parameter}).  
Combining this result with the properties of parametric \psidos\ provides us with sharp Sobolev regularity properties and Schatten-class properties for the resolvent $(P-\lambda)^{-1}$. Namely, for any $s\in \R$, the resolvent $(P-\lambda)^{-1}$ is an $\Hol^{-1+a}(\Lambda(P))$-family in $\cL(\cH_\theta^{(s)},\cH_\theta^{(s+aw)})$ for every $a\in [0,1]$ (Proposition~\ref{prop:Resolvent.Sobolev}) and, if we set $p=nw^{-1}$, then this is also an $\Hol^{-1+pq^{-1}}(\Lambda(P))$-family in $\cL^{(q,\infty)}$ for every $q\geq p$ (Proposition~\ref{prop:Resolvent.Schatten}). 
Furthermore, we have sharp asymptotic estimates for the traces of operators of the form, 
\begin{equation*}
A_0 (P-\lambda)^{-1}A_1 \cdots  A_{N-1}(P-\lambda)^{-1} A_N, \qquad A_j\in \Psi^{a_j}(\cA_\theta),
\end{equation*}
provided that $-Nw+a_0+\cdots + a_N<-n$ (see Proposition~\ref{prop:Resolvent.trace-Q}). Such operators often appear in resolvent expansions. 

 In~\cite{LP:Powers} we shall apply the above results to construct complex powers $P^{s}$ of an elliptic \psido\ as holomorphic families of \psidos. As a preview of this we show when $P\in \Psi^{w}(\cA_\theta)$ is an elliptic \psido\ with positive principal symbol  and  positive spectrum its complex powers $P^{z}$, $z\in \C$,  are \psidos\  (Theorem~\ref{thm:Powers.powers}). This implies that, for any elliptic \psido\ $P\in \Psi^{w}(\cA_\theta)$ the absolute value $|P|=\sqrt{P^*P}$ and its  
complex powers $|P|^z$, $z\in \C$, are \psidos\ as well (Theorem~\ref{thm:Powers.powers-absolute-value}). It was mentioned without proof in~\cite{FGK:MPAG17} that complex powers of positive elliptic \psidos\ on noncommutative 2-tori are \psidos. Thus, our results confirm the claim of~\cite{FGK:MPAG17}. In any case, we postpone to~\cite{LP:Powers} a more thorough account on complex powers of elliptic \psidos\ on noncommutative tori.

The paper is organized as follows. In Section~\ref{sec:NCtori}, we give an overview the main facts on noncommutative tori. In Section~\ref{sec:PsiDOs}, we review the main definitions and properties regarding pseudodifferential operators on noncommutative tori. In Section~\ref{sec:Parametric-symbols}, we introduce our classes of symbols with parameter and derive some of their properties. In Section~\ref{sec:parametric-PsiDOs}, we define our parametric pseudodifferential operators and establish their main properties. In Section~\ref{sec:Resolvent}, given any \psido\ $P$ that is elliptic with parameter, we construct a parametrix for $P-\lambda$ in our parametric \psido\ calculus and show that the resolvent $(P-\lambda)^{-1}$ is an element of this calculus.  These results are used in Section~\ref{sec:square-roots} to show that the complex powers of positive elliptic \psidos\ are \psidos.
In Appendix~\ref{sec:Improper-Integrals}, we review basic facts on improper integrals with values in locally convex spaces. In Appendix~\ref{app:topo-smoothing}, we describe the strong topology of $\cL(\cA_\theta',\cA_\theta)$ in terms of Sobolev norms. This includes a characterization of bounded sets in $\cA_\theta'$.

\subsection*{Acknowlegements} G.L.~acknowledges the support of BK21 PLUS SNU, Mathematical Sciences Division (South Korea). R.P.~wishes to thank University of New South Wales (Sydney, Australia) and University of Qu\'ebec at Montr\'eal (Montr\'eal, Canada) for their hospitality during the preparation of this manuscript. 

\section{Noncommutative Tori} \label{sec:NCtori}
In this section, we review the main definitions and properties of noncommutative $n$-tori, $n\geq 2$. We refer to~\cite{HLP:Part1}, and the references therein (more especially~\cite{Co:NCG, Ri:PJM81, Ri:CM90}), for more comprehensive accounts, including proofs of the results mentioned in this section.
 
Throughout this paper, we let $\theta =(\theta_{jk})$ be a real anti-symmetric $n\times n$-matrix ($n\geq 2$). We denote by $\theta_1, \ldots, \theta_n$ its column vectors. In what follows we also let $\T^n=\R^n\slash 2\pi \Z^n$ be the ordinary $n$-torus, and we equip $L^2(\T^n)$ with the  inner product, 
\begin{equation} \label{eq:NCtori.innerproduct-L2}
 \acoup{\xi}{\eta}= \int_{\T^n} \xi(x)\overline{\eta(x)}\dbar x, \qquad \xi, \eta \in L^2(\T^n), 
\end{equation}
 where we have set $\dbar x= (2\pi)^{-n} dx$.  For $j=1,\ldots, n$, let $U_j:L^2(\T^n)\rightarrow L^2(\T^n)$ be the unitary operator defined by 
 \begin{equation*}
 \left( U_j\xi\right)(x)= e^{ix_j} \xi\left( x+\pi \theta_j\right), \qquad \xi \in L^2(\T^n). 
\end{equation*}
 We then have the relations, 
 \begin{equation} \label{eq:NCtori.unitaries-relations}
 U_kU_j = e^{2i\pi \theta_{jk}} U_jU_k, \qquad j,k=1, \ldots, n. 
\end{equation}
The \emph{noncommutative torus} $A_\theta$ is the $C^*$-algebra generated by the unitary operators $U_1, \ldots, U_n$.  For $\theta=0$ we obtain the $C^*$-algebra $C^{0}(\T^n)$ of continuous functions on the ordinary $n$-torus $\T^n$. Note that~(\ref{eq:NCtori.unitaries-relations}) implies that $A_\theta$ is the closure in $\cL(L^2(\T^n))$ of the algebra $\cA_\theta^0$, where $\cA_\theta^0$ is the span of the unitary operators, 
 \begin{equation*}
 U^k:=U_1^{k_1} \cdots U_n^{k_n}, \qquad k=(k_1,\ldots, k_n)\in \Z^n. 
\end{equation*}

 Let $\tau:\cL(L^2(\T^n))\rightarrow \C$ be the state defined by the constant function $1$, i.e., 
 \begin{equation*}
 \tau (T)= \acoup{T1}{1}=\int_{\T^n} (T1)(x) \dbar x, \qquad T\in \cL\left(L^2(\T^n)\right).
\end{equation*}
In particular, we have
 \begin{equation*} 
 \tau\left( U^k\right) = \left\{
\begin{array}{ll}
 1 &  \text{if $k=0$},  \\
 0 &  \text{otherwise.}
 \end{array}\right. 
\end{equation*}
This induces a continuous linear trace on the $C^*$-algebra $A_\theta$.

The GNS construction (see, e.g., \cite{Ar:Springer81}) allows us to associate with $\tau$ a $*$-representation of $A_\theta$ as follows. Let $\acoup{\cdot}{\cdot}$ be the sesquilinear form on $A_\theta$ defined by
\begin{equation}
 \acoup{u}{v} = \tau\left( uv^* \right), \qquad u,v\in A_\theta. 
 \label{eq:NCtori.cAtheta-innerproduct}
\end{equation}
The family $\{ U^k; k \in \Z^n\}$ is orthonormal with respect to this sesquilinear form. In particular, we have a pre-inner product on the dense subalgebra $\cA_\theta^0$. 

\begin{definition}
 $\cH_\theta$ is the Hilbert space arising from the completion of $\cA_\theta^0$ with respect to the pre-inner product~(\ref{eq:NCtori.cAtheta-innerproduct}). 
\end{definition}

When $\theta=0$ we recover the Hilbert space $L^2(\T^n)$ with the inner product~(\ref{eq:NCtori.innerproduct-L2}). In what follows we shall denote by $\|\cdot\|_0$ the norm of $\cH_\theta$. This notation allows us to distinguish it from the norm of $A_\theta$, which we denote by $\|\cdot\|$.

By construction $(U^k)_{k \in \Z^n}$ is an orthonormal basis of $\cH_\theta$. Thus, every $u\in \cH_\theta$ can be uniquely written as 
\begin{equation} \label{eq:NCtori.Fourier-series-u}
 u =\sum_{k \in \Z^n} u_k U^k, \qquad u_k=\acoup{u}{U^k}, 
\end{equation}
where the series converges in $\cH_\theta$. When $\theta =0$ we recover the Fourier series decomposition in  $L^2(\T^n)$. By analogy with the case $\theta=0$ we shall call the series $\sum_{k \in \Z^n} u_k U^k$ in~(\ref{eq:NCtori.Fourier-series-u}) the Fourier series of $u\in \cH_\theta$. 

\begin{proposition}\label{prop:NCTori.GNS-representation}
The following holds.
\begin{enumerate}
\item The multiplication of $\cA_\theta^0$ uniquely extends to a continuous bilinear map $A_\theta\times \cH_\theta \rightarrow \cH_\theta$. This provides us with a 
         unital $*$-representation of $A_\theta$. In particular, we have
          \begin{equation*} 
                      \left\| u \right\| = \sup_{\|v\|_0=1} \|uv\|_0 \qquad \forall u \in A_\theta. 
           \end{equation*}
\item The inclusion of $\cA_\theta^0$ into $\cH_\theta$ uniquely extends to a continuous embedding of $A_\theta$ into $\cH_\theta$. 
\end{enumerate}
\end{proposition}

In particular, the 2nd part  allows us to identify any $u \in A_\theta$ with the sum of its Fourier series in $\cH_\theta$. In general this Fourier series need not converge in $A_\theta$.

The natural action of $\R^n$ on $\T^n$ by translation gives rise to an action on $\cL(L^2(\T^n))$ which induces a $*$-action $(s,u)\rightarrow \alpha_s(u)$ on $A_\theta$ such that
\begin{equation*}
\alpha_s(U^k)= e^{is\cdot k} U^k, \qquad  \text{for all $k\in \Z^n$ and $s\in \R^n$}. 
\end{equation*}
This action is strongly continuous, and so we obtain a $C^*$-dynamical system $(A_\theta, \R^n, \alpha)$. We are especially interested in the subalgebra of the smooth elements of this $C^*$-dynamical system, i.e., the \emph{smooth noncommutative torus}, 
\begin{equation*}
 \cA_\theta:=\left\{ u \in A_\theta; \ \alpha_s(u) \in C^\infty(\R^n; A_\theta)\right\}. 
\end{equation*}
All the unitaries $U^k$, $k\in \Z^n$, are contained in $\cA_\theta$, and so $\cA_\theta$ is a dense subalgebra of $A_\theta$. When $\theta=0$ we recover the algebra $C^\infty(\T^n)$ of smooth functions on the ordinary torus $\T^n$.

For $j=1,\ldots, n$, let $\delta_j:\cA_\theta\rightarrow \cA_\theta$ be the  derivation defined by 
\begin{equation*}
 \delta_j(u) = D_{s_j} \alpha_s(u)|_{s=0}, \qquad u\in \cA_\theta, 
\end{equation*}
where we have set $D_{s_j}=\frac{1}{i}\partial_{s_j}$. When $\theta=0$ the derivation $\delta_j$ is just the derivation $D_{x_j}=\frac{1}{i}\frac{\partial}{\partial x_j}$ on $C^\infty(\T^n)$. In general, for $j,l=1,\ldots, n$, we have
\begin{equation*}
 \delta_j(U_l) = \left\{ 
 \begin{array}{ll}
 U_j & \text{if $l=j$},\\
 0 & \text{if $l\neq j$}. 
\end{array}\right.
\end{equation*}
More generally, given $u\in \cA_\theta$ and $\beta \in \N_0^n$, $|\beta|=N$, we define 
\begin{equation*}
 \delta^\beta(u) = D_s^\beta \alpha_s(u)|_{s=0} = \delta_1^{\beta_1} \cdots \delta_n^{\beta_n}(u). 
\end{equation*}
In what follows, we endow $\cA_\theta$ with the locally convex topology defined by the semi-norms,
\begin{equation}
 \cA_\theta \ni u \longrightarrow \left\|\delta^\beta (u)\right\| ,  \qquad \beta\in \N_0^n. 
\label{eq:NCtori.cAtheta-semi-norms}
\end{equation}
With the involution inherited from $A_\theta$ this turns $\cA_\theta$ into a (unital) Fr\'echet $*$-algebra. It can be further shown that $\cA_\theta$ is even a nuclear Fr\'echet-Montel space (see, e.g.,~\cite{HLP:Part1}). 
 
In what follows, we denote by $\cA_\theta^{-1}$ (resp., $A_\theta^{-1}$) the group of invertible elements of $\cA_\theta$ (resp., $A_\theta$). Given any $u\in \cA_\theta$ we shall denote by $\Sp(u)$ its \emph{spectrum}, i.e., 
\begin{equation*}
 \Sp(u)=\left\{\lambda \in \C;\  u-\lambda \not\in \cA_\theta^{-1}\right\}. 
\end{equation*}

\begin{proposition}[see \cite{Co:AdvM81, HLP:Part1}] \label{prop:NCtori.invertibility-cAtheta}
The following holds. 
\begin{enumerate}
 \item We have $\cA_\theta^{-1}= A_\theta^{-1} \cap \cA_\theta$. 

 \item $\cA_\theta$ is stable under holomorphic functional calculus.
 
 \item For all $u\in \cA_\theta$, we have
\begin{equation*}
 \Sp (u) =\left\{\lambda \in \C; \ \text{$u-\lambda:\cH_\theta\rightarrow \cH_\theta$ is not a bijection}\right\}. 
\end{equation*}
\end{enumerate}
\end{proposition}

Let $\cA_\theta'$ be the topological dual of $\cA_\theta$. We equip it with its {strong topology}, i.e., the locally convex topology generated by the semi-norms,
\begin{equation*}
v\longrightarrow\sup_{u\in B}|\acou{v}{u}| , \qquad \text{$B\subset \cA_\theta$ bounded}. 
\end{equation*}
It is tempting  to think of elements of $\cA_\theta'$ as distributions on $\cA_\theta$. This is consistent with the definition of distributions on $\mathbb{T}^{n}$ as continuous linear forms on $C^\infty(\mathbb{T}^n)$. 
 Any $u \in \cA_\theta$ defines a linear form on $\cA_\theta$ by
\begin{equation*}
 \acou{u}{v} =\tau(uv) \qquad \text{for all $v\in \cA_\theta$}.  
\end{equation*}
 Note that, for all $u,v \in \cA_\theta$, we have 
 \begin{equation} \label{eq:NCtori.distrb-innerproduct-eq}
  \acou{u}{v} =\acoup{v}{u^*}=\acoup{u}{v^*}.
\end{equation}
 In particular, given any $u \in \cA_\theta$, the map $v\rightarrow \acou{u}{v}$ is a continuous linear form on $\cA_\theta$. This gives rise to a continuous embedding of $\cA_\theta$ into $\cA_\theta'$. In view of~(\ref{eq:NCtori.distrb-innerproduct-eq}) it uniquely extends to a continuous embedding of $\cH_\theta$ into $\cA_\theta'$. This allows us to extend the definition of Fourier series to $\cA_\theta'$. Namely, given any $v\in \cA_\theta'$ its Fourier series is the series, 
\begin{equation}
 \sum_{k\in \Z^n} v_k U^k, \qquad \text{where}\ v_k:= \acou{v}{(U^k)^*}. 
 \label{eq:NCtori.Fourier-series}
\end{equation}
Here the unitaries $U^k$, $k\in \Z^n$, are regarded as elements of $\cA_\theta'$. It can be shown that every $u\in \cA_\theta'$ is the sum of its Fourier series in $\cA_\theta'$ (see~\cite{HLP:Part1}). 

\section{Pseudodifferential Operators on Noncommutative Tori} \label{sec:PsiDOs}
In this section, we recall the main definitions and properties of the pseudodifferential calculus on noncommutative tori~\cite{Ba:CRAS88, Co:CRAS80,  HLP:Part1, HLP:Part2}. The exposition closely follows~\cite{HLP:Part1, HLP:Part2} (see also~\cite{Ta:JPCS18}). 

\subsection{Symbols on noncommutative tori}
There are various classes of symbols on noncommutative tori. First, we have the \emph{standard symbols}.  

\begin{definition}[Standard Symbols; see~\cite{Ba:CRAS88, Co:CRAS80}]
$\stS^m (\Rn ; \cA_\theta)$, $m\in\R$, consists of maps $\rho(\xi)\in C^\infty (\Rn ; \cA_\theta)$ such that, for all multi-orders $\alpha$ and $\beta$, there exists $C_{\alpha \beta} > 0$ such that
\begin{equation*} 
\norm{\delta^\alpha \partial_\xi^\beta \rho(\xi)} \leq C_{\alpha \beta} \left( 1 + | \xi | \right)^{m - | \beta |} \qquad \forall \xi \in \R^n .
\end{equation*}
\end{definition}

\begin{remark}[see~\cite{HLP:Part1}] \label{rmk:PsiDOs.Frechet-space-symbols}
$\stS^m (\Rn ; \cA_\theta)$ is a Fr\'echet space with respect to the topology generated by the semi-norms, 
\begin{equation} \label{eq:Symbols.standard-symbol-semi-norm-definition}
p_N^{(m)}(\rho):=\sup_{|\alpha|+|\beta|\leq N} \sup_{\xi\in\Rn}(1+|\xi|)^{-m+|\beta|}\norm{\delta^\alpha\partial_\xi^\beta\rho(\xi)}, \qquad N\in\N_0 .
\end{equation}
\end{remark}

\begin{remark}
 Let $(m_j)_{j\geq 0}\subset \R$ be a decreasing sequence such that $m_j\rightarrow -\infty$. Given $\rho(\xi)$ in $C^\infty(\R^n;\cA_\theta)$ and $\rho_j(\xi)\in \stS^{m_j}(\R^n;\cA_\theta)$, $j\geq 0$, we shall write $\rho(\xi) \sim \sum_{j \geq 0} \rho_j(\xi)$ when 
\begin{equation}
 \rho(\xi)-\sum_{j<N} \rho_j(\xi)\in \stS^{m_N}(\R^n;\cA_\theta) \qquad \forall N\geq 0.
 \label{eq:PsiDOs.asymptotic-expansion-standard} 
\end{equation}
Note this implies that $\rho(\xi)\in \stS^{m_0}(\R^n; \cA_\theta)$. 
\end{remark}

\begin{definition} 
  $\cS(\R^n; \cA_\theta)$ consists of maps $\rho(\xi)\in C^\infty (\Rn ; \cA_\theta)$ such that, for all  $N\geq 0$ and multi-orders $\alpha$, $\beta$, there exists $C_{N\alpha \beta} > 0$ such that
\begin{equation*} 
\norm{\delta^\alpha \partial_\xi^\beta \rho(\xi)} \leq C_{N\alpha \beta} \left( 1 + | \xi | \right)^{-N} \qquad \forall \xi \in \R^n .
\end{equation*}
\end{definition}

\begin{remark}\label{rmk:Symbols.symbols-intersection}
$\cS(\R^n; \cA_\theta)=\bigcap_{m\in\R}\stS^m( \R^n;\cA_\theta)$.
\end{remark}

\begin{definition}[Homogeneous Symbols] 
$S_q (\R^n; \cA_\theta )$, $q \in \C$, consists of  maps $\rho(\xi) \in C^\infty(\R^n\backslash 0;\cA_\theta)$ that are homogeneous of degree $q$, i.e., 
$\rho( t \xi ) =t^q \rho(\xi)$ for all $\xi \in \R^n \backslash 0$ and $t > 0$. 
\end{definition}

\begin{remark}
 If $\rho(\xi)\in S_q(\R^n;\cA_\theta)$ and $\chi(\xi)\in C^\infty_c(\R^n)$ is such that $\chi(\xi)=1$ near $\xi=0$, then $(1-\chi(\xi))\rho(\xi)\in \stS^{\Re q}( \R^n;\cA_\theta)$. 
\end{remark}

\begin{definition}[Classical Symbols; see \cite{Ba:CRAS88}]\label{def:Symbols.classicalsymbols}
$S^q (\R^n; \cA_\theta )$, $q \in \C$, consists of maps $\rho(\xi)\in C^\infty(\R^n;\cA_\theta)$ that admit an asymptotic expansion,
\begin{equation*}
\rho(\xi) \sim \sum_{j \geq 0} \rho_{q-j} (\xi),  \qquad \rho_{q-j} \in S_{q-j} (\R^n; \cA_\theta ), 
\end{equation*}
where $\sim$ means that, for all $N\geq 0$ and multi-orders $\alpha$, $\beta$, there exists $C_{N\alpha\beta} >0$ such that, for all $\xi \in \R^n$ with $| \xi | \geq 1$, we have
\begin{equation} \label{eq:Symbols.classical-estimates}
\Big\| \delta^\alpha \partial_\xi^\beta \big( \rho - \sum_{j<N} \rho_{q-j} \big)(\xi)\Big\| \leq C_{N\alpha\beta} | \xi |^{\Re{q}-N-| \beta |} .
\end{equation}
\end{definition}

\begin{remark} \label{rmk:Symbols.classical-inclusion}
$S^q(\R^n;\cA_\theta)\subset \stS^{\Re{q}}(\R^n;\cA_\theta)$. 
\end{remark}

\begin{remark}
 Let $\chi(\xi)\in C^\infty_c(\R^n)$ be such that $\chi(\xi)=1$ near $\xi=0$. Then $\rho(\xi)\sim \sum_{j\geq 0} \rho_{q-j}(\xi)$ in the sense of~(\ref{eq:Symbols.classical-estimates}) if and only if 
 $\rho(\xi)\sim \sum_{j\geq 0} (1-\chi(\xi))\rho_{q-j}(\xi)$ in the sense of~(\ref{eq:PsiDOs.asymptotic-expansion-standard}). 
\end{remark}

\begin{example}
 Any polynomial map $\rho(\xi)=\sum_{|\alpha|\leq m} a_\alpha \xi^\alpha$, $a_\alpha\in \cA_\theta$, $m\in \N_0$, is in $S^m(\R^n;\cA_\theta)$.  
\end{example}

\begin{example} \label{ex:Symbols.example-symbol}
 Set $\brak{\xi}=(1+|\xi|^2)^{\frac12}$, $\xi\in \R^n$. Then $\brak{\xi}^q \in S^q(\R^n;\cA_\theta)$ for all $q\in \C$ (see, e.g., \cite{HLP:Part1}). 
\end{example}

\subsection{Amplitudes and oscillating integrals}
We briefly recall the construction of the $\cA_\theta$-valued oscillating integral in~\cite{HLP:Part1} (see also~\cite{Ri:MAMS93}). In this setting the oscillating integral is defined on $\cA_\theta$-valued amplitudes.   

\begin{definition}[Amplitudes; see~\cite{HLP:Part1}]
$A^m ( \Rn \times \Rn ; \cA_\theta )$, $m \in \R$, consists of maps $a(s,\xi)$ in $C^{\infty} (\Rn \times \Rn ; \cA_\theta )$ such that, for all multi-orders $\alpha$, $\beta$, $\gamma$, there is $C_{\alpha \beta \gamma} > 0$ such that
\begin{equation} \label{eq:Amplitudes.amplitudes-estimates}
\norm{\delta^\alpha \partial_s^\beta \partial_\xi^\gamma  a(s,\xi)} \leq C_{\alpha \beta \gamma} \left( 1 + |s| + |\xi| \right)^m \quad \forall (s,\xi) \in \Rn \times \Rn .
\end{equation}
\end{definition}

Each space $A^m(\Rn\times\Rn;\cA_\theta)$, $m\in \R$, is a Fr\'echet space with respect to the locally convex topology generated by the semi-norms,
\begin{equation} \label{eq:Amplitudes.amplitudes-semi-norms}
q_N^{(m)} (a) := \sup_{|\alpha|+|\beta|+|\gamma| \leq N}\sup_{(s,\xi) \in \Rn \times \Rn} (1 + |s| + |\xi| )^{-m} \norm{\delta^\alpha \partial_s^\beta \partial_\xi^\gamma a(s,\xi)}, \quad N\in\N_0 .
\end{equation}

If $a(s,\xi)\in A^m (\Rn\times\Rn;\cA_\theta)$ with $m<-2n$, then the estimates~(\ref{eq:Amplitudes.amplitudes-estimates}) ensure us that the $\cA_\theta$-valued map 
$(s,\xi)\rightarrow e^{is\cdot\xi} a(s,\xi)$ is absolutely integrable on $\R^{n}\times \R^n$ (see~\cite{HLP:Part1}). Therefore,  we may define
\begin{equation} \label{eq:Amplitudes.definition-J0}
J_0(a) := \iint e^{is\cdot\xi} a(s,\xi) ds\dbar\xi , 
\end{equation}
where we have set $\dbar\xi=(2\pi)^{-n}d\xi$. This defines a continuous linear map from $A^m (\Rn\times\Rn;\cA_\theta)$ to $\cA_\theta$ for every $m<-2n$ (see~\cite{HLP:Part1}). 

In the following we set  $A^{+\infty}(\Rn\times\Rn;\cA_\theta):=\bigcup_{m\in\R} A^m(\Rn\times\Rn;\cA_\theta)$. 

\begin{proposition}[see~\cite{HLP:Part1}] \label{prop:Amplitudes.extension-J0}
$J_0$ uniquely extends to a linear map $J:A^{+\infty}(\Rn\times\Rn;\cA_\theta)\rightarrow \cA_\theta$ that induces a continuous linear map on each amplitude space $A^m(\Rn\times\Rn;\cA_\theta)$, $m\in\R$. 
\end{proposition}

More specifically, given $\chi(s,\xi) \in C_c^{\infty} (\Rn \times \Rn)$ such that $\chi(s,\xi)= 1$ near $(s,\xi)=(0,0)$, let $L$ be the differential operator on $\R^{2n}$ defined by 
\begin{equation*}
L:= \chi(s,\xi) + \frac{1-\chi(s,\xi)}{|s|^2 + |\xi|^2} \sum_{1 \leq j \leq n} (\xi_j D_{s_j} + s_j D_{\xi_j}) , 
\end{equation*}
where we have set $D_{x_j}=\frac{1}{i}\partial_{x_j}$, $j=1, \ldots, n$. Note that $L(e^{is\cdot\xi}) = e^{is\cdot\xi}$. Moreover, given any $m\in \R$, it can be shown that the transpose  $L^t$ (in the sense of differential operators) gives rise to a continuous linear map $L^t : A^m (\Rn \times \Rn ; \cA_\theta) \rightarrow A^{m-1} (\Rn \times \Rn ; \cA_\theta)$ (see~\cite{HLP:Part1}). Successive integrations by parts show that, given any $a(s,\xi)\in A^m(\Rn\times\Rn;\cA_\theta)$, $m\in \R$, we have
\begin{equation*} 
J(a)=\iint e^{is\cdot\xi}(L^t)^N [a(s,\xi)]ds\dbar\xi , 
\end{equation*}
where $N$ is any non-negative integer $>m+2n$. 

We call $J(a)$ the \emph{oscillating integral} with amplitude $a(s,\xi)$. It is convenient write it as an integral $\iint e^{is\cdot\xi}a(s,\xi)ds\dbar\xi$. 

\subsection{\psidos~on noncommutative tori}
Let $\rho(\xi)\in \stS^m(\R^n;\cA_\theta)$, $m\in \R$. Given any $u\in \cA_\theta$, it can be shown that $\rho(\xi)\alpha_{-s}(u)$ is an amplitude in 
$A^{m_+}(\Rn\times\Rn;\cA_\theta)$, where $m_+=\op{max}(m,0)$ (see~\cite{HLP:Part1}). Therefore, we can define
\begin{equation*} 
P_\rho u = \iint e^{is\cdot\xi}\rho(\xi)\alpha_{-s}(u)ds\dbar\xi,
\end{equation*}
where the above integral is meant as an oscillating integral. This defines a continuous linear operator $P_\rho:\cA_\theta \rightarrow \cA_\theta$ (see~\cite{HLP:Part1}). 

\begin{definition}
$\Psi^q(\cA_\theta)$,  $q\in \C$, consists of all linear operators $P_\rho:\cA_\theta\rightarrow \cA_\theta$ with $\rho(\xi)$ in $S^q(\R^n; \cA_\theta)$.
\end{definition}

\begin{remark} \label{rem:PsiDOs.symbol-uniqueness}
If $P=P_\rho$ with $\rho(\xi)$ in $S^q(\R^n; \cA_\theta)$, $\rho(\xi)\sim \sum \rho_{q-j}(\xi)$, then $\rho(\xi)$ is called a \emph{symbol} for $P$. This symbol is not unique, but its restriction to $\Z^n$ and its class modulo $\cS(\R^n;\cA_\theta)$ are unique (see Remark~\ref{rmk:PsiDOs.rho(k)-P}). As a result, the homogeneous symbols $\rho_{q-j}(\xi)$ are uniquely determined by $P$. The leading symbol $\rho_q(\xi)$ is called the \emph{principal symbol} of $P$. 
\end{remark}

\begin{example}
 A differential operator on $\cA_\theta$ is of the form $P=\sum_{|\alpha|\leq m}a_\alpha\delta^\alpha$, $a_\alpha\in\cA_\theta$ (see~\cite{Co:CRAS80, Co:NCG}). This is a \psido\ of order $m$ with symbol $\rho(\xi)= \sum a_\alpha \xi^\alpha$ (see~\cite{HLP:Part1}).
\end{example}

\begin{example} \label{eq:PsiDOs.Lambda-to-the-s-PsiDO}
 The  Laplacian $\Delta:= \delta_1^2 + \cdots + \delta_n^2$ is a selfadjoint unbounded operator on $\cH_\theta$ with domain 
 $\op{Dom}(\Delta):=\{ u=\sum_{k\in \Z^n} u_kU^k\in \cH_\theta; \ \sum_{k\in \Z^n} |k|^2 |u_k|^2<\infty\}$. It is isospectral to the Laplacian on $\T^n$. For all $s\in \C$ set $\Lambda^s=(1+\Delta)^s$. Then $\Lambda^s$ is the \psido\ with symbol $\brak{\xi}^s$, and so $\Lambda^s\in \Psi^s(\cA_\theta)$ (see~\cite{HLP:Part1}). In fact, the family $(\Lambda^s)_{s\in \C}$ form a 1-parameter group of \psidos. 
\end{example}

\subsection{Composition of \psidos}
Suppose we are given symbols $\rho_1(\xi)\in\stS^{m_1}(\Rn;\cA_\theta)$, $m_1\in\R$, and $\rho_2(\xi)\in\stS^{m_2}(\Rn;\cA_\theta)$, $m_2\in\R$. As $P_{\rho_1}$ and $P_{\rho_2}$ are linear operators on $\cA_\theta$, the composition $P_{\rho_1}P_{\rho_2}$ makes sense as such an operator.  
In addition, we define the map $\rho_1\sharp\rho_2:\Rn\rightarrow \cA_\theta$ by
\begin{equation} \label{eq:Composition.symbol-sharp}
\rho_1\sharp\rho_2(\xi) = \iint e^{it\cdot\eta}\rho_1(\xi+\eta)\alpha_{-t}[\rho_2(\xi)]dt\dbar\eta , \qquad \xi\in\Rn .
\end{equation}
Here, for every $\xi\in\Rn$, the map $(t,\eta)\rightarrow\rho_1(\xi+\eta)\alpha_{-t}[\rho_2(\xi)]$ is an amplitude, and so the right-hand side makes sense as an oscillating integral (see~\cite{Ba:CRAS88, HLP:Part2}). 

\begin{proposition}[see \cite{Ba:CRAS88, Co:CRAS80, HLP:Part2}] \label{prop:Composition.sharp-continuity-standard-symbol}
Let $\rho_1(\xi)\in \stS^{m_1}(\R^n; \cA_\theta)$ and  $\rho_2(\xi)\in \stS^{m_2}(\R^n; \cA_\theta)$, $m_1,m_2\in \R$. 
\begin{enumerate}
 \item $\rho_1\sharp\rho_2(\xi)\in \stS^{m_1+m_2}(\Rn;\cA_\theta)$, and we have $ \rho_1\sharp\rho_2(\xi) \sim \sum\frac{1}{\alpha !}\partial_\xi^\alpha\rho_1(\xi)\delta^\alpha\rho_2(\xi)$. 

 \item The operators $P_{\rho_1}P_{\rho_2}$ and $P_{\rho_1\sharp \rho_2}$ agree.           
\end{enumerate}
\end{proposition}

In the case of classical \psidos\ we further have the following result.

\begin{proposition}[see \cite{Ba:CRAS88, Co:CRAS80, HLP:Part2}] \label{prop:Composition.composition-PsiDOs}
Let  $P_1\in \Psi^{q_1}(\cA_\theta)$, $q_1\in \C$, have symbol $\rho_1(\xi)\sim\sum_{j\geq 0}\rho_{1,q_1-j}(\xi)$, and let  $P_2\in \Psi^{q_2}(\cA_\theta)$, $q_2\in \C$, have symbol $\rho_2(\xi)\sim\sum_{j\geq 0}\rho_{2,q_2-j}(\xi)$.
\begin{enumerate}
 \item $\rho_1 \sharp \rho_2(\xi)\in S^{q_1+q_2}(\R^n; \cA_\theta)$ with $\rho_1\sharp\rho_2(\xi) \sim\sum (\rho_1\sharp\rho_2)_{q_1+q_2-j}(\xi)$, where 
 \begin{equation} \label{eq:Composition-Sym.sharp-homo-asymptotics}
(\rho_1\sharp\rho_2)_{q_1+q_2-j}(\xi)=\sum_{k+l+|\alpha|=j}\frac{1}{\alpha !}\partial_\xi^\alpha \rho_{1,q_1-k}(\xi)\delta^\alpha \rho_{2,q_2-l}(\xi), \qquad j\geq 0.
\end{equation}

\item The composition $P_1P_2= P_{\rho_1 \sharp \rho_2}$ is contained in $\Psi^{q_1+q_2}(\cA_\theta)$.  
\end{enumerate}
\end{proposition}

\subsection{Adjoints of \psidos\ and action on $\cA_\theta'$}
Given a linear operator $P: \cA_\theta \rightarrow \cA_\theta$, a formal adjoint is any linear operator $P^*: \cA_\theta \rightarrow \cA_\theta$ such that
\begin{equation*} 
\acoup{Pu}{v} = \acoup{u}{P^* v} \qquad \forall u,v \in \cA_\theta,
\end{equation*}
where $\acoup{\cdot}{\cdot}$ is the inner product~(\ref{eq:NCtori.cAtheta-innerproduct}). When it exists a formal adjoint must be unique.

Let $\rho(\xi)\in\stS^m(\Rn;\cA_\theta)$, $m\in\R$, and set
\begin{equation} \label{eq:Adjoints.symbol-star}
\rho^\star(\xi) = \iint e^{it\cdot\eta}\alpha_{-t}[\rho(\xi+\eta)^*]dt\dbar\eta , \qquad \xi\in\Rn .
\end{equation}
For every $\xi\in\Rn$, the map $(t,\eta)\rightarrow\alpha_{-t}[\rho(\xi+\eta)^*]$ is an amplitude, and so the right-hand side makes sense as an oscillating integral (see~\cite{HLP:Part2}). 

\begin{proposition}[see \cite{Ba:CRAS88, HLP:Part2}] \label{prop:Adjoints.rhostar-formal-adjoint}
Let $\rho(\xi)\in\stS^m(\Rn;\cA_\theta)$, $m\in\R$.
\begin{enumerate}
 \item $\rho^\star(\xi) \in \stS^m(\Rn;\cA_\theta)$, and we have $\rho^\star(\xi) \sim \sum \frac{1}{\alpha !}\delta^\alpha\partial_\xi^\alpha [\rho(\xi)^*]$. 
        
\item $P_{\rho^\star}$ is the formal adjoint of $P_\rho$.     
\end{enumerate}
\end{proposition}

We have the following version of this result for classical \psidos. 

\begin{proposition}[see \cite{Ba:CRAS88, HLP:Part2}] \label{prop:Adjoints.adjoint-classical-pdos}
Let $P\in \Psi^q(\cA_\theta)$, $q\in \C$, have symbol $\rho(\xi)\sim\sum_{j\geq 0}\rho_{q-j}(\xi)$.  
\begin{enumerate}
 \item $\rho^\star(\xi) \in S^{\widebar{q}}(\Rn;\cA_\theta)$, and we have  $\rho^\star(\xi)\sim \sum \rho_{\widebar{q}-j}^\star(\xi)$, where 
 \begin{equation} \label{eq:PsiDOs.symbol-adjoint-classical}
\rho_{\widebar{q}-j}^\star(\xi)=\sum_{k+|\alpha|=j}\frac{1}{\alpha !}\delta^\alpha\partial_\xi^\alpha \left[\rho_{q-k}(\xi)^*\right] , \qquad j\geq 0. 
\end{equation}

\item The formal adjoint $P^*=P_{\rho^\star}$ is contained in $\Psi^{\widebar{q}}(\cA_\theta)$.
\end{enumerate}
\end{proposition}

An application of Proposition~\ref{prop:Adjoints.rhostar-formal-adjoint} is the following extension result. 

\begin{proposition}[see \cite{HLP:Part2}] \label{Adjoints.PsiDOs-extension}
 Let $\rho(\xi)\in\stS^m(\Rn;\cA_\theta)$, $m\in\R$. Then $P_\rho$ uniquely extends to a continuous linear operator $P_\rho:\cA_\theta' \rightarrow \cA_\theta'$.
\end{proposition}

\subsection{Sobolev spaces}
 Let $\Delta=\delta_1^2+\cdots+\delta_n^2$ be the Laplacian on $\cA_\theta$. Given any $s\in \R$, the operator $\Lambda^s=(1+\Delta)^{\frac{s}2}$ is a (classical) \psido\ of order $s$ (\emph{cf}.\ Example~\ref{eq:PsiDOs.Lambda-to-the-s-PsiDO}). Therefore, by Proposition~\ref{Adjoints.PsiDOs-extension} it uniquely extends to a continuous linear operator $\Lambda^s:\cA_\theta'\rightarrow \cA_\theta'$.
 
\begin{definition}[see \cite{HLP:Part2, Sp:Padova92, XXY:MAMS18}] 
The Sobolev space $\cH_\theta^{(s)}$, $s\in \R$, consists of all $u\in \cA_\theta'$ such that $\Lambda^su \in \cH_\theta$. It is equipped with the norm, 
\begin{equation*}
 \| u\|_s= \|\Lambda^s u\|_0, \qquad u \in \cH^{(s)}_{\theta}.
\end{equation*}
\end{definition}

\begin{proposition}
$\cH_\theta^{(s)}$ is a Hilbert space.  
\end{proposition}
 
In terms of the Fourier decomposition $u = \sum u_k U^k$ in $\cA_\theta'$ we have 
 \begin{gather}
\nonumber u\in \cH^{(s)}_{\theta} \Longleftrightarrow \sum_{k\in \Z^n}(1+|k|^2)^{s} |u_k|^2<\infty,\\
\label{eq:Sobolev.Sobolev-norm-formula} \| u\|_s^2= \sum_{k\in \Z^n} (1+|k|^2)^{s}|u_k|^2, \qquad u\in \cH_\theta^{(s)}.
\end{gather}

We have the following version of Sobolev's embedding theorem. 

\begin{proposition}[see \cite{HLP:Part2, XXY:MAMS18}] \label{prop:Sobolev.Sobolev-embedding}
If $t>s$, then the inclusion of $\cH^{(t)}_{\theta}$ into $\cH^{(s)}_{\theta}$ is compact. 
\end{proposition}

As the following result shows the Sobolev spaces provide us with a natural ``topological'' scale of Hilbert spaces interpolating between $\cA_\theta$ and $\cA_\theta'$.

\begin{proposition}[see \cite{HLP:Part2, Po:PJM06, Sp:Padova92}] \label{prop:Sobolev.Hs-inclusion-cAtheta}
The following holds. 
\begin{enumerate}
 \item We have 
           \begin{equation*}
                   \cA_\theta= \bigcap_{s\in\R} \cH^{(s)}_\theta \qquad \text{and} \qquad \bigcup_{s\in\R} \cH^{(s)}_\theta=  \cA_\theta'. 
           \end{equation*}
 
 \item The topology of $\cA_\theta$ is generated by the Sobolev norms $\|\cdot \|_s$, $s\in\R$. 
 
 \item The strong topology of $\cA_\theta'$ coincides with the inductive limit of the $\cH_\theta^{(s)}$-topologies.
\end{enumerate}
\end{proposition}

We have the following Sobolev mapping properties of \psidos\ on noncommutative tori.

\begin{proposition}[see \cite{HLP:Part2}] \label{prop:Sob-Mapping.rho-on-Hs}
Let $\rho(\xi)\in \stS^m(\R^n; \cA_\theta)$, $m\in\R$. 
\begin{enumerate}
 \item $P_\rho$ uniquely extends to continuous linear operator
$P_\rho:\cH^{(s+m)}_{\theta}\rightarrow \cH^{(s)}_{\theta}$ for every $s\in \R$. 

\item If $m\leq 0$, then $P_\rho$ uniquely extends to a bounded operator $P_\rho: \cH_\theta \rightarrow \cH_\theta$. This operator is compact when $m<0$. 
\end{enumerate}
\end{proposition}

\subsection{Toroidal symbols and \psidos}\label{subsec:toroidal}  
In this section, we recall the relationship between standard \psidos\ as defined above and the toroidal \psidos\ considered in~\cite{GJP:MAMS17, LNP:TAMS16, RT:Birkhauser10}. 

Let $\cA_\theta^{\Z^n}$ denote the space of sequences with values in $\cA_\theta$ that are indexed by $\Z^n$. In addition, let $(e_1, \ldots, e_n)$ be the canonical basis of $\R^n$. For $i=1,\ldots, n$, the difference operator $\Delta_i: \cA_\theta^{\Z^n} \rightarrow \cA_\theta^{\Z^n}$ is defined by
\begin{equation*}
 \Delta_i u_k= u_{k+e_i}-u_k, \qquad (u_k)_{k\in \Z^n} \in \cA_\theta^{\Z^n}. 
\end{equation*}
The operators $\Delta_1, \ldots, \Delta_n$ pairwise commute. For $\alpha =(\alpha_1, \ldots, \alpha_n)\in \N_0^n$, we set $\Delta^\alpha =\Delta_1^{\alpha_1} \cdots \Delta_n^{\alpha_n}$. We similarly define backward difference operators $\overline{\Delta}=\overline{\Delta}_1^{\alpha_1} \cdots \overline{\Delta}_n^{\alpha_n}$, where $\overline{\Delta}_i:  \cA_\theta^{\Z^n} \rightarrow \cA_\theta^{\Z^n}$ is defined by
\begin{equation*}
\overline{\Delta}_i u_k= u_{k}-u_{k-e_i}, \qquad (u_k)_{k\in \Z^n} \in \cA_\theta^{\Z^n}. 
\end{equation*}

\begin{definition}[\cite{LNP:TAMS16, RT:Birkhauser10}] $\stS^m(\Z^n; \cA_\theta)$, $m\in \R$, consists of sequences $(\rho_k)_{k\in \Z^n} \subset\cA_\theta$ such that, for all multi-orders $\alpha$ and $\beta$, there is $C_{\alpha\beta}>0$ such that 
\begin{equation*}
 \big\| \delta^\alpha \Delta^\beta \rho_k \| \leq C_{\alpha\beta} (1+|k|)^{m-|\beta|} \qquad \forall k \in \Z^n. 
\end{equation*}
\end{definition}

\begin{definition}
 $\cS(\Z^n; \cA_\theta)$ consists of sequences $(\rho_k)_{k\in \Z^n} \subset \cA_\theta$ such that, for all $N\in \N_0$ and $\alpha\in \N_0^n$, there is $C_{N\alpha}>0$ such that 
 \begin{equation*}
 \|\delta^\alpha \rho_k\|\leq C_{N\alpha}(1+|k|)^{-N}\qquad  \forall k \in \Z^n.  
\end{equation*}
\end{definition}
 
\begin{remark} $\cS(\Z^n; \cA_\theta) = \bigcap_{m\in \R} \stS^m(\Z^n; \cA_\theta)$.
\end{remark}
 
 Following~\cite{LNP:TAMS16} (see also~\cite{GJP:MAMS17}) the \psido\ associated with a  toroidal symbol $(\rho_k)_{k\in \Z^n}\in \stS^m(\Z^n; \cA_\theta)$, $m\in \R$, is the linear operator $P:\cA_\theta \rightarrow \cA_\theta$ given by 
\begin{equation*}
 P u = \sum_{k\in \Z^n} u_k\rho_kU^k, \qquad u=\sum u_kU^k\in \cA_\theta. 
\end{equation*}

Any standard \psido\ is a toroidal \psido. More precisely, we have the following result. 
 
\begin{proposition}[see \cite{CT:Baltimore11, HLP:Part1, RT:Birkhauser10}] \label{prop:toroidal.restriction-of-standard-symbol-is-toroidal-symbol}
 Let $\rho(\xi)\in  \stS^m(\R^n;\cA_\theta)$, $m\in \R$. 
\begin{enumerate}
 \item The restriction of  $\rho(\xi)$ to $\Z^n$ is a toroidal symbol $(\rho(k))_{k\in \Z^n}$ in $\stS^m(\Z^n; \cA_\theta)$. 
 
 \item If $(\rho(k))_{k\in \Z^n}\in \cS(\Z^n; \cA_\theta)$, then $\rho(\xi)\in \cS(\R^n;\cA_\theta)$. 
 \item For all $u=\sum_{k \in \Z^n}u_kU^k$ in $\cA_\theta$, we have 
           \begin{equation} \label{eq:toroidal.Prhou-equation}
                P_{\rho}u = \sum_{k\in \Z^n} u_k \rho(k)U^k. 
           \end{equation}
\end{enumerate}
\end{proposition}

\begin{remark}\label{rmk:PsiDOs.rho(k)-P}
 Using~(\ref{eq:toroidal.Prhou-equation}) we get
\begin{equation}
 \rho(k)=P_{\rho}(U^k)(U^k)^{-1}=P_{\rho}(U^k)(U^k)^{*} \qquad \text{for all $k\in \Z^n$}.
 \label{eq:PsiDOs.rho(k)}
\end{equation}
Thus, the restriction to $\Z^n$ of $\rho(\xi)$ is uniquely determined by $P_\rho$. Combining this with the 2nd part shows that the class of $\rho(\xi)$ modulo $\cS(\R^n;\cA_\theta)$ is uniquely determined by $P_\rho$. 
\end{remark}

Conversely, toroidal symbols can be extended to standard symbols of the same order as follows. 

\begin{lemma}[{\cite[Lemma 4.5.1]{RT:Birkhauser10}}] \label{lem:toroidal.phi} There exists a function $\phi(\xi)\in \cS(\Rn)$ such that
\begin{enumerate}
\item[(i)] $\phi(0)=1$ and $\phi(k)=0$ for all $k\in \Z^n\setminus 0$. 

\item[(ii)] For every multi-order $\alpha$, there is $\phi_\alpha(\xi)\in \cS(\Rn)$ such that $\partial_\xi^\alpha \phi(\xi) =\overline{\Delta}^\alpha \phi_\alpha(\xi)$. 

\item[(iii)] $\int \phi(\xi)=1$. 
\end{enumerate}
\end{lemma}

Let  $\phi(\xi)\in \cS(\Rn)$ be a function satisfying the properties (i)--(iii) of Lemma~\ref{lem:toroidal.phi}. Given any toroidal symbol $(\rho_k)_{k\in \Z^n}$ in $\stS^m(\Z^n;\cA_\theta)$, $m\in \R$, define
\begin{equation}
\tilde{\rho}(\xi) = \sum_{k\in \Z^n} \phi(\xi-k) \rho_k, \qquad \xi \in \R^n. 
\label{eq:toroidal.trho} 
\end{equation}
Given any $k\in \Z^n$, we have 
 $ \tilde{\rho}(k)= \sum_{\ell \in \Z^n} \phi(k-\ell) \rho_{\ell}=  \sum_{\ell \in \Z^n} \delta_{\ell,k} \rho_{\ell}=\rho_k$. 
Therefore, this provides us with an extension map from $\stS^m(\Z^n;\cA_\theta)$ to $\stS^m(\R^n;\cA_\theta)$.

\begin{proposition}[see~\cite{HLP:Part1,LNP:TAMS16, RT:Birkhauser10}] \label{prop:toroidal.extension-of-toroidal-symbol-is-standard-symbol}
 Let $(\rho_k)_{k\in \Z^n}\in \stS^m(\Z^n;\cA_\theta)$, $m\in \R$, and denote by $P$ the corresponding toroidal \psido. 
\begin{enumerate}
\item $\tilde{\rho}(\xi)$ is a standard symbol in $\stS^m(\R^n;\cA_\theta)$. 

\item The operators $P$ and $P_\rho$ agree.  
\end{enumerate}
\end{proposition}

Combining together Proposition~\ref{prop:toroidal.restriction-of-standard-symbol-is-toroidal-symbol} and Proposition~\ref{prop:toroidal.extension-of-toroidal-symbol-is-standard-symbol} we arrive at the following result. 

\begin{proposition}[see \cite{HLP:Part1}]  \label{prop:toroidal.standard-and-toroidal-psidos-agree}
For every $m\in \R$, the classes of toroidal and standard \psidos\ of order $m$ agree. 
\end{proposition}

By construction the extension map $(\rho_k)_{k\in \Z^n} \rightarrow \tilde{\rho}(\xi)$ given by~(\ref{eq:toroidal.trho}) is a right-inverse of the restriction map $\rho(\xi) \rightarrow (\rho(k))_{k\in \Z^n}$. It is also interesting to look at the reverse composition.

\begin{proposition}[see \cite{HLP:Part1}] \label{prop:PsiDOs.normalization-symbol} 
 Let $\rho(\xi)\in \stS^m(\R^n;\cA_\theta)$, $m\in \R$, and denote by $\tilde{\rho}(\xi)$ the extension~(\ref{eq:toroidal.trho}) of $(\rho(k))_{k\in \Z^n}$. 
\begin{enumerate}
 \item $\tilde{\rho}(\xi)$ is a standard symbol in $\stS^m(\R^n;\cA_\theta)$ which agrees with $\rho(\xi)$ on $\Z^n$. 
 
\item $P_{\tilde{\rho}}=P_\rho$ and $\tilde{\rho}(\xi)-\rho(\xi)\in \cS(\R^n)$. 
\end{enumerate}
\end{proposition}

\subsection{Smoothing operators}
A linear operator $R:\cA_\theta \rightarrow \cA_\theta'$ is called \emph{smoothing} when it extends to a continuous linear operator $R:\cA_\theta'\rightarrow \cA_\theta$.
Note that, as  $\cA_\theta$ is dense in $\cA_\theta'$,  the extension of a smoothing operator to $\cA_\theta'$ is unique. 

\begin{proposition}[\cite{HLP:Part1, HLP:Part2}] \label{prop:PsiDOs.smoothing-operator-characterization}
Let $R:\cA_\theta \rightarrow \cA_\theta$ be a linear map. The following are equivalent. 
\begin{enumerate}
\item[(i)]  $R$ is a smoothing operator. 

\item[(ii)] There is $\rho(\xi)\in \cS(\R^n; \cA_\theta)$ such that $R=P_\rho$, i.e., $P\in \Psi^{-\infty}(\cA_\theta)$. 

\item[(iii)] $R$ extends to a continuous linear map $R:\cH_\theta^{(s)}\rightarrow \cH_\theta^{(t)}$ for all $s,t\in \R$. 
\end{enumerate}
\end{proposition}

\begin{remark}
If $R$ is in all the spaces $\Psi^q(\cA)$, $q\in \C$, then by Proposition~\ref{prop:Sob-Mapping.rho-on-Hs} it extends to a continuous linear map $R:\cH_\theta^{(s)}\rightarrow \cH_\theta^{(t)}$ for all $s,t\in \R$, and so by using (ii) we see that $R\in \Psi^{-\infty}(\cA_\theta)$. Combining this with Remark~\ref{rmk:Symbols.symbols-intersection} shows that
$\Psi^{-\infty}(\cA_\theta) = \bigcap_{q\in \C} \Psi^q(\cA_\theta)$. 
\end{remark}

\subsection{Ellipticity and parametrices}
In this section, we recall the main facts regarding elliptic operators. 

\begin{definition}
An operator $P\in\Psi^q(\cA_\theta)$, $q\in \C$, is \emph{elliptic} when its principal symbol $\rho_{q}(\xi)$ is invertible for all $\xi \in \R^n\setminus 0$. 
\end{definition} 

\begin{remark}
 The ellipticity condition implies that $\rho_{q}(\xi)^{-1} \in S_{-q}(\R^n;\cA_\theta)$ (see~\cite{Ba:CRAS88, HLP:Part2}). 
\end{remark}

\begin{example}
 Suppose that the principal symbol $\rho_{q}(\xi)$ of $P$ is such that there is $c>0$ such that
 \begin{equation} \label{eq:Elliptic.positivity-criterion}
 \acoup{\rho_q(\xi)\eta}{\eta}\geq c|\xi|^q \|\eta\|_0^2 \qquad \text{for all $\eta \in \cH_\theta$ and $\xi \in \R^n\setminus \{0\}$}.
\end{equation}
 Then $P$ is elliptic (see~\cite{HLP:Part2}). This condition is satisfied in the following examples:
\begin{itemize}
 \item[-] The (flat) Laplacian $\Delta=\delta_1^2+\cdots+\delta_n^2$. 
 
 \item[-] The conformal deformations $k\Delta k$, $k\in \cA_\theta$, $k>0$. These operators were introduced by Connes-Tretkoff~\cite{CT:Baltimore11}. They were considered by various other authors as well. 
 
\item[-] The Laplace-Beltrami operators of~\cite{HP:Laplacian}. 
\end{itemize}
\end{example}

\begin{theorem}[see \cite{Ba:CRAS88, Co:CRAS80, HLP:Part2}] \label{prop:Elliptic.existence-parametrice}
Let $P\in\Psi^q(\cA_\theta)$, $q\in\C$, be elliptic with symbol  $\rho(\xi)\sim\sum_{j\geq 0} \rho_{q-j}(\xi)$. 
\begin{enumerate}
 \item $P$ admits a parametrix $Q\in \Psi^{-q}(\cA_\theta)$, i.e., 
          \begin{equation*}
                PQ=QP=1 \quad \bmod \Psi^{-\infty}(\cA_\theta). 
          \end{equation*}

\item Any parametrix $Q\in \Psi^{-q}(\cA_\theta)$ has symbol $\sigma(\xi)\sim\sum_{j\geq 0} \sigma_{-q-j}(\xi)$, where 
\begin{gather} 
 \sigma_{-q}(\xi)=\rho_q(\xi)^{-1},
 \label{eq:Elliptic.para-homo-principal}\\
\sigma_{-q-j}(\xi)=-\!\!\!\sum_{\substack{k+l+|\alpha|=j \\ l<j}}\!\! \frac{1}{\alpha !}\rho_q(\xi)^{-1}\partial_\xi^\alpha\rho_{q-k}(\xi)\delta^\alpha\sigma_{-q-l}(\xi), \qquad j\geq 1 .
\label{eq:Elliptic.para-homo-asymp} 
\end{gather}
\end{enumerate}
\end{theorem}

An important application of Theorem~\ref{prop:Elliptic.existence-parametrice} is the following version of the elliptic regularity theorem. 

\begin{proposition}[\cite{HLP:Part2}] \label{prop:Elliptic.regularity}
Let $P\in\Psi^q(\cA_\theta)$, $q\in\C$, be elliptic, and set $m=\Re q$. 
\begin{enumerate}
\item For any $s\in \R$ and $u\in{\cA_\theta}'$, we have 
\begin{equation*}
Pu\in\cH^{(s)}_{\theta}\Longleftrightarrow u\in\cH^{(s+m)}_{\theta} .
\end{equation*}

\item The operator $P$ is hypoelliptic, i.e.,  for any $u\in \cA_\theta'$, we have
\begin{equation*}
 Pu \in \cA_\theta \Longleftrightarrow u \in \cA_\theta. 
\end{equation*}

\item If $m>0$, then the operator $P-\lambda$ is hypoelliptic in the above sense for every $\lambda \in \C$.  
\end{enumerate}
\end{proposition}

\begin{corollary}[\cite{HLP:Part2}] \label{cor:Elliptic.P-lambda-hypoell}
 Let $P\in\Psi^q(\cA_\theta)$, $q\in \C$, be elliptic. Then 
   \begin{equation*}
 \ker P :=\left\{u\in \cA_\theta'; \ Pu=0\right\}\subset \cA_\theta. 
\end{equation*}
If $\Re q>0$, then, for every $\lambda \in \C$, we also have 
\begin{equation*}
 \ker (P-\lambda) :=\left\{u\in \cA_\theta'; \ (P-\lambda )u=0\right\}\subset \cA_\theta. 
\end{equation*}
 \end{corollary}

\subsection{Spectral theory of elliptic \psidos} \label{subsec:spectra-of-psidos}
Let $P\in \Psi^q(\cA_\theta)$ be elliptic with $m:=\Re q>0$. We shall regard $P$ as an unbounded operator of $\cH_\theta$ with domain $\cH^{(m)}_\theta$. We also let $P^*\in \Psi^{\bar{q}}(\cA_\theta)$ be its formal adjoint. Recall that $P^*\in \Psi^{\widebar{q}}(\cA_\theta)$ by Proposition~\ref{prop:Adjoints.rhostar-formal-adjoint}. 

\begin{proposition}[see \cite{Ba:CRAS88, Co:CRAS80, HLP:Part2}]\label{prop:Spectrum.closedness}
The following holds.
\begin{enumerate}
 \item The operator $P$ with domain $\cH^{(m)}_\theta$ is closed, Fredholm and has closed range.
 
 \item The adjoint of $P$ is the formal adjoint $P^*$ with domain $\cH^{(m)}_\theta$. 
 
 \item   If $P$ is formally selfadjoint (resp., $P$ commutes with its formal adjoint), then $P$ is selfdjoint (resp., normal).  
\end{enumerate}
 \end{proposition}

\begin{definition}
 The \emph{resolvent set} of $P$ consists of all $\lambda\in \C$ such that $P-\lambda:\cH_\theta^{(m)}\rightarrow \cH_\theta$ is a bijection with bounded inverse. The \emph{spectrum} of $P$, denoted by $\Sp(P)$, is the complement of its resolvent set. 
\end{definition}

\begin{remark}\label{rmk:PsiDOs.resolvent-boundedness}
If $\lambda \in \C \setminus \Sp(P)$, then $(P-\lambda)^{-1}\in \cL(\cH_\theta, \cH^{(m)}_\theta)$. 
\end{remark}

\begin{remark}\label{rmk:PsiDOs.spectrum-adjoint} 
 It can be shown that $\Sp(P^*)=\{\overline{\lambda}; \ \lambda \in \Sp(P)\}$ (see~\cite[Corollary~12.8]{HLP:Part2}). 
\end{remark}

\begin{proposition}[see \cite{HLP:Part2}] \label{prop:Spectral.spectrum-P}
 There are only two possibilities for the spectrum of $P$. Either $\Sp(P)$ is all $\C$, or this is an unbounded discrete set consisting of isolated eigenvalues with finite multiplicity. 
 \end{proposition}
 
 In the special case of normal operators we even obtain the following result. 
 
 \begin{proposition}[see \cite{HLP:Part2}] \label{prop:Spectral.ortho-splitting}
Suppose that $P$ is normal and $\Sp (P)\neq \C$ (e.g., $P$ is selfadjoint). 
 \begin{enumerate}
 \item We have an orthogonal decomposition, 
 \begin{equation} \label{eq:Spectral.ortho-splitting}
 \cH_\theta = \bigoplus_{\lambda \in \Sp(P)} \ker (P-\lambda). 
\end{equation}

\item There is an orthonormal basis $(e_\ell)_{\ell \geq 0}$ of $\cH_\theta$ such that $e_\ell\in\cA_\theta $ and 
$P e_\ell =\lambda_\ell e_\ell$ with  $|\lambda_\ell|\rightarrow \infty$ as $\ell \rightarrow \infty$.
\end{enumerate}
\end{proposition}

\subsection{Schatten-class properties of \psidos} \label{subsec:Schatten}
We denote by $\cK$ the closed two-sided ideal of compact operators on $\cH_\theta$. For $p\geq 1$ we let $\cL^p$ and $\cL^{(p,\infty)}$ be the corresponding Schatten and weak Schatten classes equipped with their usual norms (see, e.g., 
\cite{GK:AMS69, Si:AMS05}). In particular,  we have 
\begin{gather*}
 \cL^{p}=\big\{T\in \cK; \ \Trace |T|^p<\infty\big\}, \qquad p\geq 1\\
 \cL^{(p,\infty)} = \big\{T\in \cK;\ \sum_{k<N}\mu_k(T)= \op{O}\big(N^{1-\frac1{p}}\big)\big\},\qquad p>1,\\
 \cL^{(1,\infty)} = \big\{T\in \cK;\ \sum_{k<N}\mu_k(T)= \op{O}(\log N)\big\}.
 \end{gather*}
Here $\mu_k(T)$ is the $(k+1)$-th eigenvalue of $|T|=\sqrt{T^*T}$ counted with multiplicity. Moreover,  $\cL^p$ and $\cL^{(p,\infty)}$ are Banach ideals with respect to their natural norms, 
\begin{gather*}
 \|T\|_{\cL^p}=\big( \Tra |T|^p\big)^{\frac1{p}}, \qquad T\in \cL^p, \ p\geq 1,\\
   \|T\|_{\cL^{(1,\infty)}}= \sup_{N\geq 1}\biggl\{N^{-1+\frac1{p}} \sum_{k<N} \mu_k(T)\biggr\}, \qquad T\in \cL^{(p,\infty)},\  p>1,\\ 
 \|T\|_{\cL^{(1,\infty)}}= \sup_{N\geq 2}\biggr\{\frac{1}{\log N} \sum_{k<N} \mu_k(T)\biggl\}, \qquad T\in \cL^{(1,\infty)}. 
\end{gather*}


\begin{proposition}[see~\cite{HLP:Part2}] 
 Let $\rho(\xi)\in \stS^m(\R^n;\cA_\theta)$, $m<0$. 
\begin{enumerate}
 \item If $m\in[-n,0)$ and $p=n|m|^{-1}$, then $P_\rho\in \cL^{(p,\infty)}$, and hence $P_\rho\in \cL^{q}$ for all $q>p$. 
 \item If $m<-n$, then $P_\rho\in \cL^1$. 
\end{enumerate}
\end{proposition}

%

 \section{Parametric Symbols} \label{sec:Parametric-symbols}
In this section, we introduce our classes of symbols with parameter. These classes are similar to the classes introduced in~\cite{Po:PhD, Po:CRAS01}. 

Throughout the rest of this paper, we denote by $\C^*$ the complex plane with the origin deleted, i.e., $\C^*=\C\setminus \{0\}$. 

We refer to~\cite[Appendix~C]{HLP:Part1}, and the references therein, for background on differentiable maps with values in a given locally convex space $\sE$. In particular, if $U\subset \R^d$ is open, then a map $f:U\rightarrow \sE$ is $C^N$, $N\geq 0$, when all the partial derivatives of order~$\leq N$ exist and are continuous on $U$. The map is $C^\infty$ or smooth when it is $C^N$ for all $N\geq 0$. Moreover, the space $C^\infty(U;\sE)$ of smooth maps from $U$ to $\sE$ is a locally convex space with respect to the topology generated by the semi-norms, 
\begin{equation*}
 u \longrightarrow \max_{|\alpha| \leq N} \sup_{x\in K} p\big[ \partial_x^\alpha u(x)\big],
\end{equation*}
where $N$ ranges over non-negative integers, $K$ ranges over compact subsets of $U$ and $p$ ranges over continuous semi-norms on $\sE$. 

\subsection{Pseudo-cones and $\Hol^d(\Lambda)$-families} 
In what follows, by a cone $\Theta \subset \C^*$ we shall always mean a cone with vertex at the origin, so that $ \lambda \in \Theta \Rightarrow t\lambda \in \Theta$  for all $t>0$. 

 \begin{definition}
A  connected set $\Lambda\subset \C$ is called a \emph{pseudo-cone} when there is a cone $\Theta \subset \C^*$ and a disk $D$ about the origin such that 
$\Lambda\setminus D=\Theta\setminus D$. The cone $\Theta$ is called the \emph{conical par}t of $\Lambda$.  
 \end{definition}

\begin{remark}
 In the above definition the conical part $\Theta$ is unique. 
\end{remark}

\begin{remark}
If $\Lambda$ is a pseudo-cone, then its interior $\op{Int}(\Lambda)$ and its closure $\overline{\Lambda}$ are pseudo-cones as well. 
\end{remark}

\begin{example}
 Any angular sector $\Theta=\{ \phi <\arg \lambda <\phi'\}$ is a cone, and hence is a pseudo-cone. Chopping off a disk from $\Theta$ or glueing to it a disk or an annulus  provides us with pseudo-cones with conical component $\Theta$. 
 \end{example}
 
\begin{figure}[h]
\begin{minipage}{0.25\linewidth}
\centering{\def\svgwidth{\columnwidth}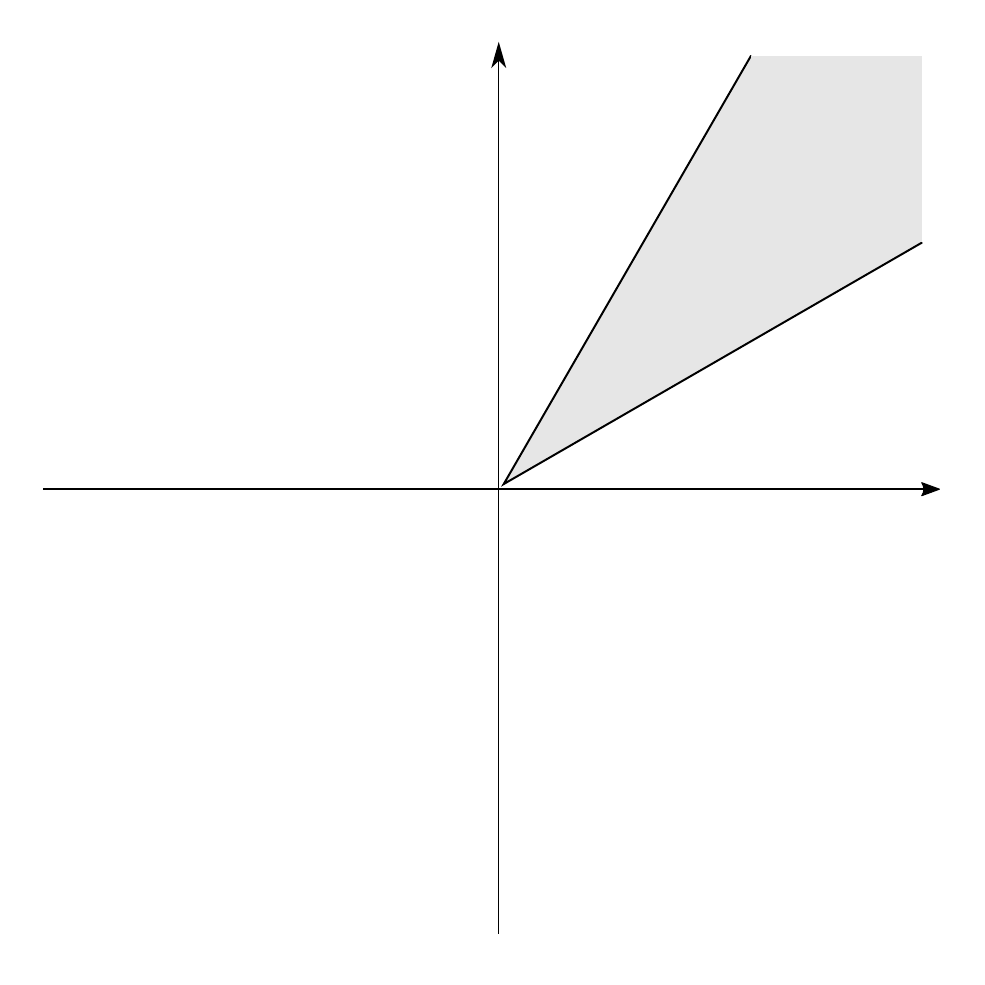}
\end{minipage}
\begin{minipage}{0.25\linewidth}
\centering{\def\svgwidth{\columnwidth}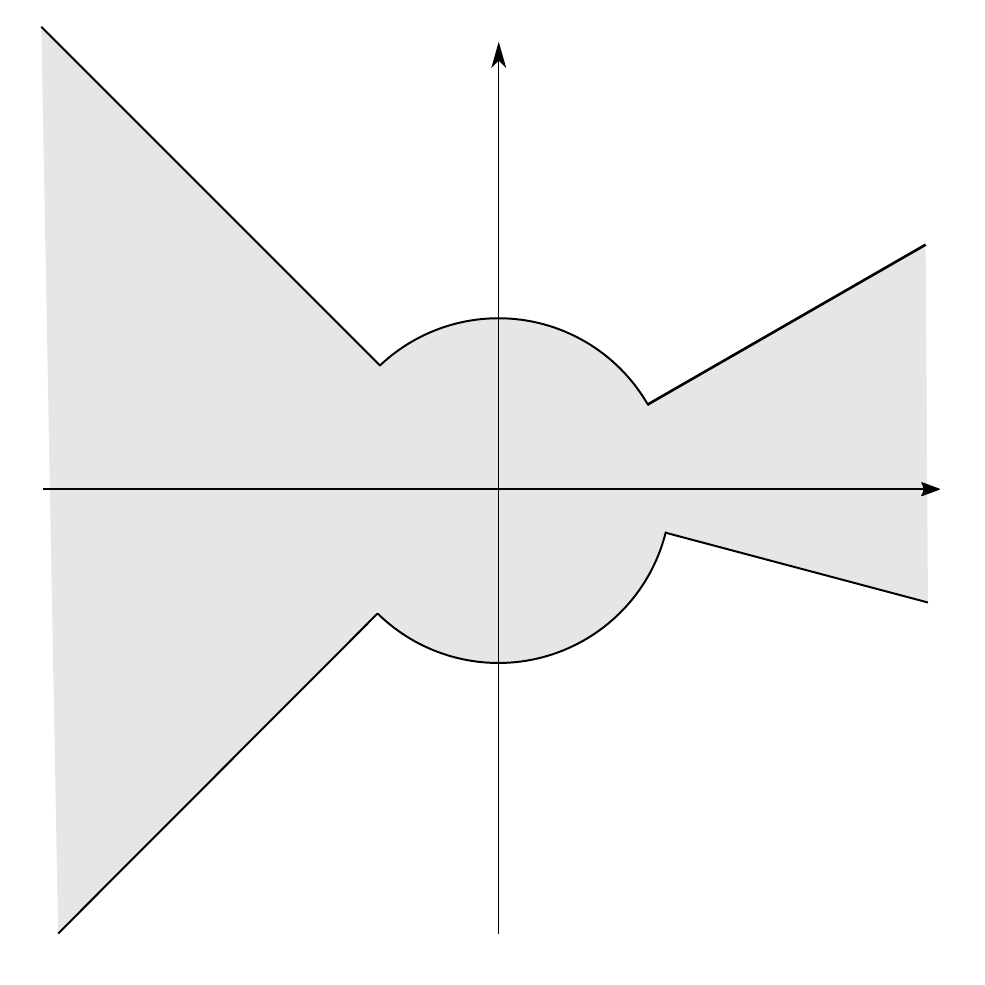}
\end{minipage}
\begin{minipage}{0.25\linewidth}
\centering{\def\svgwidth{\columnwidth}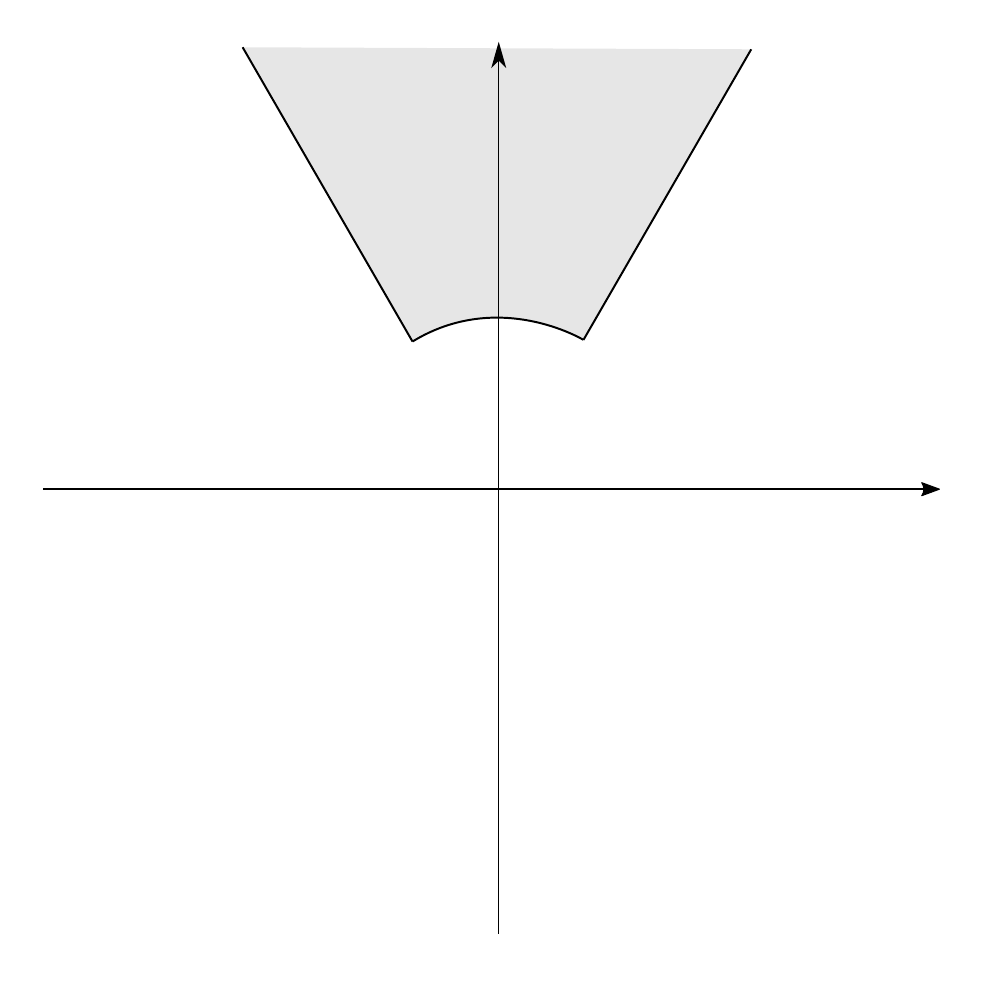}
\end{minipage}
\begin{minipage}{0.25\linewidth}
\centering{\def\svgwidth{\columnwidth}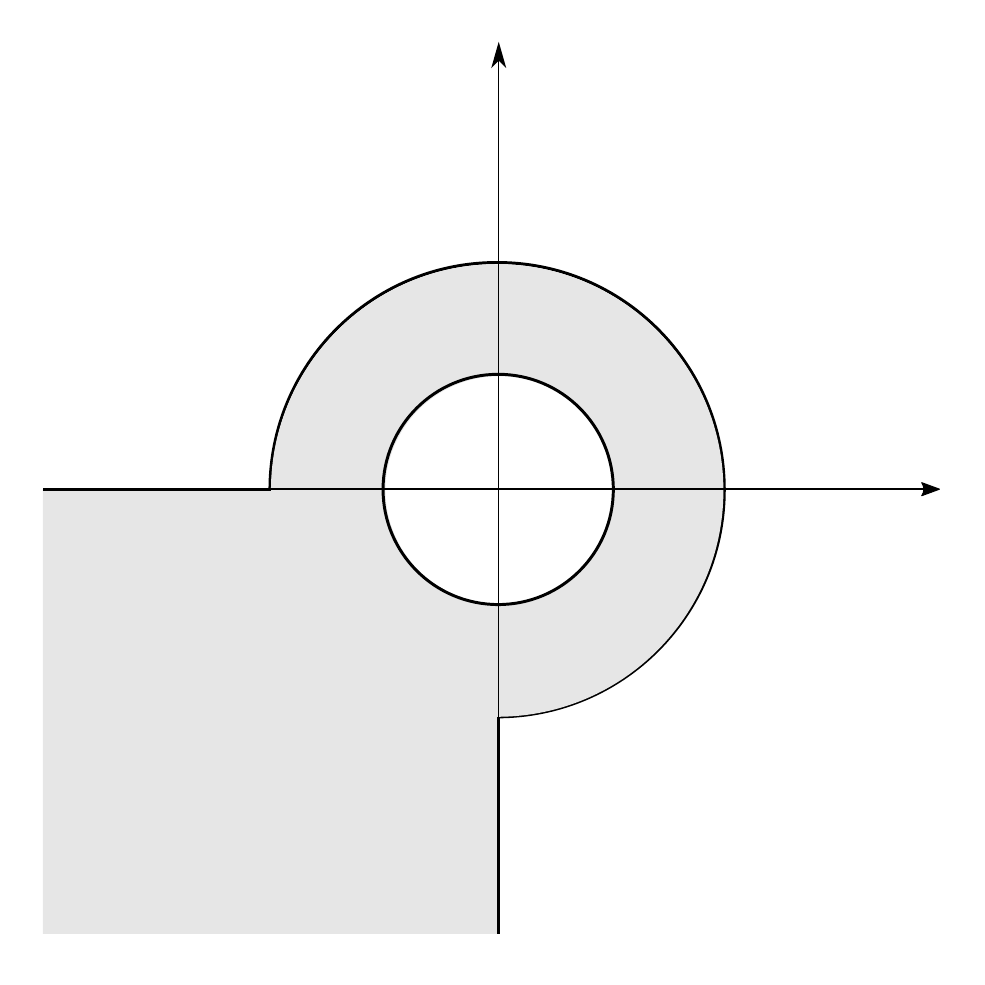}
\end{minipage}
\begin{minipage}{0.25\linewidth}
\centering{\def\svgwidth{\columnwidth}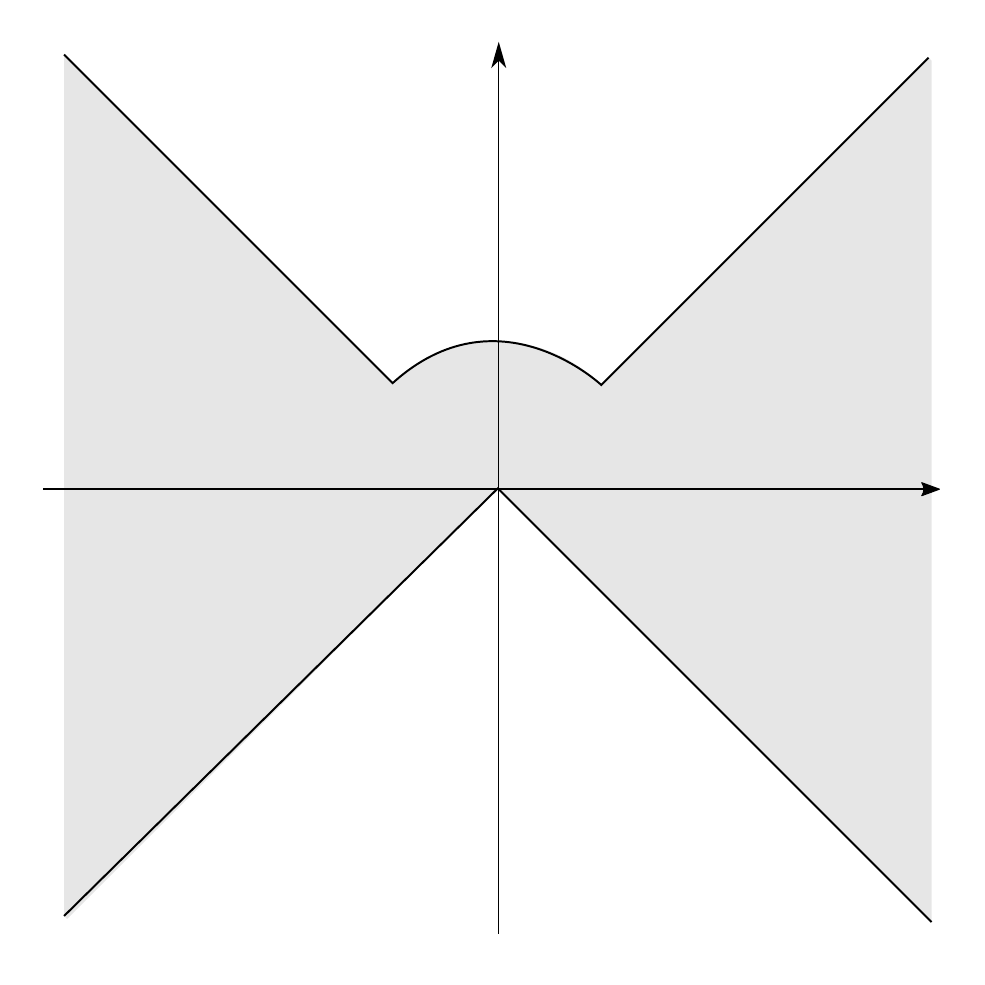}
\end{minipage}
\begin{minipage}{0.25\linewidth}
\centering{\def\svgwidth{\columnwidth}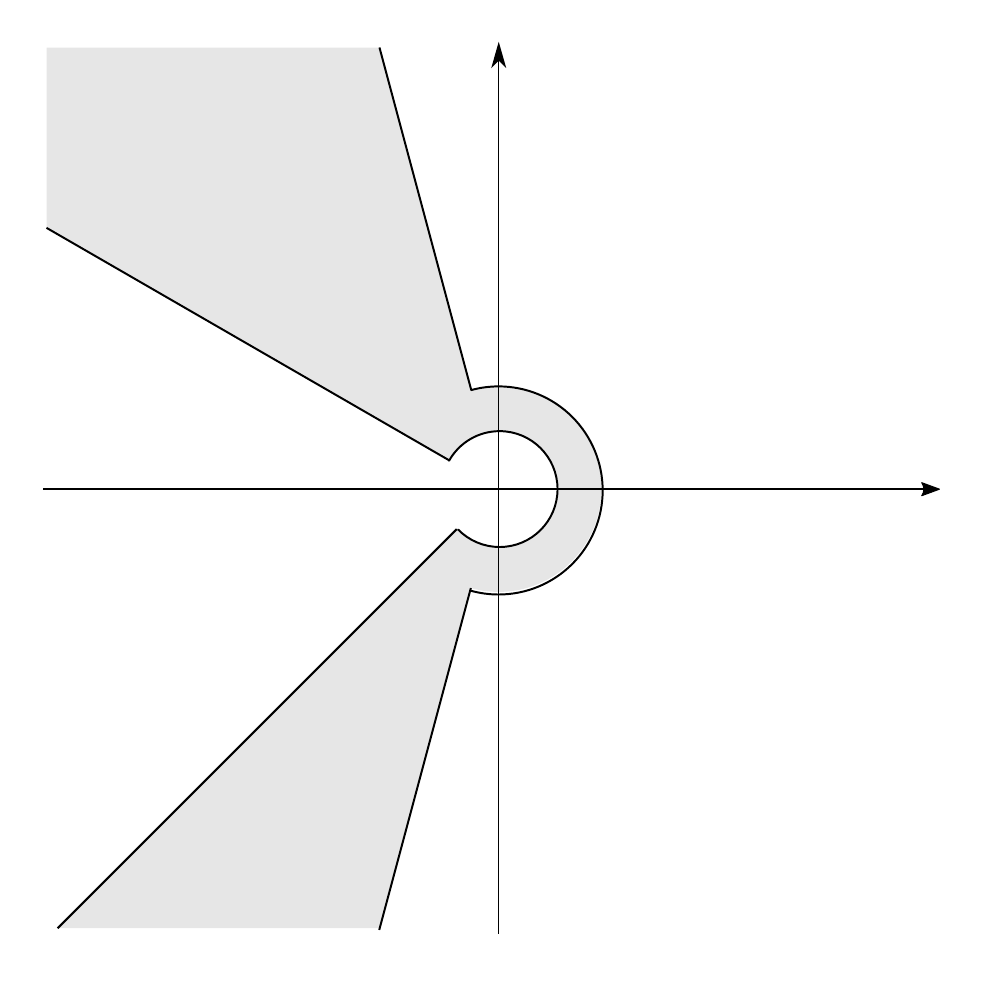}
\end{minipage}
\caption{Examples of Pseudo-Cones}
\end{figure}

\begin{definition}
 Given pseudo-cones  $\Lambda$ and $\Lambda'$ we shall write $\Lambda'\subsubset\Lambda$ to mean that $\overline{\Lambda'}\subset \op{Int}(\Lambda)$.
\end{definition}

\begin{remark}
If $\Theta$ and $\Theta'$ are the respective conical parts of $\Lambda$ and $\Lambda'$, then the relation $\Lambda'\subsubset\Lambda$ implies that 
$\Theta'\cap\bS^1$ is a relatively compact subset of $\Theta \cap \bS^{1}$ (where $\bS^{1}\subset \C$ is the unit circle). 
\end{remark}

 \begin{remark}
 If $\Lambda \subset \C^*$, then the relation $\Lambda'\subsubset\Lambda$ implies there is an open disk $D$ about the origin such that $\Lambda'\subset \C\setminus D$. 
 Thus, there is $r>0$ such that $|\lambda|\geq r$ for all $\lambda \in \Lambda'$.  
\end{remark}

\begin{remark} \label{rmk:Parameter.pseudo-cone-exhaustion}
Any open pseudo-cone $\Lambda$ admits a pseudo-cone exhaustion $\Lambda = \bigcup_{j\geq 0} \Lambda_j$, where the $\Lambda_j$ are closed pseudo-cones such that  $\Lambda_j \subsubset \Lambda_{j+1}$. 
\end{remark}
 
Throughout this section we let $\Lambda$ be an \emph{open} pseudo-cone. 

\begin{definition}
 $\Hol^d(\Lambda)$, $d\in \R$, consists of holomorphic functions $f:\Lambda \rightarrow \C$ such that, for every pseudo-cone $\Lambda'\subsubset \Lambda$, there is $C_{\Lambda'}>0$ such that 
\begin{equation*}
 |f(\lambda)| \leq C_{\Lambda'} (1+|\lambda|)^d \qquad \forall \lambda \in \Lambda'. 
\end{equation*}
\end{definition}

We equip the vector space $\Hol^d(\Lambda)$, $d\in \R$, with the locally convex topology generated by the semi-norms, 
\begin{equation*}
 f \longrightarrow \sup_{\lambda \in \Lambda'} (1+|\lambda|)^{-d} |f(\lambda)|, 
\end{equation*}
where $\Lambda'$ ranges over pseudo-cones $\Lambda'\subsubset \Lambda$. This turns $\Hol^d(\Lambda)$ into a Fr\'echet space. The space $\Hol^{-\infty}(\Lambda)$ is also a Fr\'echet space with respect to the topology generated by these semi-norms, where $\Lambda'$ ranges over all pseudo-cones $\subsubset \Lambda$ and $d$ ranges over $\R$. 

In what follows, given an open $\Omega \subset \C$ and a locally convex space $\sE$, we shall say that a map $f:\Omega\rightarrow \sE$ is \emph{holomorphic} at a given point $\lambda_0\in \Omega$ when 
\begin{equation*}
 \partial_{\lambda}f(\lambda_0):= \lim_{\lambda \rightarrow \lambda_0} \frac{f(\lambda)-f(\lambda_0)}{\lambda -\lambda_0} \ \text{exists in} \ \sE. 
\end{equation*}
We say that $f$ is \emph{holomorphic} on $\Omega$ when it is holomorphic at every point of $\Omega$.

\begin{remark}
 Let $f:\Omega \rightarrow \sE$ be a holomorphic map. Then, for all $\lambda\in \Omega$, we have the Cauchy formula, 
\begin{equation}
 f(\lambda)= \frac{1}{2i\pi} \int_{\Gamma}f(\zeta)(\zeta-\lambda)^{-1}d\zeta,
 \label{eq:symbols.Cauchy-Formula}
\end{equation}
where $\Gamma$ is any direct-oriented Jordan curve contained in $\Omega$  whose interior contains $\lambda$. Indeed, for all $\varphi \in \cE'$, the composition $\varphi\circ f$ is a holomoprhic function on $\Omega$, and so we have 
\begin{equation*}
 \varphi\big[f(\lambda)\big] = \frac{1}{2i\pi} \int_{\Gamma}(\varphi\circ f)(\zeta)(\zeta-\lambda)^{-1}d\zeta = \frac{1}{2i\pi} \varphi\biggl(   \int_{\Gamma}f(\zeta)(\zeta-\lambda)^{-1}d\zeta\biggr). 
\end{equation*}
As $\sE'$ separates the points of $\sE$ this gives~(\ref{eq:symbols.Cauchy-Formula}).  If we denote by $D$ the interior of $\Gamma$, then the integrand in the right-hand side of~(\ref{eq:symbols.Cauchy-Formula}) lies in $C^\infty(D\times \Gamma;\sE)$, and so the integral actually converges in $C^\infty(D; \sE)$ (see, e.g.,~\cite[Appendix C]{HLP:Part1}). It then follows that $f$ is a smooth map from $\Omega$ to $\sE$. Moreover, by differentiating under the integral sign (see, e.g., \cite[Appendix C]{HLP:Part1}) we see that $f(\lambda)$ satisfies the Cauchy-Riemann equation $\partial_{\overline{\lambda}}f=0$. Conversely, any differentiable map satisfying this equation is a holomorphic map. 
\end{remark}

\begin{definition}\label{def:symbols.Hold-sE}
Suppose that $\sE$ is a locally convex space. Then $\Hol^d(\Lambda;\sE)$, $d\in\R$, consists of holomorphic maps $x: \Lambda \rightarrow \sE$ such that, for every continuous semi-norm $p$ on $\sE$ and every pseudo-cone $\Lambda'\subsubset\Lambda$, there is $C_{p\Lambda'}>0$ such that
\begin{equation} \label{eq:Parameter.vector-with-parameter-estimate}
p\left[x(\lambda)\right]\leq C_{p\Lambda'}(1+|\lambda|)^d \qquad \forall\lambda\in\Lambda' .
\end{equation}
\end{definition}

\begin{remark}
 $\Hol^d(\Lambda; \sE)$, $d\in \R$, is a locally convex space with respect to the topology generated by the semi-norms, 
 \begin{equation} \label{eq:Parameter.vector-with-parameter-semi-norm}
x(\lambda) \longrightarrow \sup_{\lambda\in\Lambda'}(1+|\lambda|)^{-d} p\left[x(\lambda)\right], 
\end{equation}
where $p$ ranges over continuous semi-norms on $\sE$ and $\Lambda'$ ranges over pseudo-cones $\Lambda'\subsubset \Lambda$. In addition, if $\sE$ is a Fr\'echet space, then $\Hol^d(\Lambda; \sE)$ is a Fr\'echet space as well. 
\end{remark}

\begin{proposition} \label{prop:Parameter.linear-map}
Suppose that $\sE$ and $\sF$ are locally convex spaces, and let $d\in \R$. Any continuous linear map $T:\sE \rightarrow \sF$ gives rise to a continuous linear map $T:\Hol^d(\Lambda; \sE) \rightarrow \Hol^d(\Lambda; \sF)$.
 \end{proposition}
\begin{proof}
Let $x(\lambda)\in\Hol^d(\Lambda;\sE)$, $d\in \R$. As $T$ is a continuous linear map, the composition $T[x(\lambda)]$ is a holomorphic map from $\Lambda$ to $\sF$.
Let $q$ be a continuous semi-norm on $\sF$. Then $q\circ T$ is a continuous semi-norm on $\sE$, and so, for every pseudo-cone $\Lambda'\subsubset \Lambda$, we have 
\begin{equation*}
 \sup_{\lambda \in \Lambda'} (1+|\lambda|)^{-d}q\left(T\big[x(\lambda)\big]\right)= \sup_{\lambda \in \Lambda'} (1+|\lambda|)^{-d}(q\circ T)\big[x(\lambda)\big]<\infty. 
\end{equation*}
This shows that $T[x(\lambda)]$ is contained in $ \Hol^d(\Lambda; \sF)$ and depends continuously on $x(\lambda)$. The proof is complete. 
\end{proof}

For bilinear bilinear maps we have the following statement. 

\begin{proposition} \label{prop:Parameter.bilinear-map}
 Suppose that $\sE_i$, $i=1,2$, and $\sF$ are locally convex spaces, and let $d_1, d_2\in \R$. Any (jointly) continuous bilinear map $\Phi:\sE_1\times \sE_2 \rightarrow \sF$ gives rise to  a continuous bilinear map $ \Phi: \Hol^{d_1}(\Lambda;\sE_1)\times \Hol^{d_2}(\Lambda;\sE_2) \rightarrow \Hol^{d_1+d_2}(\Lambda; \sF)$. 
\end{proposition}
\begin{proof}
Let $x_j(\lambda)\in\Hol^{d_j}(\Lambda;\sE_j)$, $j=1,2$. As $x_1(\lambda)$ (resp., $x_2(\lambda)$) is a holomorphic family with values in $\sE_1$ (resp., $\sE_2$) and the map $\Phi:\sE_1\times \sE_2\rightarrow \sF$ is continuous and bilinear, we see that $\Phi[x_1(\lambda),x_2(\lambda)]$, $\lambda \in \Lambda$, is a holomorphic family in $\sF$. Moreover, the continuity of $\Phi$ means that, given any continuous semi-norm $q$ on $\sF$, there are a continuous semi-norm $p_1$ on $\sE_1$ and 
a continuous semi-norm $p_2$ on $\sE_2$ such that 
\begin{equation*}
q \big( \Phi[x_1,x_2] \big) \leq p_1(x_1) p_2(x_2)  \qquad \forall x_i \in \sE_i.
\end{equation*}
Thus, for every pseudo-cone $\Lambda'\subsubset \Lambda$, we have 
\begin{align*}
 \sup_{\lambda \in \Lambda'} (1+|\lambda|)^{-(d_1+d_2)} q \big( \Phi[x_1(\lambda),x_2(\lambda) ] \big)  & \leq 
   \sup_{\lambda \in \Lambda'} (1+|\lambda|)^{-d_1} p_1[x_1(\lambda)]  \cdot  \sup_{\lambda \in \Lambda'} (1+|\lambda|)^{-d_2} p_2[x_2(\lambda)]\\
    &<\infty.  
\end{align*}
This shows that  $\Phi[x_1(\lambda),x_2(\lambda) ]$ is an $\Hol^{d_1+d_2}(\Lambda)$-family and depends continuously on $x_1(\lambda)$ and $x_2(\lambda)$. The proof is complete. 
\end{proof}

In addition, it is also convenient to introduce the following class of maps. 

\begin{definition}
 Suppose that $\Omega$ is an open set of $\R^N\times \C$ for some $N\geq 1$ and $\sE$ is a locally convex space. Then $C^{\infty,\omega}(\Omega; \sE)$ consists of maps $\Omega\ni(\eta,\lambda) \rightarrow f(\eta;\lambda)\in \sE$ that are $C^\infty$ with respect to $\eta$ and holomorphic with respect to $\lambda$.
\end{definition}

\begin{remark}\label{rmk:Symbols.identification-Cinftyw}
 By using the Cauchy formula~(\ref{eq:symbols.Cauchy-Formula}) we can identify  $C^{\infty,\omega}(\Omega; \sE)$ with the (closed) subspace of $C^\infty(\Omega;\sE)$ of solutions of the Cauchy-Riemann equation $\partial_{\widebar{\lambda}} f=0$. In particular, when $\Omega=V\times \Lambda$, where $V\subset \R^N$ is an open set, then under the natural identification $C^\infty(V\times \Lambda; \sE)\simeq C^\infty(\Lambda; C^\infty(V;\sE))$ we get a one-to-one correspondence, 
\begin{equation*}
 C^{\infty,\omega}(\Omega; \sE)\simeq \Hol\big(\Lambda; C^\infty(V;\sE)\big). 
\end{equation*}
\end{remark}

\subsection{Standard parametric symbols} 
\begin{definition}
 $\stS^{m,d}(\R^n\times \Lambda; \cA_\theta)$, $m,d\in \R$, consists of maps $\rho(\xi;\lambda)\in C^{\infty,\omega} (\R^n\times \Lambda;\cA_\theta)$ such that, for all pseudo-cones $\Lambda'\subsubset \Lambda$ and multi-orders $\alpha$ and $\beta$, there is $C_{\Lambda'\alpha\beta}>0$ such that
\begin{equation} \label{eq:Parameter.standard-symbol-with-parameter-estimate}
 \big\|\delta^\alpha \partial_\xi^\beta \rho(\xi;\lambda)\big\| \leq  C_{\Lambda'\alpha\beta}(1+|\lambda|)^d(1+|\xi|)^{m-|\beta|} \qquad \forall (\xi, \lambda)\in \R^n\times \Lambda'. 
\end{equation}
\end{definition}

\begin{example}
 Any symbol $\rho(\xi)\in \stS^m(\R^n; \cA_\theta)$ can be regarded as an element of $\stS^{m,0}(\R^n\times \Lambda;\cA_\theta)$ that does not depend on $\lambda$. 
\end{example}

\begin{remark}
 $\stS^{m,d}(\Rn\times\Lambda;\cA_\theta)$ is a locally convex space with respect to the locally convex topology generated  by the semi-norms,
\begin{equation} \label{eq:Parameter.standard-symbol-with-parameter-semi-norm}
\rho(\xi;\lambda)\longrightarrow\max_{|\alpha|+|\beta|\leq N}\sup_{(\xi,\lambda)\in\Rn\times\Lambda'}(1+|\lambda|)^{-d}(1+|\xi|)^{|\beta|-m}\|\delta^\alpha\partial_\xi^\beta\rho(\xi;\lambda)\| , 
\end{equation}
where $N$ ranges over non-negative integers and $\Lambda'$ ranges over pseudo-cones such that $\Lambda'\subsubset\Lambda$. 
\end{remark}

\begin{remark} \label{rmk:Parameter.standard-symbols-with-parameter-identification}
Let $\rho(\xi;\lambda)\in \stS^{m,d}(\R^n\times \Lambda;\cA_\theta)$. For every $\lambda\in \Lambda$ we get a symbol $\rho(\cdot;\lambda)\in \stS^m(\R^n;\cA_\theta)$, and so we get a family in $\stS^m(\R^n;\cA_\theta)$ parametrized by $\Lambda$. This is actually an $\Hol^d(\Lambda)$-family. This provides us with a  one-to-one correspondence,  
\begin{equation}
 \stS^{m,d}(\R^n\times \Lambda; \cA_\theta)\simeq \Hol^d(\Lambda;\stS^m(\R^n;\cA_\theta)).
 \label{eq:Parameter.parametric-symbols-identification} 
\end{equation}
Under this correspondence the semi-norms~(\ref{eq:Parameter.standard-symbol-with-parameter-semi-norm}) generate the topology of $\Hol^d(\Lambda;\stS^m(\R^n;\cA_\theta))$, and so we actually have a topological vector space isomorphism.  
\end{remark}

\begin{definition}
 $\stS^{-\infty,d}(\R^n\times \Lambda; \cA_\theta)$, $d\in \R$, consists of maps $\rho(\xi;\lambda)\in C^{\infty,\omega}(\R^n\times \Lambda;\cA_\theta)$ such that, given any $N\geq 0$, for all pseudo-cones $\Lambda'\subsubset \Lambda$ and multi-orders $\alpha$ and $\beta$, there is $C_{\Lambda'N\alpha\beta}>0$ such that
\begin{equation*}
 \big\|\delta^\alpha \partial_\xi^\beta \rho(\xi;\lambda)\big\| \leq  C_{\Lambda'N\alpha\beta}(1+|\lambda|)^d(1+|\xi|)^{-N} \qquad \forall (\xi, \lambda)\in \R^n\times \Lambda'. 
\end{equation*}
\end{definition}

\begin{remark}
We have $\stS^{-\infty,d}(\R^n\times \Lambda; \cA_\theta)=\bigcap_{m\in \R} \stS^{m,d}(\R^n\times \Lambda; \cA_\theta)$. 
\end{remark}

\begin{remark}
$\stS^{-\infty,d}(\R^n\times \Lambda; \cA_\theta)$ is a locally convex space with respect to the topology generated by the semi-norms~(\ref{eq:Parameter.standard-symbol-with-parameter-semi-norm}), where we let $m$ ranges over all real numbers.
\end{remark}

\begin{remark}
 The one-to-one correspondance~(\ref{eq:Parameter.parametric-symbols-identification}) induces a topological vector space isomorphism,  
\begin{equation}
\stS^{-\infty,d}(\R^n\times \Lambda; \cA_\theta)\simeq \Hol^d(\Lambda;\cS(\R^n;\cA_\theta)).
\label{eq:Parameter.smoothing-parametric-symbols-identification}
\end{equation}
\end{remark}

Combining the topological vector space identifications~(\ref{eq:Parameter.parametric-symbols-identification}) and~(\ref{eq:Parameter.smoothing-parametric-symbols-identification}) with Proposition~\ref{prop:Parameter.linear-map} yields the following result. 

\begin{proposition}\label{prop:Parameter.linear-map-symbols}
 Suppose that $\sE$ is a locally convex space, and let $d\in \R$. 
\begin{enumerate}
 \item Given $m\in \R$, any continuous linear map $T:\stS^{m}(\R^n; \cA_\theta) \rightarrow \sE$ gives rise to a continuous linear map 
 $T:\stS^{m,d}(\R^n\times \Lambda; \cA_\theta) \rightarrow \Hol^d(\Lambda;\sE)$. 
 
 \item Any continuous linear map $T:\cS(\R^n; \cA_\theta) \rightarrow \sE$ gives rise to a continuous linear map 
 $T:\stS^{-\infty,d}(\R^n\times \Lambda; \cA_\theta) \rightarrow \Hol^d(\Lambda;\sE)$.
\end{enumerate}
\end{proposition}

Likewise, by combining the identification~(\ref{eq:Parameter.parametric-symbols-identification}) with Proposition~\ref{prop:Parameter.bilinear-map} we arrive at the following statement. 

\begin{proposition}\label{prop:Parameter.bilinear-map-symbols}
 Given $m_1,m_2, m_3\in \R$ and $d_1,d_2\in \R$, any (jointly) continuous bilinear map $\Phi: \stS^{m_1}(\R^n; \cA_\theta)\times  \stS^{m_2}(\R^n; \cA_\theta) \rightarrow  \stS^{m_3}(\R^n; \cA_\theta)$ gives rise to a continuous bilinear map 
 $\Phi: \stS^{m_1,d_1}(\R^n\times \Lambda; \cA_\theta)\times  \stS^{m_2, d_2}(\R^n\times \Lambda; \cA_\theta) \rightarrow  \stS^{m_3,d_1+d_2}(\R^n\times \Lambda; \cA_\theta)$.
\end{proposition}

\subsection{Asymptotic expansions of parametric symbols}
In what follows we let $(m_j)_{j \geq 0}$ be a (strictly) decreasing sequence of real numbers converging to $-\infty$. 

\begin{definition} \label{def:Parameter.Standard-asymptotic}
Given $\rho(\xi;\lambda)\in C^{\infty,\omega}(\Rn\times\Lambda;\cA_\theta)$ and $\rho_{j}(\xi;\lambda)\in\stS^{m_j,d}(\Rn\times\Lambda;\cA_\theta)$, $j\geq 0$, we shall write $\rho(\xi;\lambda)\sim\sum_{j\geq 0}\rho_{j}(\xi;\lambda)$ when
\begin{equation} \label{eq:Parameter.Standard-asymptotic}
\rho(\xi;\lambda)-\sum_{j<N}\rho_{j}(\xi;\lambda)\in\stS^{m_N,d}(\Rn\times\Lambda;\cA_\theta) \qquad \text{for all $N\geq 0$} .
\end{equation}
\end{definition}

\begin{remark}
 The condition~(\ref{eq:Parameter.Standard-asymptotic}) for $N=0$ means that $\rho(\xi;\lambda)\in \stS^{m_0,d}(\Rn\times\Lambda;\cA_\theta)$. 
\end{remark}

For our purpose, it will be convenient to have a \emph{qualitative} version of the conditions~(\ref{eq:Parameter.Standard-asymptotic}).  

\begin{lemma} \label{lem:Parameter.standard-symbol-asymptotics-qualitative-equivalence}
 Given $\rho(\xi;\lambda)\in C^{\infty,\omega}(\Rn\times\Lambda;\cA_\theta)$ and $\rho_{j}(\xi;\lambda)\in\stS^{m_j,d}(\Rn\times\Lambda;\cA_\theta)$, $j\geq 0$, the following are equivalent:
\begin{enumerate}
 \item[(i)] $\rho(\xi;\lambda)\sim\sum_{j\geq 0}\rho_{j}(\xi;\lambda)$ in the sense of~(\ref{eq:Parameter.Standard-asymptotic}). 
 
 \item[(ii)] For all $N\geq 0$, given any pseudo-cone $\Lambda'\subsubset \Lambda$ and multi-orders $\alpha$ and $\beta$, as soon as $J$ is large enough there is a constant $C>0$ such that, for all $(\xi,\lambda)\in \R^n\times \Lambda'$, we have 
\begin{equation} \label{eq:Parameter.Standard-asymptotic-qualitative}
 \big\|\delta^\alpha \partial_\xi^\beta \big(\rho-\sum_{j<J}\rho_{j}\big)(\xi;\lambda)\big\| \leq C(1+|\lambda|)^d(1+|\xi|)^{-N}. 
\end{equation}
\end{enumerate}
\end{lemma}
\begin{proof}
 It is immediate that (i) implies (ii), so we only have to prove the converse. Suppose that (ii) holds. Let $N\geq 0$. Then (ii) implies that, given any pseudo-cone 
 $\Lambda'\subsubset \Lambda$, if $J$ is large enough, then there is $C>0$ such that 
\begin{equation} \label{eq:Parameter-Standard-asymptotic-qualitative-to-usual-definition-proof}
 \big\|\rho(\xi;\lambda)-\sum_{j<J}\rho_{j}(\xi;\lambda)\big\| \leq C(1+|\lambda|)^{d}(1+|\xi|)^{m_N}\qquad \forall (\xi,\lambda)\in \R^n\times \Lambda'.
\end{equation}
If we take $J$ large enough so as to have  $J\geq N$ as well, then we have 
\begin{equation*}
\rho(\xi;\lambda)-\sum_{j<N}\rho_{j}(\xi;\lambda)= \rho(\xi;\lambda)-\sum_{j<J}\rho_{j}(\xi;\lambda)+ \sum_{N\leq j<J} \rho_{j}(\xi;\lambda).  
\end{equation*}
If $j \geq N$, then $\rho_{j}(\xi;\lambda)$ satisfies the estimate~(\ref{eq:Parameter-Standard-asymptotic-qualitative-to-usual-definition-proof}),  since it is contained in $\stS^{m_j, d}(\R^n\times \Lambda; \cA_\theta) \subset \stS^{m_N, d}(\R^n\times \Lambda; \cA_\theta)$. Therefore, the difference $\rho(\xi;\lambda)-\sum_{j<N}\rho_{j}(\xi;\lambda)$ satisfies the estimate~(\ref{eq:Parameter-Standard-asymptotic-qualitative-to-usual-definition-proof}).

Likewise, given any multi-orders $\alpha$ and $\beta$ there is $C_{N\Lambda'\alpha\beta}>0$ such that 
\begin{equation*}
 \big\|\delta^\alpha \partial_\xi^\beta\big(\rho-\sum_{j<N}\rho_{j}\big)(\xi;\lambda)\big\| \leq C_{N\Lambda'\alpha\beta}(1+|\lambda|)^{d}(1+|\xi|)^{m_N-|\beta|} \qquad \forall (\xi,\lambda)\in \R^n\times \Lambda'. 
\end{equation*}
This shows that $\rho(\xi;\lambda)-\sum_{j<N}\rho_{j}(\xi;\lambda)\in \stS^{m_N,d}(\R^n\times \Lambda; \cA_\theta)$ for all $N\geq 0$. That is, we have $\rho(\xi;\lambda)\sim\sum_{j\geq 0}\rho_{j}(\xi;\lambda)$ in the sense of~(\ref{eq:Parameter.Standard-asymptotic}). The proof is complete.   
\end{proof}

\subsection{Homogeneous parametric symbols}
Throughout the rest of this paper we let $w$ be some fixed (strictly) positive real number. We also denote by $\Theta$ the conical component of $\Lambda$. Note this is an open cone contained in $\C^*$. In addition, given any $c\geq 0$, we  set
\begin{equation} \label{eq:Parameter.parameter-set-for-homogeneous-symbols}
\Omega_c(\Theta) = \big\{ (\xi,\lambda)\in(\Rn\setminus 0)\times\C; \ \text{$\lambda\in\Theta$ or $|\lambda|<c|\xi|^w$} \big\} .
\end{equation}
In particular, $\Omega_0(\Theta) = (\Rn\setminus 0)\times\Theta$.

\begin{definition} \label{def:Parameter.homogeneous-symbol}
$S_m^d(\Omega_c(\Theta);\cA_\theta)$, $m,d\in\R$, consists of maps $\rho(\xi;\lambda)\in C^{\infty,\omega}(\Omega_c(\Theta);\cA_\theta)$ that satisfy the following two properties:
\begin{enumerate}
\item[(i)] $\rho(t\xi;t^w\lambda) = t^m\rho(\xi;\lambda)$ for all $(\xi,\lambda)\in\Omega_c(\Theta)$ and $t>0$.

\item[(ii)] Given any cone $\Theta'$ such that $\overline{\Theta'}\setminus \{0\} \subset \Theta$, for all compacts $K\subset \R^n\setminus 0$ and multi-orders $\alpha$, $\beta$, there is $C_{\Theta'K\alpha\beta}>0$ such that 
\begin{equation} \label{eq:Parameter.homogeneous-symbol-with-parameter-estimate}
 \norm{\delta^\alpha \partial_\xi^\beta \rho(\xi;\lambda)} \leq C_{\Theta'K\alpha\beta} \big(1+|\lambda|\big)^d \qquad \forall (\xi,\lambda)\in K\times \Theta'. 
\end{equation}
\end{enumerate}
\end{definition}

\begin{example} \label{ex:Parameter.homogeous-symbol-can-be-seen-as-homogeneous-symbol-with-parameter}
 Any homogeneous symbol $\rho(\xi)\in S_m(\R^n; \cA_\theta)$ can be regarded as an element of $S^{0}_m(\Omega_c(\Theta);\cA_\theta)$ for every $c>0$. 
\end{example}

\begin{remark} \label{rem:Parameter.homogeneous-symbol-partial-derivative}
If $\rho(\xi;\lambda)\in S_m^d(\Omega_c(\Theta);\cA_\theta)$, then $\delta^\alpha\partial_\xi^\beta\rho(\xi;\lambda)\in S_{m-|\beta|}^d(\Omega_c(\Theta);\cA_\theta)$ for all $\alpha,\beta\in\N_0^n$.
\end{remark}

\begin{remark} \label{rmk:Parameter.homogeneous-symbol-product}
 If $\rho_1(\xi;\lambda)\in S_{m_1}^{d_1} (\Omega_c(\Theta);\cA_\theta)$ and $\rho_2(\xi;\lambda)\in S_{m_2}^{d_2} (\Omega_c(\Theta);\cA_\theta)$, then 
 $\rho_1(\xi;\lambda)\rho_2(\xi;\lambda)$ is contained in $S_{m_1+m_2}^{d_1+d_2} (\Omega_c(\Theta);\cA_\theta)$. 
\end{remark}

The next lemma shows how to cut off homogeneous parametric symbols in order to get standard parametric symbols. 

\begin{lemma} \label{lem:Parameter.homogeneous-symbol-estimate}
Suppose that $\rho(\xi;\lambda)\in S_m^d(\Omega_c(\Theta);\cA_\theta)$, $m,d\in\R$. Let $\chi(\xi)\in C_c^\infty(\Rn)$ be such that $\chi(\xi) = 1$ for 
$|\xi|\leq (c^{-1}R)^{\frac{1}{w}}$, where $R>0$ is such that $\Lambda \subset \Theta \cup D(0,R)$. 
\begin{enumerate}
 \item[(i)] If $d\geq 0$, then $(1-\chi(\xi))\rho(\xi;\lambda)\in \stS^{m,d}(\Rn\times\Lambda;\cA_\theta)$.
 
 \item[(ii)] If $d<0$, then  
 \begin{equation}
 \label{eq:Parameter.homogeneous-symbol-times-cutoff-is-standard-symbol-negative-d-case}
 (1-\chi(\xi))\rho(\xi;\lambda)\in  \bigcap_{d\leq d'\leq 0} \stS^{m+w|d'|,d'}(\Rn\times\Lambda;\cA_\theta). 
 \end{equation}
\end{enumerate}
\end{lemma}
\begin{proof}
As the relation $\Lambda'\subsubset \Lambda$ implies that $\Lambda'\subsubset \Theta \cup D(0,R)$, the restriction to $\R^n\times \Lambda$ of any 
element in $\stS^{m,d}(\R^n\times (\Theta \cup D(0,R)); \cA_\theta)$ is contained in $\stS^{m,d}(\R^n\times \Lambda; \cA_\theta)$. Therefore, without any loss of generality we may assume that $\Lambda = \Theta \cup D(0,R)$. 
 
Set $B=\{\xi\in \R^n; \ |\xi| <(c^{-1}R)^{\frac1{w}}\}$ and $U=\R^n\setminus B=\{\xi\in \R^n;\ c|\xi|^w\geq R\}$. As $\Lambda = \Theta \cup D(0,R)$, we have $U\times \Lambda = (U\times \Theta)\cup (U\times D(0,R))$. It is immediate that $U\times \Theta\subset \Omega_c(\Theta)$. Moreover, if $(\xi,\lambda)\in U\times D(0,R)$, then
 $c|\xi|^w\geq R>|\lambda|$, and hence $(\xi,\lambda)\in \Omega_c(\Theta)$. Therefore, we see that $U\times \Lambda\subset \Omega_c(\Theta)$, and hence we have
\begin{equation*}
 \R^n\times \Lambda = \big( U\times \Lambda\big) \cup \big( B\times \Lambda\big) \subset \Omega_c(\Theta) \cup \big( B\times \Lambda\big). 
\end{equation*}
As $\rho(\xi; \lambda) \in C^{\infty,\omega}(\Omega_c(\Theta); \cA_\theta)$ and $\chi(\xi)=1$ on $B$, if we extend $(1-\chi(\xi))\rho(\xi;\lambda)$ to be zero on $B\times \Lambda$, then we get a $C^{\infty,\omega}$-map on $\Omega_c(\Theta) \cup ( B\times \Lambda)$. By restriction this gives a map in $C^{\infty,\omega}(\R^n\times \Lambda;\cA_\theta)$. It then remains to show that $(1-\chi(\xi))\rho(\xi;\lambda)$ satisfies the estimates~(\ref{eq:Parameter.standard-symbol-with-parameter-estimate}) on $\R^n\times \Lambda$. 

In what follows we set  $d_{-}=\max(0,-d)$. 

\begin{claim*}
Given any pseudo-cone $\Lambda'\subsubset \Lambda$ and multi-orders $\alpha, \beta \in \N_0^n$, there is $C_{\Lambda'\alpha\beta}>0$ such that
\begin{equation} \label{eq:Parameter.homogeneous-symbol-estimate-after-separate-xi}
\|\delta^\alpha\partial_\xi^\beta\rho(\xi;\lambda)\| \leq C_{\Lambda'\alpha\beta}(1+|\lambda|)^d|\xi|^{m+wd_{-}-|\beta|} \qquad \forall (\xi,\lambda)\in U\times\Lambda'. 
\end{equation}
\end{claim*}
\begin{proof}[Proof of the Claim]
Without any loss of generality we may assume that $\Lambda'=\Theta'\cup D(0,R')$, where $R'<R$ and $\Theta'\subset \C^*$ is a cone such that $\overline{\Theta'}\setminus \{0\}\subset \Theta$. Moreover, thanks to the homogeneity of $\delta^\alpha \partial_\xi^\beta \rho(\xi;\lambda)$, for all $(\xi,\lambda)\in \Omega_c(\Theta)$, we have 
\begin{equation} \label{eq:Parameter.homogeneous-symbol-estimate-inducing-to-compact-set}
\delta^\alpha \partial_\xi^\beta \rho(\xi;\lambda)  = \delta^\alpha \partial_\xi^\beta \rho\left[ |\xi|\big(|\xi|^{-1}\xi\big);  |\xi|^{w}\big(|\xi|^{-w}\lambda\big)\right] 
 = |\xi|^{m-|\beta|} \delta^\alpha \partial_\xi^\beta \rho\big(|\xi|^{-1}\xi; |\xi|^{-w}\lambda\big). 
\end{equation}

Set $D'=D(0,R')$, so that $\Lambda' = \Theta'\cup D'$. As $\bS^{n-1}$ is a compact subset of $\R^n\setminus 0$ and $\overline{\Theta'}\setminus \{0\}\subset \Theta$, we know by~(\ref{eq:Parameter.homogeneous-symbol-with-parameter-estimate}) there is $C_1>0$ such that
\begin{equation*}
 \big\| \delta^\alpha \partial_\xi^\beta \rho(\eta; \mu)\big\| \leq C_1 (1+|\mu|)^d \qquad \forall (\eta,\mu)\in \bS^{n-1}\times \Theta'. 
\end{equation*}
If $(\xi,\lambda)\in U\times \Theta'$, then $(|\xi|^{-1}\xi, |\xi|^{-w}\lambda)\in \bS^{n-1}\times \Theta'$. Therefore, by using~(\ref{eq:Parameter.homogeneous-symbol-estimate-inducing-to-compact-set}) we see that, for all 
$(\xi,\lambda)\in U\times \Theta'$,  we have 
\begin{equation} \label{eq:Parameter.homogeneous-symbol-estimate-on-cone}
  \big\| \delta^\alpha \partial_\xi^\beta \rho(\xi;\lambda)\big\| =  |\xi|^{m-|\beta|} \big\|\delta^\alpha \partial_\xi^\beta \rho\big(|\xi|^{-1}\xi; |\xi|^{-w}\lambda\big)\big\| 
   \leq C_1 |\xi|^{m-|\beta|}\big( 1+ |\xi|^{-w}|\lambda|\big)^d. 
\end{equation}

If $(\xi,\lambda) \in U\times D'$, then $c|\xi|^w\geq R= R(R')^{-1}R'\geq R(R')^{-1}|\lambda|$, and so $|\xi|^{-w}|\lambda|\leq cR'R^{-1}$. Thus, if we set $K=\bS^{n-1}\times 
\overline{D}(0,cR'R^{-1})$, then $K$ is compact and $(|\xi|^{-1}\xi, |\xi|^{-w}\lambda)\in K$ for all $(\xi,\lambda)\in U\times D'$. Note also that
\begin{align*}
 \Omega_c(\Theta) \cap \big(\bS^{n-1}\times \C\big) & = \left\{ (\xi,\lambda)\in \bS^{n-1}\times \C; \ \text{$\lambda\in\Theta$ or $|\lambda|<c|\xi|^w$}\right\} \\
 & = \left( \bS^{n-1}\times \Theta\right) \cup  \left( \bS^{n-1}\times D(0,c)\right) . 
\end{align*}
Therefore, we see that $K$ is a compact subset of $\Omega_c(\Theta)$. As $\rho(\xi,\lambda)\in C^{\infty,\omega}(\Omega_c(\Theta); \cA_\theta)$ we deduce there is $C_2>0$ such that, for all $(\eta,\mu)\in K$, we have
\begin{equation} \label{eq:Parameter.homogeneous-symbol-estimate-on-sphere-times-disk}
 \big\| \delta^\alpha \partial_\xi^\beta \rho(\eta; \mu)\big\| \leq C_2 \leq C_2'(1+|\mu|)^{d},
\end{equation}
where we have set $C_2'=C_2\sup\{(1+|\mu|)^{-d}; \mu\in D(0,cR'R^{-1})\}$. As mentioned above, if $(\xi,\lambda)$ is in $U\times D'$, then
$(|\xi|^{-1}\xi, |\xi|^{-w}\lambda)\in K$. Therefore, by using~(\ref{eq:Parameter.homogeneous-symbol-estimate-on-sphere-times-disk}) we deduce that, for all $(\xi,\lambda)\in U\times D'$, we have 
\begin{equation} \label{eq:Parameter.homogeneous-symbol-estimate-on-disk}
  \big\| \delta^\alpha \partial_\xi^\beta \rho(\xi;\lambda)\big\| = |\xi|^{m-|\beta|} \big\|\delta^\alpha \partial_\xi^\beta \rho\big(|\xi|^{-1}\xi; |\xi|^{-w}\lambda\big)\big\| 
   \leq C_2' |\xi|^{m-|\beta|}\big( 1+ |\xi|^{-w}|\lambda|\big)^d. 
\end{equation}

Note that if $\xi\in U$, then $c|\xi|^w \geq R$, and so $|\xi|^{-w} \leq c R^{-1}$ and $1\geq c^{-1} R|\xi|^{-w}$. Thus, 
\begin{gather*}
 1+|\xi|^{-w} |\lambda| \leq 1+ cR^{-1} |\lambda| \leq c_1(1+|\lambda|),\\
 1+|\xi|^{-w} |\lambda| \geq  c^{-1}R|\xi|^{-w} + |\xi|^{-w} |\lambda|  \geq c_1^{-1}|\xi|^{-w}(1+|\lambda|), 
\end{gather*}
where we have set $c_1=\max(1, cR^{-1})$. It then follows that
\begin{equation*}
 \big( 1+ |\xi|^{-w}|\lambda|\big)^d\leq c_1^{|d|}(1+|\lambda|)^d |\xi|^{wd_{-}} \qquad \forall (\xi,\lambda)\in U\times \Lambda'. 
\end{equation*}
Combining this with~(\ref{eq:Parameter.homogeneous-symbol-estimate-on-cone}) and~(\ref{eq:Parameter.homogeneous-symbol-estimate-on-disk}) proves the claim. 
\end{proof}

Given $\alpha,\beta\in\N_0^n$, let $\Lambda'$ be a pseudo-cone such that  $\Lambda'\subsubset \Lambda$. By Leibniz's rule we have
\begin{equation} \label{eq:Parameter.cutoff-function-times-homogeneous-symbol-Leibniz-rule}
\delta^\alpha\partial_\xi^\beta[(1-\chi(\xi))\rho(\xi;\lambda)]  = \sum \binom \beta {\beta'} \partial_\xi^{\beta'}(1-\chi(\xi))\delta^\alpha\partial_\xi^{\beta''}\rho(\xi;\lambda) ,
\end{equation}
where the sum ranges over all multi-orders $\beta'$ and $\beta''$ such that $\beta'+\beta'' = \beta$. Note that each $\partial_\xi^{\beta'}(1-\chi(\xi))$ is uniformly bounded on $\R^n$ and vanishes on $B\times \Lambda'$. Combining this with~(\ref{eq:Parameter.homogeneous-symbol-estimate-after-separate-xi}) allows us to show there is $C_{\Lambda'\alpha\beta}>0$ such that
\begin{equation} \label{eq:Parameter.cutoff-function-times-homogeneous-symbol-estimate-before-separate-xi}
\Big\| \delta^\alpha\partial_\xi^\beta[(1-\chi(\xi))\rho(\xi;\lambda)] \Big\| \leq  C_{\Lambda'\alpha\beta}(1+|\lambda|)^d(1+|\xi|)^{m+wd_{-}-|\beta|}
\quad \forall (\xi,\lambda)\in \R^n\times\Lambda'.
\end{equation}
This shows that $(1-\chi(\xi))\rho(\xi;\lambda)\in \stS^{m+wd_{-},d}(\R^n\times \Lambda; \cA_\theta)$. In particular,  $(1-\chi(\xi))\rho(\xi;\lambda)$ is contained in  $\stS^{m,d}(\R^n\times \Lambda; \cA_\theta)$ when $d\geq 0$. 

Suppose that $d<0$. In this case $d_{-}=-d=|d|$, and so  $(1-\chi(\xi))\rho(\xi;\lambda)$ is contained in $\stS^{m+w|d|,d}(\R^n\times \Lambda; \cA_\theta)$. Note that if $d'\in [d,0]$, then $\rho(\xi;\lambda)$ is contained in $S_q^{d'}(\Omega_c(\Theta);\cA_\theta)$, since $S_q^{d'}(\Omega_c(\Theta);\cA_\theta)$ contains $S_q^{d}(\Omega_c(\Theta);\cA_\theta)$. Therefore, the above arguments also show that  $(1-\chi(\xi))\rho(\xi;\lambda)$  is in $\stS^{m+w|d'|,d'}(\Rn\times\Lambda;\cA_\theta)$ for any $d'\in [d,0]$. This gives~(\ref{eq:Parameter.homogeneous-symbol-times-cutoff-is-standard-symbol-negative-d-case}). The proof is complete.
 \end{proof}

\subsection{Classical parametric symbols} 
In what follows, given any $r>0$, we denote by $B(r)$ the ball of radius $r$ about the origin in $\R^n$. We observe that if $c>0$ and $R>0$ is such that $\Lambda \subset \Theta \cup D(0,R)$, then 
\begin{equation*}
 \Omega_c(\Theta) \supset \big( \R^n \setminus B(r) \big) \times \Lambda \qquad \text{for all}\  r> (c^{-1}R)^{\frac1{w}}. 
\end{equation*}
Indeed, $(\R^n\setminus 0)\times \Theta \supset ( \R^n \setminus B(r))\times \Theta$. Moreover, if $|\lambda|<R$ and $|\xi|\geq r>(c^{-1}R)^{\frac1{w}}$, 
then $|\lambda|<R<c|\xi|^w$. Thus,
\begin{equation*}
  \Omega_c(\Theta) \supset \big[ \big( \R^n \setminus B(r) \big) \times \Theta\big] \bigcup \big[ \big( \R^n \setminus B(r) \big) \times D(0,R)\big] \supset \big( \R^n \setminus B(r) \big) \times \Lambda. 
\end{equation*}

\begin{definition} \label{def:Parameter.classical-symbol}
$S^{m,d}(\Rn\times\Lambda;\cA_\theta)$, $m, d\in\R$, consists of maps $\rho(\xi;\lambda)\in C^{\infty,d}(\Rn\times\Lambda;\cA_\theta)$ for which there are $c>0$ and 
$\rho_{m-j}(\xi;\lambda)\in S_{m-j}^d(\Omega_c(\Theta);\cA_\theta)$, $j\geq 0$, such that 
\begin{equation*}
\rho(\xi;\lambda)\sim\sum_{j\geq 0}\rho_{m-j}(\xi;\lambda), 
\end{equation*}
in the sense that, for all $N\geq 0$ and multi-orders $\alpha$, $\beta$, given any pseudo-cone $\Lambda'\subsubset \Lambda$ and $r>0$ such that 
$ \Omega_c(\Theta) \supset \big( \R^n \setminus B(r) \big) \times \Lambda'$, as soon as $J$ is large enough there is $C_{\Lambda'NJr\alpha \beta}>0$ such that, for all $(\xi,\lambda)\in (\R^n\setminus B(r)) \times\Lambda'$, we have 
\begin{equation} \label{eq:Parameter.classical-symbol-asymptotics}
 \big\| \delta^\alpha \partial_\xi^\beta \big(\rho-\sum_{j<J}\rho_{m-j}\big)(\xi;\lambda)\big\| \leq C_{\Lambda'NJr\alpha \beta}|\xi|^{-N}(1+|\lambda|)^{d} . 
\end{equation}
\end{definition}

\begin{example} \label{ex:Parameter.classical-symbol-can-be-seen-as-classical-symbol-with-parameter}
 Let $\rho(\xi)\in S^m(\R^n;\cA_\theta)$, $\rho(\xi)\sim \sum \rho_{m-j}(\xi)$. As mentioned in Example~\ref{ex:Parameter.homogeous-symbol-can-be-seen-as-homogeneous-symbol-with-parameter} each homogeneous symbol $\rho_{m-j}(\xi)$ can be regarded as an element of $S^{0}_{m-j}(\Omega_c(\Theta);\cA_\theta)$ for any $c>0$. Note also that having an asymptotic expansion $\rho(\xi)\sim \sum \rho_{m-j}(\xi)$ in the sense of~(\ref{eq:Symbols.classical-estimates}) implies that we also have an asymptotic expansion in the sense of~(\ref{eq:Parameter.classical-symbol-asymptotics}) with $d=0$. It then follows that $\rho(\xi)\in S^{m,0}(\R^n\times \Lambda;\cA_\theta)$. 
\end{example}


\begin{remark} \label{rem:Parameter.classical-symbol-partial-derivative}
Let $\rho(\xi;\lambda)\in S^{m,d}(\Rn\times\Lambda;\cA_\theta)$. Thus, $\rho(\xi;\lambda)\sim\sum_{j\geq 0}\rho_{m-j}(\xi;\lambda)$,  where $\rho_{m-j}(\xi;\lambda)$ is in $S_{m-j}^{d}(\Omega_c(\Theta);\cA_\theta)$. Given any multi-orders $\alpha$, $\beta$,  we also have $\delta^\alpha\partial_\xi^\beta\rho(\xi;\lambda)\sim\sum_{j\geq 0}\delta^\alpha\partial_\xi^\beta\rho_{m-j}(\xi;\lambda)$ in the sense of~(\ref{eq:Parameter.classical-symbol-asymptotics}), and hence $\delta^\alpha\partial_\xi^\beta\rho(\xi;\lambda)\in S^{m-|\beta|,d}(\Rn\times\Lambda;\cA_\theta)$.
\end{remark}

\begin{remark} \label{rmk:Parameter.classical-symbol-asymptotics-equivalent-conditions}
Suppose we are given $\rho_{m-j}(\xi;\lambda)\in S_{m-j}^{d}(\Omega_c(\Theta);\cA_\theta)$, $j\geq 0$. Let $\chi(\xi)\in C^\infty_c(\R^n)$ be as in Lemma~\ref{lem:Parameter.homogeneous-symbol-estimate}. Recall that by Lemma~\ref{lem:Parameter.homogeneous-symbol-estimate} $(1-\chi(\xi))\rho_{m-j}(\xi;\lambda)\in \stS^{m-j+wd_{-},d}(\Omega_c(\Theta);\cA_\theta)$, where $d_{-}=\max(0,-d)$.  
Then, the following are equivalent: 
\begin{enumerate}
 \item[(i)] $\rho(\xi;\lambda)\sim \sum_{j\geq 0} \rho_{m-j}(\xi;\lambda)$ in the sense of~(\ref{eq:Parameter.classical-symbol-asymptotics}). 
 
 \item[(ii)] $\rho(\xi;\lambda)\sim \sum_{j\geq 0} (1-\chi(\xi))\rho_{m-j}(\xi;\lambda)$ in the sense of~(\ref{eq:Parameter.Standard-asymptotic}) or~(\ref{eq:Parameter.Standard-asymptotic-qualitative}). 
\end{enumerate}
Note that here~(\ref{eq:Parameter.Standard-asymptotic}) means that,  for all $N\geq 0$, as soon as $J\geq N+wd_{-}$, we have 
\begin{equation} \label{eq:Parameter.classical-symbol-asymptotics-qualitative}
\rho(\xi;\lambda) - \sum_{j<J}  (1-\chi(\xi))\rho_{m-j}(\xi;\lambda) \in \stS^{m-N,d}(\R^n\times \Lambda; \cA_\theta). 
\end{equation}
This provides us with a quantitative version of the estimates~(\ref{eq:Parameter.classical-symbol-asymptotics}). 
 \end{remark}
 
Combining the above remark with Lemma~\ref{lem:Parameter.homogeneous-symbol-estimate} we also get the following result. 
\begin{proposition}\label{symbols:inclusion-classical-standard}
 Let $m,d\in \R$. 
\begin{enumerate}
 \item If $d\geq 0$, then $S^{m,d}(\Rn\times\Lambda;\cA_\theta)\subset\stS^{m,d}(\Rn\times\Lambda;\cA_\theta)$. 
 
 \item If $d<0$, then  
        \begin{equation*}
              S^{m,d}(\Rn\times\Lambda;\cA_\theta)\subset  \bigcap_{d\leq d'\leq 0} \stS^{m+w|d'|,d'}(\Rn\times\Lambda;\cA_\theta). 
         \end{equation*}
\end{enumerate}
 \end{proposition}

\subsection{Borel lemma for parametric symbols} 
In the following, given any $m\in \R$, we denote by $\stS^m(\R^n)$ the space of scalar standard symbols on $\R^n$. It consists of functions $\sigma(\xi)\in C^\infty(\R^n)$ such that, for every multi-order $\alpha$, there is $C_\alpha>0$ such that
\begin{equation*}
 |\partial_\xi^\alpha \sigma(\xi)| \leq C_\alpha (1+|\xi|)^{m-|\alpha|} \qquad \text{for all $\xi \in \R^n$}. 
\end{equation*}
We equip $\stS^m(\R^n)$ with the locally convex topology generated by the semi-norms, 
\begin{equation*}
 \sigma(\xi) \longrightarrow \sup_{|\alpha|\leq N} \sup_{\xi \in \R^n} (1+|\xi|)^{-m+|\alpha|} |\partial_\xi^\alpha \sigma(\xi)|, \qquad N\geq 0. 
\end{equation*}

\begin{lemma}[{\cite[Prop.~18.1.2]{Ho:Springer85}}] \label{lem:Parameter.standard-symbol-density-lemma}
Let $\chi\in \cS(\R^n)$ be such that $\chi(0)=1$. Then the family $(\chi(\epsilon \xi))_{\epsilon >0}$ converges to $1$ in $\stS^m(\R^n)$ as $\epsilon \rightarrow 0^+$ for every $m>0$. 
\end{lemma}

We have the following density result in $\stS^{m,d}(\R^n\times \Lambda; \cA_\theta)$. 

\begin{lemma} \label{lem:Parameter.standard-symbol-with-parameter-density-lemma}
Let $\chi(\xi)\in \cS(\R^n)$ be such that $\chi(0)=1$. Given any $\rho(\xi;\lambda)\in \stS^{m,d}(\R^n\times \Lambda; \cA_\theta)$, $m,d\in \R$, as $\epsilon \rightarrow 0^+$, the family $\chi(\epsilon \xi)\rho(\xi;\lambda)$ converges  to $\rho(\xi;\lambda)$ in $ \stS^{m',d}(\R^n\times\Lambda; \cA_\theta)$ for all $m'>m$. 
\end{lemma}
\begin{proof}
 Let $m'>m$. The product of $\cA_\theta$ gives rise to a continuous bilinear maps from $\stS^{m_1}(\R^n;\cA_\theta)\times \stS^{m_2}(\R^n;\cA_\theta)$ to $\stS^{m_1+m_2}(\R^n;\cA_\theta)$ for all $m_1, m_2\in \R$. Proposition~\ref{prop:Parameter.bilinear-map-symbols} then ensures us that we get a continuous bilinear map, 
\begin{equation} \label{eq:Parameter.standard-symbol-with-parameter-product}
\stS^{m_1,d_1}(\R^n\times \Lambda;\cA_\theta)\times \stS^{m_2,d_2}(\R^n \times \Lambda;\cA_\theta) \longrightarrow \stS^{m_1+m_2,d_1+d_2}(\R^n\times \Lambda;\cA_\theta).
\end{equation}
In particular, by using the continuity of the inclusion of $\stS^{m'-m}(\R^n)$ into $\stS^{m'-m,0}(\R^n \times \Lambda;\cA_\theta)$ we obtain a continuous bilinear map, 
\begin{equation*}
 \stS^{m'-m}(\R^n)\times \stS^{m,d}(\R^n \times \Lambda;\cA_\theta) \longrightarrow \stS^{m',d}(\R^n\times \Lambda;\cA_\theta). 
\end{equation*}
Combining Lemma~\ref{lem:Parameter.standard-symbol-density-lemma} with the continuity of the above bilinear map~(\ref{eq:Parameter.standard-symbol-with-parameter-product}) then gives the result. The proof is complete. 
\end{proof}

The following is the version of Borel's lemma for standard parametric symbols.  

\begin{lemma}\label{lem:Parameter.Borel}
For $j=0,1,\ldots$, let $\rho_{j}(\xi;\lambda)\in\stS^{m_j,d}(\Rn\times\Lambda;\cA_\theta)$.  Then there exists $\rho(\xi;\lambda)$ in $\stS^{m_0,d}(\R^n\times \Lambda;\cA_\theta)$ such that  $\rho(\xi;\lambda)\sim\sum_{j\geq 0}\rho_{j}(\xi;\lambda)$.
\end{lemma}
\begin{proof}
Let $\Lambda = \bigcup_{j\geq 0} \Lambda_j$ be a pseudo-cone exhaustion of $\Lambda$ with $\Lambda_{j} \subsubset \Lambda_{j+1}$ (\emph{cf}.\ Remark~\ref{rmk:Parameter.pseudo-cone-exhaustion}). Given any $m\in \R$ the topology of the $\stS^{m,d}(\R^n\times \Lambda; \cA_\theta)$ is generated by the semi-norms, 
\begin{equation*}
 p_j^{(m)}(\rho) = \max_{|\alpha|+|\beta|\leq j} \sup_{(\xi,\lambda)\in \R^n\times \Lambda_j} (1+|\lambda|)^{-d} (1+|\xi|)^{|\beta|-m} \big\| \delta^\alpha \partial_\xi^\beta \rho(\xi;\lambda)\big\|, \quad j\geq 0. 
\end{equation*}
Note that $p_{j}^{(m)}\leq p_{j+1}^{(m)}$. In addition, we let $\chi(\xi)\in \cS(\R^n)$ be such that $\chi(0)=1$. 
For $\epsilon>0$ we set $\chi_\epsilon(\xi)=\chi(\epsilon \xi)$, $\xi\in \R^n$. 

We know by Lemma~\ref{lem:Parameter.standard-symbol-with-parameter-density-lemma} that, for $j=1,2,\ldots $, the family $\chi_\epsilon(\xi) \rho_j(\xi;\lambda)$ converges to $\rho_j(\xi;\lambda)$ in $\stS^{m_{j-1},d}(\R^n\times \Lambda; \cA_\theta)$ as $\epsilon \rightarrow 0^+$ (since $m_{j-1}>m_j$). Thus, we can find $\epsilon_j>0$ such that $p_j^{(m_{j-1})}\big[ (1-\chi_{\epsilon_j})\rho_j\big] \leq 2^{-j}$. We also set $\epsilon_0=1$. Given any $N\geq 0$, for any $\ell \geq 0$, we have 
\begin{equation*}
 \sum_{j\geq \ell+1} p_\ell^{(m_N)}\big[ (1-\chi_{\epsilon_j})\rho_j\big] \leq \sum_{j\geq \ell+1} p_j^{(m_{j-1})}\big[ (1-\chi_{\epsilon_j})\rho_j\big] \leq \sum_{j\geq \ell+1} 2^{-j}<\infty. 
\end{equation*}
This implies that the series $\sum_{j\geq N}(1-\chi_{\epsilon_j}(\xi))\rho_j(\xi;\lambda)$ converges normally with respect to each semi-norm $p_\ell^{(m_N)}$, $\ell\geq N$. As $\stS^{m_{N},d}(\R^n\times \Lambda; \cA_\theta)$ is a Fr\'echet space whose topology is generated by these semi-norms, we deduce that the series $\sum_{j\geq N}(1-\chi_{\epsilon_j}(\xi))\rho_j(\xi;\lambda)$ converges  in  $\stS^{m_{N},d}(\R^n\times \Lambda; \cA_\theta)$ for every $N\geq 0$. Set 
\begin{equation*}
 \rho(\xi;\lambda) = \sum_{j\geq 0}\big(1-\chi_{\epsilon_j}(\xi)\big)\rho_j(\xi;\lambda). 
\end{equation*}
Then $\rho(\xi;\lambda)\in \stS^{m_{0},d}(\R^n\times \Lambda; \cA_\theta)$. Moreover, for all $N\geq 1$, we have
\begin{equation*}
 \rho(\xi;\lambda) -  \sum_{j<N}\rho_j(\xi;\lambda)=  \sum_{j<N}\chi_{\epsilon_j}(\xi)\rho_j(\xi;\lambda) +  \sum_{j\geq N}\big(1-\chi_{\epsilon_j}(\xi)\big)\rho_j(\xi;\lambda). 
\end{equation*}
Here the series $\sum_{j\geq N}(1-\chi_{\epsilon_j}(\xi))\rho_j(\xi;\lambda)$ converges  in  $\stS^{m_{N},d}(\R^n\times \Lambda; \cA_\theta)$  and each map $\chi_{\epsilon_j}(\xi)\rho_j(\xi;\lambda)$ is contained in $\stS^{-\infty,d}(\R^n\times \Lambda; \cA_\theta) \subset \stS^{m_{N},d}(\R^n\times \Lambda; \cA_\theta)$. Therefore, we see that the remainder term $ \rho(\xi;\lambda) -  \sum_{j<N}\rho_j(\xi;\lambda)$ is in $\stS^{m_{N},d}(\R^n\times \Lambda; \cA_\theta)$ for all $N\geq 0$. That is, 
$ \rho(\xi;\lambda) \sim \sum_{j\geq 0}\rho_j(\xi;\lambda)$. The proof is complete. 
\end{proof}

We are now in a position to get a version of Borel's lemma for classical parametric symbols.  

\begin{proposition} \label{prop:Parameter.Borel-for-classical-symbol}
 Given $m\in \R$ and $c>0$, let $\rho_{m-j}(\xi;\lambda)\in S_{m-j}^{d}(\Omega_c(\Theta);\cA_\theta)$, $j\geq 0$. Then there is $\rho(\xi;\lambda)\in S^{m,d}(\Rn\times\Lambda;\cA_\theta)$ such that 
 $\rho(\xi;\lambda)\sim\sum_{j\geq 0}\rho_{m-j}(\xi;\lambda)$. 
\end{proposition}
\begin{proof}
 Let $\chi(\xi)\in C^\infty_c(\R^n)$ be as in Lemma~\ref{lem:Parameter.homogeneous-symbol-estimate}. We know by Lemma~\ref{lem:Parameter.homogeneous-symbol-estimate} that $(1-\chi(\xi))\rho_{m-j}(\xi;\lambda)$ lies in  $\stS^{m+wd_{-}-j,d}(\Omega_c(\Theta);\cA_\theta)$. Therefore, by Lemma~\ref{lem:Parameter.Borel} there is $\rho(\xi;\lambda)$ in $\stS^{m+wd_{-},d}(\R^n\times\Lambda; \cA_\theta)$ such that 
 $\rho(\xi;\lambda) \sim \sum_{j\geq 0}(1-\chi(\xi))\rho_{m-j}(\xi;\lambda)$ in the sense of~(\ref{eq:Parameter.Standard-asymptotic}). By Remark~\ref{rmk:Parameter.classical-symbol-asymptotics-equivalent-conditions} this implies that
  $\rho(\xi;\lambda) \sim \sum_{j\geq 0}\rho_{m-j}(\xi;\lambda)$ in the sense of~(\ref{eq:Parameter.classical-symbol-asymptotics}). In particular, we see that $\rho(\xi;\lambda)\in S^{m,d}(\Rn\times\Lambda;\cA_\theta)$. The proof is complete.
\end{proof}

\section{Parametric Pseudodifferential Operators} \label{sec:parametric-PsiDOs} 
In this section, we introduce our classes of \psidos\ with parameter and derive some of their properties. 

\subsection{Classes of \psidos\ with parameter} 
Let $\rho(\xi;\lambda)\in \stS^{m,d}(\R^n\times \Lambda; \cA_\theta)$, $m,d\in \R$. Given any $\lambda \in \Lambda$, we get a symbol in $\stS^m(\R^n;\cA_\theta)$. We denote by $P_\rho(\lambda)$ the \psido\ associated with this symbol. That is, $P_\rho(\lambda)$ is the continuous operator on $\cA_\theta$ defined by 
\begin{equation*}
P_\rho(\lambda) u = \iint e^{is\cdot\xi}\rho(\xi;\lambda)\alpha_{-s}(u)ds\dbar\xi , \qquad u\in\cA_\theta,
\end{equation*}
where the above integral is meant as an oscillating integral (see Section~\ref{sec:PsiDOs}). 

In what follows we equip $\cL(\cA_\theta)$ with its strong topology (i.e., the topology of uniform convergence on bounded subsets of $\cA_\theta$). We similarly equip  $\cL(\cA_\theta')$ with its strong topology.  By~\cite[Proposition~5.4]{HLP:Part1} and~\cite[Proposition~8.6]{HLP:Part2} we have continuous linear maps,
\begin{equation*}
 \stS^m(\R^n;\cA_\theta) \ni \rho(\xi)\rightarrow P_\rho \in \cL(\cA_\theta), \qquad \stS^m(\R^n;\cA_\theta) \ni \rho(\xi)\rightarrow P_\rho \in \cL(\cA_\theta'). 
\end{equation*}
Combining this Proposition~\ref{prop:Parameter.linear-map-symbols} gives the following result. 

\begin{proposition} \label{prop:PsiDOs-parameter.PsiDO-gives-rise-to-family-of-continuous-operators-on-cAtheta}
 For any $\rho(\xi;\lambda)\in \stS^{m,d}(\R^n\times \Lambda; \cA_\theta)$, $m,d\in \R$, the family $P_\rho(\lambda)$ is contained in $\Hol^d(\Lambda; \cL(\cA_\theta))$ and uniquely extends to a family $P_\rho(\lambda)\in \Hol^d(\Lambda; \cL(\cA_\theta'))$. 
\end{proposition}

\begin{definition}
$\Psi^{m,d}(\cA_\theta;\Lambda)$, $m,d\in \R$, consists of families of operators $P_{\rho}(\lambda):\cA_\theta \rightarrow \cA_\theta$ with $\rho(\xi;\lambda)$ in $S^{m,d}(\Rn\times\Lambda;\cA_\theta)$.
\end{definition}

\begin{example} \label{ex:PsiDOs-parameter.usual-PsiDO-can-be-seen-as-PsiDO-with-parameter}
 It follows from Example~\ref{ex:Parameter.classical-symbol-can-be-seen-as-classical-symbol-with-parameter} that any operator $P\in \Psi^m(\cA_\theta)$ can be regarded as an element of $\Psi^{m,0}(\cA_\theta; \Lambda)$. Combining this with the obvious inclusion $\Psi^{m,0}(\cA_\theta; \Lambda)\subset \Psi^{m,1}(\cA_\theta; \Lambda)$ allows us to regard $P-\lambda$ as an element of  $\Psi^{m,1}(\cA_\theta; \Lambda)$. 
\end{example}

We also define \psidos\ with parameter of order~$-\infty$ as follows. 
\begin{definition}
$\Psi^{-\infty,d}(\cA_\theta;\Lambda)$, $d\in \R$,  consists of families of operators $P_{\rho}(\lambda):\cA_\theta \rightarrow \cA_\theta$ with $\rho(\xi;\lambda)$ in $\stS^{-\infty,d}(\Rn\times\Lambda;\cA_\theta)$.
\end{definition}

\begin{remark} \label{rmk:PsiDOs-parameter.order-minus-infty-to-intersection-inclusion}
We have $\Psi^{-\infty,d}(\cA_\theta;\Lambda)\subset \bigcap_{m\in \R}\Psi^{m,d}(\cA_\theta;\Lambda)$. This inclusion is actually an equality (see Corollary~\ref{cor:PsiDOs-parameter.intersection-of-psidos-of-all-orders}). 
\end{remark}

\subsection{Composition of \psidos\ with parameter} 
Let $\rho_1(\xi;\lambda) \in \stS^{m_1,d_1}(\R^n\times \Lambda; \cA_\theta)$ and $\rho_2(\xi;\lambda)\in \stS^{m_2,d_2}(\R^n\times \Lambda; \cA_\theta)$. For each $\lambda \in \Lambda$, we denote by $\rho_1\sharp \rho_2(\xi;\lambda)$ the symbol given by~(\ref{eq:Composition.symbol-sharp}). By Proposition~\ref{prop:Composition.sharp-continuity-standard-symbol} the composition $P_{\rho_1}(\lambda)P_{\rho_2}(\lambda)$ is the \psido\ with symbol $\rho_1\sharp \rho_2(\xi;\lambda)$. 

\begin{proposition} \label{prop:PsiDOs-parameter.standard-symbol-sharp-product}
 Let $\rho_1(\xi;\lambda)\in \stS^{m_1,d_1}(\R^n\times \Lambda; \cA_\theta)$, $m_1,d_1\in \R$, and $\rho_2(\xi;\lambda)\in \stS^{m_2,d_2}(\R^n\times \Lambda; \cA_\theta)$, $m_2,d_2\in \R$. Then 
  $\rho_1\sharp \rho_2(\xi;\lambda)\in \stS^{m_1+m_2,d_1+d_2}(\R^n\times \Lambda; \cA_\theta)$, and in the sense of~(\ref{eq:Parameter.Standard-asymptotic}) we have
  \begin{equation*}
     \rho_1\sharp \rho_2(\xi;\lambda) \sim   \sum  \frac{1}{\alpha!}\partial_\xi^\alpha \rho_1(\xi;\lambda)\delta^\alpha \rho_2(\xi;\lambda).   
  \end{equation*}
 \end{proposition}
\begin{proof}
 For $N\geq 0$, let $\sharp_N:\stS^{m_1}(\R^n; \cA_\theta) \times \stS^{m_2}(\R^n; \cA_\theta)\rightarrow \stS^{m_1+m_2-N}(\R^n; \cA_\theta)$ be the bilinear map defined by
\begin{equation*}
 \rho_1\sharp_N \rho_2(\xi) = \rho_1\sharp \rho_2(\xi) -  \sum_{|\alpha|<N}  \frac{1}{\alpha!}\partial_\xi^\alpha \rho_1(\xi)\delta^\alpha \rho_2(\xi), \qquad \rho_j(\xi)\in \stS^{m_j}(\R^n;\cA_\theta). 
\end{equation*}
By~\cite[Proposition~7.10]{HLP:Part2} this is a continuous bilinear map. Therefore, by using Proposition~\ref{prop:Parameter.bilinear-map-symbols} we see that if $\rho_1(\xi;\lambda)\in \stS^{m_1,d_1}(\R^n\times \Lambda; \cA_\theta)$ and $\rho_2(\xi;\lambda)\in \stS^{m_2,d_2}(\R^n\times \Lambda; \cA_\theta)$, then 
$\rho_1\sharp_N \rho_2(\xi;\lambda)$ is  in $\stS^{m_1+m_2-N,d_1+d_2}(\R^n\times \Lambda; \cA_\theta)$. Thus, 
\begin{equation*}
  \rho_1\sharp \rho_2(\xi;\lambda) -  \sum_{|\alpha|<N}  \frac{1}{\alpha!}\partial_\xi^\alpha \rho_1(\xi;\lambda)\delta^\alpha \rho_2(\xi;\lambda) \in 
  \stS^{m_1+m_2-N,d_1+d_2}(\R^n\times \Lambda; \cA_\theta) \qquad \forall N\geq 0.  
\end{equation*}
This means that  $\rho_1\sharp \rho_2(\xi;\lambda) \sim   \sum  \frac{1}{\alpha!}\partial_\xi^\alpha \rho_1(\xi;\lambda)\delta^\alpha \rho_2(\xi;\lambda)$ in the sense of~(\ref{eq:Parameter.Standard-asymptotic}).  In particular, we see that $\rho_1\sharp \rho_2(\xi;\lambda)\in \stS^{m_1+m_2,d_1+d_2}(\R^n\times \Lambda; \cA_\theta)$. The proof is complete. 
\end{proof}

In order to deal with the composition of \psidos\ associated with classical parametric symbols we need the following two lemmas. 

\begin{lemma} \label{lem:Parameter.classical-homogeneouspart}
Suppose that $\rho(\xi;\lambda) \sim \sum_{\ell\geq 0}\rho^{(\ell)}(\xi;\lambda)$ in the sense of~(\ref{eq:Parameter.Standard-asymptotic}), where $\rho^{(\ell)}(\xi;\lambda)$ is in 
$S^{m-\ell,d}(\R^n\times \Lambda; \cA_\theta)$ and $\rho^{(\ell)}(\xi;\lambda)\sim \sum_{k\geq 0} \rho^{(\ell)}_{m-\ell-k}(\xi;\lambda)$ with $ \rho^{(\ell)}_{m-\ell-k}(\xi;\lambda)\in S^{d}_{m-\ell-k}(\Omega_c(\Theta); \cA_\theta)$.  Then $\rho(\xi;\lambda)\in S^{m,d}(\R^n\times \Lambda; \cA_\theta)$, and we have 
\begin{equation}
 \rho(\xi;\lambda) \sim \sum_{j\geq 0}  \rho_{m-j}(\xi;\lambda), \quad \text{where}\  \rho_{m-j}(\xi;\lambda):= \sum_{0\leq \ell \leq j}  \rho_{m-j}^{(\ell)}(\xi;\lambda)\in 
 S_{m-j}^d(\Omega_c(\Theta);\Lambda). 
 \label{eq:Parameter.homogeneous-components-expansion}
\end{equation}
\end{lemma}
\begin{proof}
 Let $N\geq 0$ and $J\geq N+wd_{-}$. By assumption, 
\begin{equation*}
 \rho(\xi;\lambda) - \sum_{\ell< J}\rho^{(\ell)}(\xi;\lambda)\in \stS^{m-N,d}(\R^n\times \Lambda; \cA_\theta). 
\end{equation*}
 Let $\chi(\xi)\in C^\infty_c(\R^n)$ be as in Lemma~\ref{lem:Parameter.homogeneous-symbol-estimate}. In view of~(\ref{eq:Parameter.classical-symbol-asymptotics-qualitative}) for $\ell< J$ we have 
 \begin{align*}
\rho^{(\ell)} (\xi;\lambda) &=  \sum_{k<J-\ell}\big(1-\chi(\xi)\big) \rho^{(\ell)}_{m-\ell-k} (\xi;\lambda) \quad \bmod \stS^{(m-\ell)-(N-\ell),d}(\R^n\times \Lambda; \cA_\theta) \\
 & =  \sum_{\ell\leq j<J} \big(1-\chi(\xi)\big)\rho^{(\ell)}_{m-j} (\xi;\lambda) \quad \bmod \stS^{m-N,d}(\R^n\times \Lambda; \cA_\theta). 
\end{align*}
Thus, 
 \begin{align*}
 \rho(\xi;\lambda)   & = \sum_{\ell< J} \sum_{\ell\leq j<J} \big(1-\chi(\xi)\big) \rho^{(\ell)}_{m-j} (\xi;\lambda) \quad \bmod \stS^{m-N,d}(\R^n\times \Lambda; \cA_\theta)\\
 & = \sum_{j< J}\big(1-\chi(\xi)\big)\rho_{m-j} (\xi;\lambda) \quad \bmod \stS^{m-N,d}(\R^n\times \Lambda; \cA_\theta), 
\end{align*}
 where we have set $\rho_{m-j}(\xi;\lambda)= \sum_{\ell=0}^j  \rho_{m-j}^{(\ell)}(\xi;\lambda)$. Here $\rho_{m-j}(\xi;\lambda)\in 
 S_{m-j}^d(\Omega_c(\Theta);\Lambda)$, and so by using Remark~\ref{rmk:Parameter.classical-symbol-asymptotics-equivalent-conditions} we see that  $\rho(\xi;\lambda) \sim \sum_{j\geq 0}  \rho_{m-j}(\xi;\lambda)$ in the sense of~(\ref{eq:Parameter.classical-symbol-asymptotics}). In particular, this shows that $\rho(\xi;\lambda)\in S^{m,d}(\R^n\times \Lambda; \cA_\theta)$. The result is thus proved. 
\end{proof}

\begin{lemma} \label{lem:Parameter.classical-symbol-product}
Let $\rho_1(\xi;\lambda)\in S^{m_1,d_1}(\Rn\times\Lambda;\cA_\theta)$ and $\rho_2(\xi;\lambda)\in S^{m_2,d_2}(\Rn\times\Lambda;\cA_\theta)$, $m_i,d_i\in \R$,  be such that 
$\rho_1(\xi;\lambda)\sim\sum_{j\geq 0}\rho_{1,m_1-j}(\xi;\lambda)$ and $\rho_2(\xi;\lambda)\sim\sum_{j\geq 0}\rho_{2,m_2-j}(\xi;\lambda)$, with
$\rho_{i,m_i-j}(\xi;\lambda)$ in $S_{m_i-j}^{d_i}(\Omega_c(\Theta);\cA_\theta)$, $i=1,2$.  Then $\rho_1(\xi;\lambda)\rho_2(\xi;\lambda)\in S^{m_1+m_2,d_1+d_2}(\Rn\times\Lambda;\cA_\theta)$, and we have 
\begin{gather}
\rho_1(\xi;\lambda)\rho_2(\xi;\lambda)\sim\sum_{j\geq 0}(\rho_1\rho_2)_{m_1+m_2-j}(\xi;\lambda), \quad \text{where}\\ 
 (\rho_1\rho_2)_{m_1+m_2-j}(\xi;\lambda) := \sum_{k+l=j}\rho_{1,m_1-k}(\xi;\lambda)\rho_{2,m_2-l}(\xi;\lambda)\in S^{d_1+d_2}_{m_1+m_2-j}(\Omega_c(\Theta);\cA_\theta).  
 \label{eq:PsiDOs-parameter.pointwise-product-classical2}
\end{gather}
\end{lemma}
\begin{proof}
For $i=1,2$ let $\rho^{(i)}(\xi;\lambda)\in \stS^{m_i,d_i}(\R^n\times \Lambda; \cA_\theta)$ be such that $\rho^{(i)}(\xi;\lambda) \sim \sum_{j\geq 0} \rho^{(i)}_{m_i-j}(\xi;\lambda)$ with 
$\rho^{(i)}_{m_i-j}(\xi;\lambda)\in \stS^{m_i-j,d_i}(\R^n\times \Lambda; \cA_\theta)$. Then~(\ref{eq:Parameter.standard-symbol-with-parameter-product}) implies that, for all $N\geq 0$ and $l<N$, we have 
\begin{gather*}
 \rho^{(1)}(\xi;\lambda)\rho^{(2)}(\xi;\lambda)= \sum_{l<N} \rho^{(1)}(\xi;\lambda)\rho^{(2)}_{m_2-l}(\xi;\lambda) \quad \bmod  \stS^{m_1+m_2-N,d_1+d_2}(\R^n\times \Lambda; \cA_\theta),\\
 \rho^{(1)}(\xi;\lambda)\rho^{(2)}_{m_2-l}(\xi;\lambda) = \sum_{k<N-l} \rho^{(1)}_{m_1-k}(\xi;\lambda)\rho^{(2)}_{m_2-l}(\xi;\lambda) \quad \bmod  \stS^{m_1+m_2-N,d_1+d_2}(\R^n\times \Lambda; \cA_\theta).
\end{gather*}
Thus, 
\begin{equation*}
  \rho^{(1)}(\xi;\lambda)\rho^{(2)}(\xi;\lambda)=\sum_{k+l<N} \rho^{(1)}_{m_1-k}(\xi;\lambda)\rho^{(2)}_{m_2-l}(\xi;\lambda) \quad \bmod  \stS^{m_1+m_2-N,d_1+d_2}(\R^n\times \Lambda; \cA_\theta).
\end{equation*}
This shows that, in the sense of~(\ref{eq:Parameter.Standard-asymptotic}), we have 
\begin{equation} \label{eq:Parameter.Standard-asymptotic-product}
 \rho^{(1)}(\xi;\lambda)\rho^{(2)}(\xi;\lambda) \sim \sum_{j\geq 0} \sum_{k+l=j}\rho^{(1)}_{m_1-k}(\xi;\lambda)\rho^{(2)}_{m_2-l}(\xi;\lambda).  
\end{equation}

Bearing this in mind, let $\chi(\xi)\in C_c^\infty(\Rn)$ be such that 
$\chi(\xi) = 1$ for $|\xi|\leq (c^{-1}R)^{\frac{1}{w}}$, where $R>0$ is such that $\Lambda \subset \Theta \cup D(0,R)$. Set $\widetilde{\chi}(\xi) = 1-(1-\chi(\xi))^2$. Then $\widetilde{\chi}(\xi)\in  C_c^\infty(\Rn)$  and 
$\widetilde{\chi}(\xi) = 1$ for $|\xi|\leq (c^{-1}R)^{\frac{1}{w}}$, i.e., $\widetilde{\chi}(\xi)$ is as in Lemma~\ref{lem:Parameter.homogeneous-symbol-estimate}. We know by Remark~\ref{rmk:Parameter.classical-symbol-asymptotics-equivalent-conditions} that $\rho_1(\xi;\lambda)\sim\sum_{j\geq 0}(1-\chi(\xi))\rho_{1,m_1-j}(\xi;\lambda)$ and $\rho_2(\xi;\lambda)\sim\sum_{j\geq 0}(1-\chi(\xi))\rho_{2,m_2-j}(\xi;\lambda)$ in the sense of~(\ref{eq:Parameter.Standard-asymptotic}). 
Therefore, in view of~(\ref{eq:Parameter.Standard-asymptotic-product}), in the sense of~(\ref{eq:Parameter.Standard-asymptotic}) once again, we have
\begin{align*}
 \rho_1(\xi;\lambda)\rho_2(\xi;\lambda) &\sim \sum_{j\geq 0} \sum_{k+l=j} (1-\chi(\xi))^2  \rho_{1,m_1-k}(\xi;\lambda)\rho_{2,m_2-l}(\xi;\lambda)\\  
 & \sim \sum_{j\geq 0}\big(1-\widetilde{\chi}(\xi)\big)  (\rho_1\rho_2)_{m_1+m_2-j}(\xi;\lambda). 
\end{align*}
where $ (\rho_1\rho_2)_{m+m_2-j}(\xi;\lambda)$ is given by~(\ref{eq:PsiDOs-parameter.pointwise-product-classical2}). As $ (\rho_1\rho_2)_{m+m_2-j}(\xi;\lambda)\in S^{d_1+d_2}_{m_1+m_2-j}(\Omega_c(\Theta);\cA_\theta)$ and $\widetilde{\chi}(\xi)$ is as in Lemma~\ref{lem:Parameter.homogeneous-symbol-estimate}, it follows from Remark~\ref{rmk:Parameter.classical-symbol-asymptotics-equivalent-conditions} that in the sense of~(\ref{eq:Parameter.classical-symbol-asymptotics}) we have
$\rho_1(\xi;\lambda)\rho_2(\xi;\lambda)\sim\sum_{j\geq 0}(\rho_1\rho_2)_{m_1+m_2-j}(\xi;\lambda)$. This immediately implies that 
$\rho_1(\xi;\lambda)\rho_2(\xi;\lambda)$ is in  $S^{m_1+m_2,d_1+d_2}(\Rn\times\Lambda;\cA_\theta)$. The proof is complete. 
\end{proof}

We are now in a position to prove the following result. 

\begin{proposition} \label{prop:Parameter.composition-PsiDOs}
Let $P_1(\lambda)\in\Psi^{m_1,d_1}(\cA_\theta;\Lambda)$ have symbol $\rho_1(\xi;\lambda)\sim\sum \rho_{1,m_1-j}(\xi;\lambda)$, and 
let $P_2(\lambda)\in\Psi^{m_2,d_2}(\cA_\theta;\Lambda)$ have symbol $\rho_2(\xi;\lambda)\sim\sum \rho_{2,m_2-j}(\xi;\lambda)$. 
\begin{enumerate}
 \item $\rho_1\sharp \rho_2(\xi;\lambda)\in S^{m_1+m_2,d_1+d_2}(\R^n\times \Lambda; \cA_\theta)$ with 
 $\rho_1\sharp\rho_2(\xi;\lambda)\sim\sum (\rho_1\sharp\rho_2)_{m_1+m_2-j}(\xi;\lambda)$, where
\begin{equation} \label{eq:PsiDOs-parameter.composition-symbol-homogeneous-parts}
(\rho_1\sharp\rho_2)_{m_1+m_2-j}(\xi;\lambda) = \sum_{k+l+|\alpha|=j}\frac{1}{\alpha!}\partial_\xi^\alpha\rho_{1,m_1-k}(\xi;\lambda)\delta^\alpha\rho_{2,m_2-l}(\xi;\lambda),\qquad j\geq 0 .
\end{equation}

\item The composition $P_1(\lambda)P_2(\lambda)=P_{\rho_1\sharp \rho_2}(\lambda)$ is in $\Psi^{m_1+m_2,d_1+d_2}(\cA_\theta;\Lambda)$. 
\end{enumerate}
\end{proposition}
\begin{proof}
As the 2nd part is an immediate consequence of the first part and Proposition~\ref{prop:Composition.composition-PsiDOs}, we only have to prove the first part. Given any $\alpha \in \N_0^n$, it follows from Remark~\ref{rem:Parameter.classical-symbol-partial-derivative} and Lemma~\ref{lem:Parameter.classical-symbol-product} that, in the sense of~(\ref{eq:Parameter.classical-symbol-asymptotics}), we have
\begin{equation*}
\partial_\xi^\alpha\rho_{1}(\xi;\lambda)\delta^\alpha\rho_{2}(\xi;\lambda)\sim 
\sum \partial_\xi^\alpha\rho_{1,m_1-k}(\xi;\lambda)\delta^\alpha\rho_{2,m_2-l}(\xi;\lambda), 
\end{equation*}
 where $\partial_\xi^\alpha\rho_{1,m_1-k}(\xi;\lambda)\delta^\alpha\rho_{2,m_2-l}(\xi;\lambda)\in S^{d_1+d_2}_{m_1+m_2-|\alpha|-k-l}(\Omega_c(\Theta);\cA_\theta)$ for some $c>0$ independent of $\alpha$, $k$ and $l$. Combining this with Lemma~\ref{lem:Parameter.classical-homogeneouspart} and Proposition~\ref{prop:PsiDOs-parameter.standard-symbol-sharp-product} then shows that, in the sense of~(\ref{eq:Parameter.classical-symbol-asymptotics}), we have
\begin{equation*}
 \rho_1\sharp\rho_2(\xi;\lambda)\sim\sum (\rho_1\sharp\rho_2)_{m_1+m_2-j}(\xi;\lambda), 
\end{equation*}
where $ (\rho_1\sharp\rho_2)_{m_1+m_2-j}(\xi;\lambda)\in S^{d_1+d_2}_{m_1+m_2-j}(\Omega_c(\Theta);\cA_\theta)$ is given by~(\ref{eq:PsiDOs-parameter.composition-symbol-homogeneous-parts}). In particular, we see that $\rho_1\sharp \rho_2(\xi;\lambda)$ is contained in $S^{m_1+m_2,d_1+d_2}(\R^n\times \Lambda; \cA_\theta)$. The proof is complete. 
\end{proof}

\subsection{Sobolev space mapping properties} 
Given any $m,s\in \R$, by~\cite[Proposition~10.4]{HLP:Part2} we have a continuous linear map $\stS^m(\R^n;\cA_\theta) \ni \rho(\xi) \longrightarrow P_\rho \in \cL(\cH_\theta^{(s+m)},\cH_\theta^{(s)})$. By combining this with Proposition~\ref{prop:Parameter.linear-map-symbols} and the version of Sobolev's embedding theorem provided by Proposition~\ref{prop:Sobolev.Sobolev-embedding} we obtain the following result. 

\begin{proposition} \label{prop:PsiDOs-parameter.Sobolev-mapping-properties}
 Let $\rho(\xi;\lambda)\in \stS^{m,d}(\R^n\times \Lambda; \cA_\theta)$, $m,d\in \R$. 
 \begin{enumerate}
 \item The family $P_\rho(\lambda)$ is contained  in $\Hol^d(\Lambda; \cL(\cH_\theta^{(s+m)}, \cH_\theta^{(s)}))$ for every $s\in \R$.
 
 \item If $m\leq 0$, then $P_\rho(\lambda)\in \Hol^d(\Lambda; \cL(\cH_\theta))$. 
\end{enumerate}
\end{proposition}

 In the case of classical \psidos\ with parameter we get the following result.  

\begin{proposition} \label{prop:PsiDOs-parameter.classical-PsiDO-Sobolev-mapping-properties}
 Let $P(\lambda)\in \Psi^{m,d}(\cA_\theta; \Lambda)$, $m,d\in \R$. 
\begin{enumerate}
 \item If $d\geq 0$, then $P(\lambda) \in \Hol^d(\Lambda; \cL(\cH_\theta^{(s+m)}, \cH_\theta^{(s)}))$ for every $s\in \R$. 
 
 \item If $d<0$, then, for every $s\in \R$, we have 
          \begin{equation*}
             P(\lambda) \in \bigcap_{d\leq d'\leq 0}\Hol^{d'}\big(\Lambda; \cL\big(\cH_\theta^{(s+m+w|d'|)}, \cH_\theta^{(s)}\big)\big).  
          \end{equation*}
          
 \item If $m\leq 0$, then $P(\lambda)\in\Hol^{\bar{d}}(\Lambda; \cL(\cH_\theta))$ with $\bar{d}:=\max(d,mw^{-1})$. 
 \end{enumerate}
\end{proposition}
\begin{proof}
The first two parts follow by combining the first part of Proposition~\ref{prop:PsiDOs-parameter.Sobolev-mapping-properties} with Proposition~\ref{symbols:inclusion-classical-standard}.  It remains to prove the last part. 

Suppose that $m\leq 0$. If $d\geq 0$, then by using the last two parts of Proposition~\ref{prop:PsiDOs-parameter.Sobolev-mapping-properties} and Proposition~\ref{symbols:inclusion-classical-standard} we see that $P(\lambda)$ is in 
 $\Hol^{d}(\Lambda; \cL(\cH_\theta))$ and is even in $\Hol^d(\Lambda; \cK)$ when $m<0$. 
 
 Assume now that $d<0$. In the same way as above, given any $d'\in[d,0]$,  we see that $P(\lambda)$ is in $\Hol^{d'}(\Lambda; \cL(\cH_\theta))$ when $m+w|d'|\leq 0$ and is in $\Hol^{d'}(\Lambda; \cK)$ when $m+w|d'|<0$. In particular, we see that $P(\lambda)$ is in $\Hol^{\bar{d}}(\Lambda; \cL(\cH_\theta))$ with $\bar{d}:=\max(d,mw^{-1})$. The proof is complete. 
 \end{proof}

\subsection{Schatten class properties} 
We refer to \S\S\ref{subsec:Schatten} for the definitions of the Schatten classes $\cL^{p}$ and weak Schatten classes $\cL^{(p,\infty)}$ on $\cH_\theta$ with $p\geq 1$. Recall they are Banach ideals of $\cL(\cH_\theta)$.

\begin{lemma}[{\cite[Propositions~13.8 \& 13.13, Corollary~13.11]{HLP:Part2}}] \label{lem:PsiDos-parameter.standard-PsiDO-with-no-parameter-Schatten} The following holds.
\begin{enumerate}
\item If $-n\leq m<0$ and we set $p=n|m|^{-1}$, then $\rho \rightarrow P_\rho$ induces continuous linear maps from $\stS^{m}(\R^n;\cA_\theta)$ to 
$\cL^{(p,\infty)}$ and $\cL^{q}$, $q>p$. 

\item If $m<-n$, then $\rho \rightarrow P_\rho$ induces a continuous linear map from $\stS^{m}(\R^n;\cA_\theta)$ to $\cL^{1}$. 
\end{enumerate}
\end{lemma}

Combining this with Proposition~\ref{prop:Parameter.linear-map-symbols} gives the following result. 

\begin{proposition}\label{prop:PsiDOs-parameter.Schatten-standard} 
 Let $\rho(\xi;\lambda)\in \stS^{m,d}(\R^n\times \Lambda; \cA_\theta)$ with $d\in \R$ and $m<0$.  Set  $p=n|m|^{-1}$. 
 \begin{enumerate}
 \item If $-n\leq m<0$, then $P_\rho(\lambda)$ is contained in $\Hol^d(\Lambda; \cL^{(p,\infty)})$ and $\Hol^d(\Lambda; \cL^{q})$ for all $q>p$. 
 
 \item If $m<-n$, then $P_\rho(\lambda) \in \Hol^d(\Lambda; \cL^{1})$. 
\end{enumerate}
\end{proposition}

Let us now specialize the above result to classical \psidos\ with parameter. 

\begin{proposition} \label{prop:PsiDOs-parameter.Schatten-classical}
 Let $P(\lambda)\in \Psi^{m,d}(\cA_\theta; \Lambda)$ with $d\in \R$ and $-n\leq m<0$. Set $p=n|m|^{-1}$. In addition, for $q\geq p$ set 
 $d(q)=(m+nq^{-1})w^{-1}$.  
 \begin{enumerate}
\item If $d\geq 0$, then $P(\lambda)\in \Hol^{d}(\Lambda; \cL^{(p,\infty)})$.

\item If $mw^{-1}<d<0$, then
      \begin{equation}
 P(\lambda)\in \bigcap_{p\leq q \leq \bar{p}}  \Hol^{d(q)}\big(\Lambda; \cL^{(q,\infty)}\big), \qquad \text{where}\ \bar{p}:=n(|m|+wd)^{-1}.
 \label{eq:PsiDOs-parameter.classical1}
        \end{equation}
\item If $d\leq  mw^{-1}$, then 
      \begin{equation}
 P(\lambda)\in \bigcap_{q\geq p}  \Hol^{d(q)}\big(\Lambda; \cL^{(q,\infty)}\big). 
 \label{eq:PsiDOs-parameter.classical2}
        \end{equation}
\end{enumerate}
\end{proposition}
\begin{proof}
 Let $\rho(\xi;\lambda)\in S^{m,d}(\R^n\times \Lambda; \cA_\theta)$ be such that $P(\lambda)=P_\rho(\lambda)$. Suppose that $d\geq 0$. 
 Proposition~\ref{symbols:inclusion-classical-standard} ensures us that $\rho(\xi;\lambda)\in \stS^{m,d}(\R^n\times \Lambda; \cA_\theta)$. As $-n\leq m<0$, it then follows from the 1st part of Proposition~\ref{prop:PsiDOs-parameter.Schatten-standard}  that $P(\lambda)\in \Hol^{d}(\Lambda; \cL^{(p,\infty)})$. 
 
Suppose that $d<0$, and let $q\geq p$. As $d(q)=(m+nq^{-1})w^{-1}$ and $p=-nm^{-1}$ we have
\begin{equation} \label{eq:PsiDOs-parameter.d(p)-upper-and-lower-bounds}
 mw^{-1} <d(q) \leq \big(m+np^{-1}\big)w^{-1}=0. 
\end{equation}
Thus, if $d\leq mw^{-1}$, then $d(q)\in [d,0]$ for all $q\geq p$. If $d>mw^{-1}$, then $d\leq d(q)$ if and only if $q\leq \bar{p}$, with $\bar{p}:=n(|m|+wd)^{-1}$. 

From now on we assume that either $d\leq mw^{-1}$ and $q\geq p$, or $d>mw^{-1}$ and $p\leq q \leq \bar{p}$. This ensures that $d(q)\in [d,0]$, and so by Proposition~\ref{symbols:inclusion-classical-standard} this implies that $\rho(\xi;\lambda)$ is contained in $\stS^{m-wd(q),d(q)}(\R^n\times \Lambda;\cA_\theta)$. 
Note also that as $m\geq -n$, we have $m-wd(q)\geq m\geq -n$. 
Moreover, as by~(\ref{eq:PsiDOs-parameter.d(p)-upper-and-lower-bounds}) we have $d(q)>mw^{-1}$, we also have $m-wd(q)<0$. 
It then follows from the 1st part of Proposition~\ref{prop:PsiDOs-parameter.Schatten-standard}  that $P(\lambda)\in \Hol^{d(q)}(\Lambda; \cL^{(p(q),\infty)})$, where $p(q):=n(wd(q)-m)^{-1}$. As $d(q)=(m+nq^{-1})w^{-1}$ it is immediate to check that $p(q)=q$. Therefore, we see that $P(\lambda)$ is contained in $\Hol^{d(q)}(\Lambda; \cL^{(q,\infty)})$. 
This proves~(\ref{eq:PsiDOs-parameter.classical1}) and~(\ref{eq:PsiDOs-parameter.classical2}). The proof is complete. 
\end{proof}

For the trace-class we have the following result. 
\begin{proposition} \label{prop:PsiDOs-parameter.trace-class}
Let $P(\lambda)\in \Psi^{m,d}(\cA_\theta; \Lambda)$ with $d\in \R$ and $m<-n$. 
\begin{enumerate}
 \item If $d>(m+n)w^{-1}$, then $P(\lambda)\in \Hol^{d}(\Lambda; \cL^{1})$ and, for any pseudo-cone $\Lambda' \subsubset \Lambda$, there is $C_{\Lambda'}>0$ such that 
          \begin{equation}
             \Big| \Tra\big[P(\lambda)\big] \Big|\leq  C_{\Lambda'} (1+|\lambda|)^{d} \qquad \forall \lambda \in \Lambda'.  
             \label{eq:Resolvent.trace-estimate-P1} 
         \end{equation}

  \item If $d\leq (m+n)w^{-1}<d'\leq 0$, then $P(\lambda)\in \Hol^{d'}(\Lambda; \cL^{1})$  and, for any pseudo-cone $\Lambda' \subsubset \Lambda$, there is $C_{\Lambda'd'}>0$ such that 
          \begin{equation}
             \Big| \Tra\big[P(\lambda)\big] \Big|\leq  C_{\Lambda' d'} (1+|\lambda|)^{d'} \qquad \forall \lambda \in \Lambda'.  
             \label{eq:Resolvent.trace-estimate-P2}          
             \end{equation}
\end{enumerate}
 \end{proposition}
 \begin{proof}
Let $\rho(\xi;\lambda)\in S^{m,d}(\R^n\times \Lambda; \cA_\theta)$ be such that $P(\lambda)=P_\rho(\lambda)$. We know by Proposition~\ref{symbols:inclusion-classical-standard} that $\rho(\xi;\lambda)\in \stS^{m+wd_-,d}(\R^n\times \Lambda; \cA_\theta)$, where $d_{-}=\max(-d,0)$. Suppose that $d>(m+n)w^{-1}$. If $d\geq 0$, then $m+wd_{-}=m<-n$, and if $d< 0$, then we also have $m+wd_{-}=m-wd<-n$. It then follows from the 2nd part of Proposition~\ref{prop:PsiDOs-parameter.Schatten-standard}  that $P(\lambda)\in \Hol^{d}(\Lambda; \cL^{1})$. This means that, for every pseudo-cone $\Lambda'\subsubset \Lambda$, there is $C_{\Lambda'}>0$ such that 
         \begin{equation*}
            \norm{P(\lambda)}_{\cL^1} \leq C_{\Lambda'} (1+|\lambda|)^{d} \qquad \forall \lambda \in \Lambda'. 
         \end{equation*}
Here $\|T\|_{\cL^1}=\Tra(|T|)$, $T\in \cL^1$. Combining this with the inequality $| \Tra(T) |\leq   \norm{T}_{\cL^1}$ gives~(\ref{eq:Resolvent.trace-estimate-P1}).
 
Suppose now that $d\leq (m+n)w^{-1}$, and let $d'\in ((m+n)w^{-1},0]$. This implies that $d<d'\leq 0$, and so by using Proposition~\ref{symbols:inclusion-classical-standard} we see that 
$\rho(\xi;\lambda)\in \stS^{m-wd',d'}(\R^n\times \Lambda; \cA_\theta)$. As $d'> (m+n)w^{-1}$, we have $m-wd'<m-w(m+n)w^{-1}=-n$. Therefore, by using the 2nd part of Proposition~\ref{prop:PsiDOs-parameter.Schatten-standard} once again we see that $P(\lambda)\in \Hol^{d'}(\Lambda; \cL^{1})$. In the same way as above this gives~(\ref{eq:Resolvent.trace-estimate-P2}). The proof is complete. 
\end{proof}

\subsection{Toroidal \psidos\ with parameter} 
We know by Proposition~\ref{prop:toroidal.standard-and-toroidal-psidos-agree} that the classes of standard  and toroidal \psidos\ on $\cA_\theta$ agree. We shall now explain how to obtain analogous results for \psidos\ with parameter. We shall keep on using the notation of \S\S\ref{subsec:toroidal}. 

As mentioned in~\cite{HLP:Part1} each toroidal symbol space $\stS^m(\R^n;\cA_\theta)$, $m\in \R$, is a Fr\'echet space with respect to the 
 locally convex topology generated by the semi-norms, 
\begin{equation*}
 (\rho_k)_{k\in \Z^n} \longrightarrow \sup_{|\alpha|+|\beta|\leq N} \sup_{k\in \Z^n} (1+|k|)^{-m+|\beta|}\big\| \delta^\alpha \Delta^\beta \rho_k  \big\|, \qquad N\geq 0.  
\end{equation*}
Likewise,  $\cS(\Z^n; \cA_\theta)$ is a Fr\'echet space with respect to the topology generated by the semi-norms, 
\begin{equation*}
  (\rho_k)_{k\in \Z^n} \longrightarrow \sup_{k\in \Z^n} (1+|k|)^{-N}\big\| \delta^\alpha \rho_k\big\|, \qquad N\geq 0, \quad \alpha \in \N_0^n.
\end{equation*}

\begin{definition}
 $\stS^{m,d}(\Z^n\times \Lambda; \cA_\theta)$, $m,d\in \R$, consists of sequences $(\rho_k(\lambda))_{k\in \Z^n}\subset \Hol(\Lambda)$ such that, for all pseudo-cones $\Lambda' \subsubset \Lambda$ and multi-orders $\alpha$ and $\beta$, there is $C_{\Lambda'\alpha \beta}>0$ such that
\begin{equation*}
 \big\| \delta^\alpha \Delta^\beta \rho_k(\lambda)\big\| \leq C_{\Lambda'\alpha\beta} (1+|\lambda|)^d (1+|k|)^{m-|\beta|} \qquad \forall (k,\lambda)\in \Z^n\times \Lambda'.  
\end{equation*}
\end{definition}

\begin{remark} \label{rmk:PsiDOs-parameter.toroidal-symbols-with-parameter-identification}
In the same way as in Remark~\ref{rmk:Parameter.standard-symbols-with-parameter-identification} the space $\stS^{m,d}(\Z^n\times \Lambda; \cA_\theta)$ is naturally identified with $\Hol^d(\Lambda; \stS^{m}(\Z^n; \cA_\theta))$. 
\end{remark}

\begin{definition}
 $\stS^{-\infty,d}(\Z^n\times \Lambda; \cA_\theta)$, $d\in \R$, consists of sequences $(\rho_k(\lambda))_{k\in \Z^n}\subset \Hol(\Lambda)$ such that, given any $N\geq 0$, for all pseudo-cones $\Lambda' \subsubset \Lambda$ and multi-orders $\alpha$, $\beta$, there is $C_{\Lambda'N\alpha \beta}>0$ such that
\begin{equation*}
 \big\| \delta^\alpha \Delta^\beta \rho_k(\lambda)\big\| \leq C_{\Lambda'N\alpha\beta} (1+|\lambda|)^d (1+|k|)^{-N} \qquad \forall (k,\lambda)\in \Z^n\times \Lambda'. 
\end{equation*}
\end{definition}

\begin{remark}
$\stS^{-\infty,d}(\Z^n\times \Lambda; \cA_\theta) = \bigcap_{m\in \R} \stS^{m,d}(\Z^n\times \Lambda; \cA_\theta)$. 
 \end{remark}

\begin{remark}
In the same way as in Remark~\ref{rmk:Parameter.standard-symbols-with-parameter-identification} and Remark~\ref{rmk:PsiDOs-parameter.toroidal-symbols-with-parameter-identification} the space $\stS^{-\infty,d}(\Z^n\times \Lambda; \cA_\theta)$ is naturally identified with $\Hol^d(\Lambda; \cS(\Z^n; \cA_\theta))$. 
\end{remark}

In what follows, given $m\in \R\cup\{-\infty\}$, by a \emph{toroidal \psido\ with parameter} of order $m$ we shall mean a family of operators $P(\lambda):\cA_\theta \rightarrow \cA_\theta$ parametrized by $\lambda \in \Lambda$ for which there is a toroidal symbol $(\rho_k(\lambda))_{k\in \Z^n}\in \stS^{m,d}(\Z^n\times \Lambda; \cA_\theta)$,  such that
\begin{equation*}
 P(\lambda)u= \sum_{k\in \Z^n} u_k \rho_k(\lambda) U^k \qquad \text{for all}\ u=\sum_{k\in \Z^n}u_k U^k\in \cA_\theta. 
\end{equation*}

We will need the following versions of Proposition~\ref{prop:toroidal.restriction-of-standard-symbol-is-toroidal-symbol} and Proposition~\ref{prop:toroidal.extension-of-toroidal-symbol-is-standard-symbol}. 

\begin{lemma} \label{lem:PsiDOs-parameter-restriction-extension-continuity}
 The following linear maps are continuous: 
 \begin{enumerate}
 \item[(i)] The restriction maps $\stS^m(\R^n;\cA_\theta) \ni \rho(\xi) \rightarrow (\rho(k))_{k\in \Z^n} \in \stS^m(\Z^n;\cA_\theta)$, $m\in \R$. 
 
 \item[(ii)] The extension maps $\stS^m(\Z^n;\cA_\theta) \ni (\rho_k)_{k\in \Z^n}\rightarrow \tilde{\rho}(\xi) \in \stS^m(\R^n;\cA_\theta)$, $m\in \R$, and 
 $\cS(\Z^n;\cA_\theta) \ni (\rho_k)_{k\in \Z^n}\rightarrow \tilde{\rho}(\xi) \in \cS(\R^n;\cA_\theta)$ given by~(\ref{eq:toroidal.trho}). 
\end{enumerate}
\end{lemma}
\begin{proof}
 Recall that $\stS^m(\R^n;\cA_\theta)$ and  $\stS^m(\Z^n;\cA_\theta)$ are Fr\'echet spaces. It is straightforward to check that, for every $m\in \R$, the graph of the restriction map is a closed subspace of $\stS^m(\R^n;\cA_\theta)\times \stS^m(\Z^n;\cA_\theta)$. It then follows from the closed graph theorem we get a continuous linear map from $\stS^m(\R^n;\cA_\theta)$ and  $\stS^m(\Z^n;\cA_\theta)$. This proves (i). For a proof of~(ii) see~\cite[Remark~6.23]{HLP:Part1} and~\cite[Remark~6.26]{HLP:Part1}.  The proof is complete. 
\end{proof}

Combining the first part of Lemma~\ref{lem:PsiDOs-parameter-restriction-extension-continuity} with Proposition~\ref{prop:Parameter.linear-map-symbols} gives the following result.  

\begin{proposition} \label{prop:PsiDOs-parameter.restriction-of-standard-symbol-is-toroidal-symbol}
Let $\rho(\xi;\lambda) \in \stS^{m,d}(\R^n\times \Lambda; \cA_\theta)$, $m,d\in \R$. Then the restriction $(\rho(k;\lambda))_{k \in \Z^n}$ is contained in $\stS^{m,d}(\Z^n\times \Lambda; \cA_\theta)$.
\end{proposition}

\begin{remark}
 It is immediate from the definitions of $\stS^{-\infty,d}(\R^n\times \Lambda; \cA_\theta)$ and $\stS^{-\infty,d}(\Z^n\times \Lambda; \cA_\theta)$ that, if 
 $\rho(\xi;\lambda)\in \stS^{-\infty,d}(\R^n\times \Lambda; \cA_\theta)$, then $(\rho(k;\lambda))_{k \in \Z^n}\in \stS^{-\infty,d}(\Z^n\times \Lambda; \cA_\theta)$. 
\end{remark}

Let $(\rho_k(\lambda))_{k\in \Z^n}\in \stS^{m,d}(\Z^n\times \Lambda; \cA_\theta)$. Given any $\lambda \in \Lambda$ we get a toroidal symbol 
$(\rho_k(\lambda))_{k\in \Z^n}$ in $\stS^{m}(\Z^n; \cA_\theta)$. We denote by $\tilde{\rho}(\xi;\lambda)\in \stS^m(\R^n; \cA_\theta)$ its extension. By using the 
2nd part of Lemma~\ref{lem:PsiDOs-parameter-restriction-extension-continuity} and Proposition~\ref{prop:Parameter.linear-map} we get the following result. 

\begin{proposition} \label{prop:PsiDOs-parameter.extension-of-toroidal-symbol-is-standard-symbol}
Let $(\rho_k(\lambda))_{k\in \Z^n}\in \stS^{m,d}(\Z^n\times \Lambda; \cA_\theta)$, $m\in \R\cup \{-\infty\}$, $d\in \R$. Then $\tilde{\rho}(\xi;\lambda)$ is contained in $\stS^{m,d}(\R^n\times \Lambda; \cA_\theta)$. 
\end{proposition}

We are now in a position to prove the following result. 

\begin{proposition}\label{prop:Parametric-PsiDOs.equality-toroidal}
Given any $m\in \R\cup\{-\infty\}$, the classes of toroidal and standard \psidos\ with parameter of order $m$ agree. 
\end{proposition}
\begin{proof}
 Let $\rho(\xi;\lambda)\in \stS^{m,d}(\R^n\times \Lambda; \cA_\theta)$. Given any $\lambda \in \Lambda$, it follows from Proposition~\ref{prop:toroidal.restriction-of-standard-symbol-is-toroidal-symbol} that $P_\rho(\lambda)$ is the toroidal \psido\ associated with the restriction $(\rho(k;\lambda))_{k\in \Z^n}$. As $(\rho(k;\lambda))_{k\in \Z^n}\in \stS^m(\Z^n\times \Lambda; \cA_\theta)$ by Proposition~\ref{prop:PsiDOs-parameter.restriction-of-standard-symbol-is-toroidal-symbol} we see that $P_\rho(\lambda)$ is a toroidal \psido\ with parameter of order~$m$. 
 
 Conversely, let $(\rho_k(\lambda))_{k\in \Z^n}\in  \stS^m(\Z^n\times \Lambda; \cA_\theta)$ and denote by $P(\lambda)$ the corresponding toroidal \psido\ with parameter. It follows from Proposition~\ref{prop:toroidal.extension-of-toroidal-symbol-is-standard-symbol} that $P(\lambda)=P_{\tilde{\rho}}(\lambda)$. As we know from Proposition~\ref{prop:PsiDOs-parameter.extension-of-toroidal-symbol-is-standard-symbol} that the extension $\tilde{\rho}(\xi;\lambda)$ is contained in 
 $ \stS^m(\R^n\times \Lambda; \cA_\theta)$, it then follows that $P(\lambda)$ is a standard \psido\ with parameter of order~$m$. This completes the proof. 
\end{proof}

Given any $\rho(\xi;\lambda)\in \stS^{m,d}(\R^n\times \Lambda; \cA_\theta)$ and  $\lambda \in \Lambda$, we have a symbol 
$\rho(\xi;\lambda)\in \stS^{m}(\R^n; \cA_\theta)$. 
We then denote by $\tilde{\rho}(\xi;\lambda)$ the extension~(\ref{eq:toroidal.trho}) of the toroidal symbol $(\rho(k;\lambda))_{k\in \Z^n}$.  

\begin{lemma} \label{lem:PsiDOs-parameter.usual-standard-symbol-restriction-extension-composition-is-continuous}
 The maps $\rho(\xi)\rightarrow \tilde{\rho}(\xi)$ and $\rho(\xi) \rightarrow  \tilde{\rho}(\xi)-\rho(\xi)$ are continuous linear maps from $\stS^m(\R^n; \cA_\theta)$ to $\stS^m(\R^n; \cA_\theta)$ and $\cS(\R^n; \cA_\theta)$, respectively. 
\end{lemma}
\begin{proof}
 It follows from Lemma~\ref{lem:PsiDOs-parameter-restriction-extension-continuity} that $\rho(\xi)\rightarrow \tilde{\rho}(\xi)$ is a continuous linear map from $\stS^m(\R^n; \cA_\theta)$ to itself, since this is the composition of the restriction and extension maps. This also implies that the linear map $\rho(\xi) \rightarrow  \tilde{\rho}(\xi)-\rho(\xi)$ is continuous from $\stS^m(\R^n; \cA_\theta)$ to itself, and so its graph is a closed subspace of $\stS^m(\R^n; \cA_\theta)\times \stS^m(\R^n; \cA_\theta)$. As this graph is contained in 
 $\stS^m(\R^n; \cA_\theta)\times \cS(\R^n; \cA_\theta)$ and the inclusion of $\cS(\R^n; \cA_\theta)$ into  $\stS^m(\R^n; \cA_\theta)$ is continuous, this graph is a closed subspace of $\stS^m(\R^n; \cA_\theta)\times \cS(\R^n; \cA_\theta)$ as well. The closed graph theorem then ensures us that we have a continuous linear map 
 from $\stS^m(\R^n; \cA_\theta)$ to $\cS(\R^n; \cA_\theta)$. The proof is complete. 
\end{proof}

This leads us to the following parameter version of Proposition~\ref{prop:toroidal.restriction-of-standard-symbol-is-toroidal-symbol} and Proposition~\ref{prop:PsiDOs.normalization-symbol}. 
\begin{proposition} \label{prop:PsiDOs-parameter.extension-minus-given-symbol-is-smooth}
Let $\rho(\xi;\lambda)\in \stS^{m,d}(\R^n\times \Lambda; \cA_\theta)$, $m, d\in \R$.  
\begin{enumerate}
 \item[(i)] $\tilde{\rho}(\xi;\lambda)\in \stS^{m,d}(\R^n\times \Lambda; \cA_\theta)$ and $\tilde{\rho}(\xi;\lambda)-\rho(\xi;\lambda)\in \stS^{-\infty,d}(\R^n\times \Lambda; \cA_\theta)$. 
 
 \item[(ii)] Suppose that $(\rho(k;\lambda))_{k\in \Z^n}\in \stS^{-\infty,d}(\Z^n\times \Lambda; \cA_\theta)$. Then $\rho(\xi;\lambda)\in \stS^{-\infty,d}(\R^n\times \Lambda; \cA_\theta)$. 
\end{enumerate}
\end{proposition}
\begin{proof}
 The first part follows from Proposition~\ref{prop:Parameter.linear-map-symbols} and Lemma~\ref{lem:PsiDOs-parameter.usual-standard-symbol-restriction-extension-composition-is-continuous}. It remains to prove part (ii). Suppose that $(\rho(k;\lambda))_{k\in \Z^n}\in \stS^{-\infty,d}(\Z^n\times \Lambda; \cA_\theta)$. We then know by Proposition~\ref{prop:PsiDOs-parameter.extension-of-toroidal-symbol-is-standard-symbol} that $\tilde{\rho}(\xi;\lambda)$ is in $\stS^{-\infty,d}(\R^n\times \Lambda; \cA_\theta)$. As $\tilde{\rho}(\xi;\lambda)-\rho(\xi;\lambda)\in \stS^{-\infty,d}(\R^n\times \Lambda; \cA_\theta)$ by part (i), we then deduce that  $\rho(\xi;\lambda)\in \stS^{-\infty,d}(\R^n\times \Lambda; \cA_\theta)$. The proof is complete. 
\end{proof} 

Combining the 2nd part of Proposition~\ref{prop:PsiDOs-parameter.extension-minus-given-symbol-is-smooth} and~(\ref{eq:PsiDOs.rho(k)}) leads us to the following statement. 
\begin{corollary} \label{cor:PsiDOs-parameter.zero-operator-symbol}
 Let $\rho(\xi;\lambda)\in \stS^{m,d}(\R^n\times \Lambda; \cA_\theta)$ be such that $P_{\rho}(\lambda)=0$. Then $\rho(\xi;\lambda)$ is contained in $\stS^{-\infty,d}(\R^n\times \Lambda; \cA_\theta)$. 
\end{corollary}

\subsection{Smoothing operators with parameter}
We would like to have a $\Hol^d(\Lambda)$-version of the characterization of smoothing operators provided by Proposition~\ref{prop:PsiDOs.smoothing-operator-characterization}. Once again a smoothing operator is a linear operator on $\cA_\theta$ that extends to a continuous linear map from $\cA_\theta'$ to $\cA_\theta$. Thus, the space of smoothing operators is naturally identified with $\cL(\cA_\theta',\cA_\theta)$. 

We equip $\cL(\cA_\theta',\cA_\theta)$ with its strong topology, i.e., the topology of uniform convergence on bounded subsets of $\cA_\theta'$, where $\cA_\theta'$  itself is equipped with its strong topology. Given $s,t\in\R$, we denote by $\| \cdot \|_{s,t}$ the operator norm of $\cL(\cH_\theta^{(s)},\cH_\theta^{(t)})$.

By Proposition~\ref{prop:Sobolev.Hs-inclusion-cAtheta} the topology of $\cA_\theta$ is the projective limit of the $\cH_\theta^{(s)}$-topologies while their inductive limit yields the topology of $\cA_\theta'$. This enables us to understand the topology of $\cL(\cA_\theta', \cA_\theta)$ in terms of the norms $\| \cdot \|_{s,t}$. Namely, we have the following result.

\begin{proposition}\label{prop:topo-smoothing}
 The topology of $\cL(\cA_\theta',\cA_\theta)$ is generated by the norms $\| \cdot \|_{s,t}$, $s,t\in \R$. 
\end{proposition}
\begin{proof}
 See Appendix~\ref{app:topo-smoothing}. 
\end{proof}

\begin{proposition} \label{prop:PsiDOs-parameter.symbol-of-smoothing-operators}
The following holds. 
\begin{enumerate}
 \item The linear map $\cS(\R^n;\cA_\theta) \ni \rho(\xi) \rightarrow P_\rho \in \cL(\cA_\theta',\cA_\theta)$ is continuous.
 
 \item There is a continuous linear map $ \cL(\cA_\theta',\cA_\theta)\ni R\rightarrow \rho_R(\xi)\in \cS(\R^n; \cA_\theta)$ such that $R=P_{\rho_R}$ for all $R\in \cL(\cA_\theta',\cA_\theta)$. 
\end{enumerate}
\end{proposition}
\begin{proof}
 Let $s,t\in \R$ and set $m=s-t$. The quantization map $\rho(\xi)\rightarrow P_\rho$ gives rise to a continuous linear map from $\stS^m(\R^n;\cA_\theta)$ to $\cL(\cH^{(t+m)}_\theta, \cH_\theta^{(t)})= \cL(\cH^{(s)}_\theta, \cH_\theta^{(t)})$. Combining this with the continuity of the inclusion of $\cS(\R^n;\cA_\theta)$ into $\stS^m(\R^n;\cA_\theta)$ shows that $\rho \rightarrow \|P_\rho\|_{s,t}$ is a continuous semi-norm on  $\cS(\R^n;\cA_\theta)$. As by Proposition~\ref{prop:topo-smoothing} the norms $\| \cdot \|_{s,t}$, $s,t\in \R$, generate the topology of $\cL(\cA_\theta',\cA_\theta)$, it then follows that  $ \rho(\xi) \rightarrow P_\rho$ is a continuous linear map from $\cS(\R^n;\cA_\theta)$ to $\cL(\cA_\theta',\cA_\theta)$. This proves the first part. 
 
It remains to prove the 2nd part.  Given any $R\in  \cL(\cA_\theta',\cA_\theta)$ we know by Proposition~\ref{prop:PsiDOs.smoothing-operator-characterization} there is $\rho_R(\xi)\in \cS(\R^n; \cA_\theta)$ such that $R$ is the \psido\ associated with $\rho_R(\xi)$. As the proof of~\cite[Proposition~6.30]{HLP:Part1} shows we may take $\rho_R(\xi)$ to be the extension~(\ref{eq:toroidal.trho}) of the toroidal symbol $(R[U^k](U^k)^*)_{k\in \Z^n}$ (this can be also seen by using~(\ref{eq:PsiDOs.rho(k)}) and Proposition~\ref{prop:PsiDOs.normalization-symbol}). 
\begin{claim*}
 The linear map $ \cL(\cA_\theta',\cA_\theta)\ni  R \rightarrow (R[U^k])_{k\in \Z^n}\in \cS(\Z^n; \cA_\theta)$ is continuous.  
\end{claim*}
 \begin{proof}[Proof of the Claim]
Given $N\in \N_0$ and $\alpha \in \N_0^n$, let $s<-\frac{n}{2}-N$.  Any unitary $U^k$, $k\in \Z^n$, is contained in $\cH_\theta^{(s)}$ and it follows from~(\ref{eq:Sobolev.Sobolev-norm-formula}) that $\|U^k\|_s=(1+|k|^2)^{\frac{s}2}$.  As $2s+2N<-n$, we have
\begin{equation*}
 \sum_{k\in \Z^n} \big\| (1+|k|)^N U^k\|_{s}^2 =  \sum_{k\in \Z^n} (1+|k|^2)^{s}  (1+|k|)^{2N}<\infty.
\end{equation*}
This shows that the family $\{(1+|k|)^N U^k; k\in \Z^n\}$ is bounded in $\cH_\theta^{(s)}$. 

Bearing this in mind, set $t=|\alpha|-s$. As $-s>\frac{n}2$ it can be shown (see, e.g., \cite[Appendix~A]{HLP:Part2}) that there is $C_{s}>0$ independent of $\alpha$ such that 
\begin{equation*}
 \big\| \delta^\alpha u\|\leq C_{s}\|u\|_{t} \qquad \forall u\in \cA_\theta. 
\end{equation*}
Therefore, for all $k\in \Z^n$, we have 
\begin{equation*}
 (1+|k|)^N \left\| \delta^\alpha \left[ R(U^k)\right]\right\|  \leq C_s (1+|k|)^N  \left\|R(U^k)\right\|_{t} 
  \leq C_s \big\| R\big\|_{s,t}  \left\| (1+|k|)^N U^k\right\|_{s}. 
\end{equation*}
As $\{(1+|k|)^N U^k; k\in \Z^n\}$ is bounded in $\cH_\theta^{(s)}$, we deduce there is $C_{N\alpha}>0$ such that
\begin{equation*}
 \sup_{k\in \Z^n} (1+|k|)^N \left\| \delta^\alpha \left[ R(U^k)\right]\right\| \leq C_{N\alpha} \big\| R\big\|_{s,t} \qquad \forall R\in  \cL(\cA_\theta',\cA_\theta).  
\end{equation*}
As $\|\cdot\|_{s,t}$ is a continuous semi-norm on $ \cL(\cA_\theta',\cA_\theta)$ this shows that $R \rightarrow (R[U^k])_{k\in \Z^n}$ is continuous linear map from $ \cL(\cA_\theta',\cA_\theta)$ to   $\cS(\Z^n; \cA_\theta)$. The claim is proved. 
\end{proof}

Bearing this in mind, it follows from the proof of~\cite[Lemma~6.29]{HLP:Part1} that we have a continuous linear map, 
\begin{equation*}
 \cS(\Z^n;\cA_\theta) \ni (\rho_k)_{k\in \Z^n} \longrightarrow (\rho_k(U^k)^*)_{k\in \Z^n}\in  \cS(\Z^n;\cA_\theta). 
\end{equation*}
Composing it with the extension map~(\ref{eq:toroidal.trho}) we get a continuous linear map from $\cS(\Z^n;\cA_\theta)$ to $\cS(\R^n;\cA_\theta)$. As $\rho_R(\xi;\lambda)$ is precisely the image of $(R[U^k])_{k\in \Z^n}$ under that map, by using the claim above we see that $R\rightarrow \rho_R(\xi)$ is a continuous linear map from $ \cL(\cA_\theta',\cA_\theta)$ to   $\cS(\R^n; \cA_\theta)$. This proves the 2nd part. The proof is complete. 
\end{proof}

We are now in a position to prove the following $\Hol^d(\Lambda)$-version of Proposition~\ref{prop:PsiDOs.smoothing-operator-characterization}. 

\begin{proposition} \label{prop:PsiDOs-parameter.smoothing-operators-with-parameter-characterization}
We have
\begin{equation*}
 \Psi^{-\infty,d}(\cA_\theta; \Lambda)= \Hol^d\big(\Lambda; \cL(\cA_\theta',\cA_\theta)\big)= \bigcap_{s,t\in \R}  \Hol^d\big( \Lambda; \cL\big(\cH_\theta^{(s)},\cH_\theta^{(t)}\big)\big). 
\end{equation*}
\end{proposition}
\begin{proof}
 It follows from Proposition~\ref{prop:Parameter.linear-map-symbols} and Proposition~\ref{prop:PsiDOs-parameter.symbol-of-smoothing-operators} that we have quantization and symbol maps,
\begin{gather}
\label{eq:PsiDOs-parameter.symbol-degree-minus-infinity-with-parameter-quantization-map}
 \stS^{-\infty,d}(\R^n\times \Lambda; \cA_\theta) \ni \rho(\xi;\lambda) \longrightarrow P_\rho(\lambda) \in \Hol^d\big( \Lambda; \cL(\cA_\theta',\cA_\theta)\big),\\\nonumber
 \Hol^d\big( \Lambda; \cL(\cA_\theta',\cA_\theta)\big) \ni R(\lambda)  \longrightarrow \rho_R(\xi;\lambda)\in  \stS^{-\infty,d}(\R^n\times \Lambda; \cA_\theta). 
\end{gather}
Furthermore, the symbol map is a right-inverse of the quantization map. 

If $R(\lambda)\in \Psi^{-\infty,d}(\cA_\theta; \Lambda)$, then $R(\lambda)=P_\rho(\lambda)$ for some $\rho(\xi;\lambda)\in \stS^{-\infty,d}(\R^n\times \Lambda; \cA_\theta)$, and so~(\ref{eq:PsiDOs-parameter.symbol-degree-minus-infinity-with-parameter-quantization-map}) ensures us that $R(\lambda)\in  \Hol^d( \Lambda; \cL(\cA_\theta',\cA_\theta))$. Conversely, if $R(\lambda)\in  \Hol^d( \Lambda; \cL(\cA_\theta',\cA_\theta))$, then $\rho_R(\xi;\lambda)\in  \stS^{-\infty,d}(\R^n\times \Lambda; \cA_\theta)$ and $R(\lambda)=P_{\rho_R}(\lambda)$, and hence $ R(\lambda)\in  \Psi^{-\infty,d}(\cA_\theta; \Lambda)$. This shows that $\Psi^{-\infty,d}(\cA_\theta; \Lambda)= \Hol^d( \Lambda; \cL(\cA_\theta',\cA_\theta))$. 
 
It remains to show that $ \Hol^d( \Lambda; \cL(\cA_\theta',\cA_\theta))= \bigcap_{s,t\in \R} \Hol^d( \Lambda; \cL(\cH_\theta^{(s)},\cH_\theta^{(t)}))$. The continuity of the inclusions of $\cL(\cA_\theta',\cA_\theta)$ into $\cL(\cH_\theta^{(s)},\cH_\theta^{(t)})$, $s,t\in \R$, and Proposition~\ref{prop:Parameter.linear-map} ensure us that
\begin{equation} \label{eq:PsiDOs-parameter.smoothing-operator-with-parameter-Sobolev-mapping-inclusion}
  \Hol^d\big( \Lambda; \cL(\cA_\theta',\cA_\theta)\big)\subset  \Hol^d\big( \Lambda; \cL\big(\cH_\theta^{(s)},\cH_\theta^{(t)}\big)\big) \qquad \text{for all $s,t\in \R$}. 
\end{equation}
Conversely, let $R(\lambda)\in  \bigcap_{s,t\in \R} \Hol^d( \Lambda; \cL(\cH_\theta^{(s)},\cH_\theta^{(t)}))$. Then, for every $\lambda\in \Lambda$, the operator $R(\lambda)$ is contained in all the spaces $\cL(\cH_\theta^{(s)},\cH_\theta^{(t)})$, $s,t\in \R$, and so this is a smoothing operator by Proposition~\ref{prop:PsiDOs.smoothing-operator-characterization}. Moreover, given any $s,t\in \R$, for every  pseudo-cone $\Lambda'\subsubset \Lambda$, there is $C_{\Lambda'st}>0$ such that
\begin{equation*}
 \big\| R(\lambda)\big\|_{s,t} \leq C_{\Lambda'st} (1+|\lambda|)^d \qquad \text{for all $s,t\in \Lambda'$}. 
\end{equation*}
As by Proposition~\ref{prop:topo-smoothing} the norms $\|\cdot \|_{s,t}$,  $s,t\in \R$, generate the topology of $ \cL(\cA_\theta',\cA_\theta)$, this shows that $R(\lambda)\in  \Hol^d( \Lambda; \cL(\cA_\theta',\cA_\theta))$. Combining this with~(\ref{eq:PsiDOs-parameter.smoothing-operator-with-parameter-Sobolev-mapping-inclusion}) shows that  $ \Hol^d( \Lambda; \cL(\cA_\theta',\cA_\theta))$ agrees with $\bigcap_{s,t\in \R} \Hol^d( \Lambda; \cL(\cH_\theta^{(s)},\cH_\theta^{(t)}))$. The proof is complete. 
\end{proof}

\begin{corollary} \label{cor:PsiDOs-parameter.intersection-of-psidos-of-all-orders}
$\Psi^{-\infty,d}(\cA_\theta; \Lambda)= \bigcap_{m\in \R} \Psi^{m,d}(\cA_\theta; \Lambda)$. 
\end{corollary}
\begin{proof}
 We know that $ \Psi^{-\infty,d}(\cA_\theta; \Lambda)\subset \bigcap_{m\in \R} \Psi^{m,d}(\cA_\theta; \Lambda)$ (\emph{cf}.\ Remark~\ref{rmk:PsiDOs-parameter.order-minus-infty-to-intersection-inclusion}). Conversely, if $R(\lambda)$ is in  $\Psi^{m,d}(\cA_\theta; \Lambda)$ for every $m\in \R$, then by Proposition~\ref{prop:PsiDOs-parameter.Sobolev-mapping-properties} it is in $ \Hol^d( \Lambda; \cL(\cH_\theta^{(s)},\cH_\theta^{(t)}))$ for all $s,t\in \R$. Proposition~\ref{prop:PsiDOs-parameter.smoothing-operators-with-parameter-characterization} then shows that $R(\lambda)\in  \Psi^{-\infty,d}(\cA_\theta; \Lambda)$. The proof is complete.
 \end{proof}

%

\section{The Resolvent of an Elliptic \psido} \label{sec:Resolvent}
In this section, we show that the pseudodifferential calculus with parameter from the previous two sections is a natural nest for the resolvent of elliptic operators.  

Throughout this section, we let $P:\cA_\theta\rightarrow\cA_\theta$ be an elliptic \psido\ of order $w>0$ with symbol $\rho(\xi)\sim\sum_{j\geq 0}\rho_{w-j}(\xi)$.
 
 \subsection{Parametrix construction} 
We start by constructing an explicit parametrix for $P-\lambda$. The first task is single out the relevant set of parameters for which we have a parametrix. 

Recall that by Proposition~\ref{prop:NCtori.invertibility-cAtheta} $\rho_w(\xi)$ is invertible if and only if it is invertible in $\cL(\cH_\theta)$.

\begin{definition}
The \emph{elliptic parameter set} of $P$ is 
\begin{align*}
\Theta(P) &= \C^*\setminus \biggl[ \bigcup_{\xi\in\Rn\setminus 0} \Sp(\rho_w(\xi))\biggr]\\
& = \big\{\lambda \in \C^*; \ \text{$\rho_w(\xi)-\lambda$ is invertible for all $\xi\in \R^n\setminus 0$}\big\}. 
\end{align*}
\end{definition}

\begin{remark}\label{rmk:resolvent.thetaP-cone}
 The fact that $\Sp{\rho_w(t\xi)} = \Sp{t^w\rho_w(\xi)} = t^w\Sp{\rho_w(\xi)}$ for all $t>0$ implies that $\Theta(P)$ is invariant under positive dilations. It then follows that $\Theta(P)$ is a cone in $\C^*$.  
\end{remark}

\begin{example}
 Suppose that $\rho_w(\xi)$ is selfadjoint for all $\xi\in \R^n\setminus \{0\}$ (e.g., $P$ is selfadjoint). Then $\Sp (\rho_w(\xi))\subset \R$ for all  $\xi\in \R^n\setminus \{0\}$, and so $\Theta(P) \supset \C\setminus \R$. 
\end{example}

\begin{example}
 If $\rho_w(\xi)$ is positive in the sense of~(\ref{eq:Elliptic.positivity-criterion}), then $\Sp (\rho_w(\xi))\subset (0,\infty)$, and so $\Theta(P)$ contains $\C\setminus [0,\infty)$. As $\Theta (P)$ is a cone and $\Sp (\rho_w(\xi))$ cannot be empty, we see that $\Theta(P)=\C\setminus [0,\infty)$. 
\end{example}

\begin{definition}
 We say that $P$ is \emph{elliptic with parameter} when $\Theta(P)\neq \emptyset$. 
\end{definition}

Throughout the rest of this section we assume that $P$ is elliptic with parameter.

\begin{lemma} \label{lem:Resolvent.open_angular-sector}
$\Theta(P)$ is an open cone in $\C^*$.
\end{lemma}
\begin{proof}
We know by Remark~\ref{rmk:resolvent.thetaP-cone} that $\Theta(P)$ is a cone in $C^*$. Let $\lambda_0\in \Theta(P)$, and set $\phi_0=\arg \lambda_0\in (-\pi,\pi]$. By homogeneity the whole ray $\{te^{i\phi_0}; t>0\}$ is contained in $\Theta(P)$.  Together with the ellipticity of $P$ this implies that $\rho_w(\xi)-te^{i\phi_0}$ is invertible for all $\xi\in \R^n\setminus 0$ and $t\geq 0$. In addition, set  $r_0 = \sup_{\xi\in \bS^{n-1}}\norm{\rho_w(\xi)}$. If $|\lambda|>r_0$ and $\xi\in \bS^{n-1}$, then $|\lambda|>\|\rho_w(\xi)\|$ for all $\xi\in \bS^{n-1}$, and so $\rho_w(\xi)-\lambda$ is invertible. 

Let $r>r_0$. As mentioned above $\rho_w(\xi)-te^{i\phi_0}\in \cA_\theta^{-1}$ for all $\xi\in \bS^{n-1}$ and $t\in [0,r]$. As $ \cA_\theta^{-1}$ is an open set, we see that every $(\xi_1,t_1)\in \bS^{n-1}\times [0,r]$ admits an open neighborhood $U_1\subset  \bS^{n-1}\times [0,r]$ such that there is $\delta>0$ for which $\rho_w(\xi)-te^{i\phi}\in \cA_\theta^{-1}$ for all $(\xi,t)\in U_1$ and $\phi \in (\phi_0-\delta, \phi_0+\delta)$. The compactness of $\bS^{n-1}\times [0,r]$ allows us to cover $\bS^{n-1}\times [0,r]$ by finitely many such open sets. Therefore, we can find $\delta>0$ such that $\rho_w(\xi)-te^{i\phi}\in \cA_\theta^{-1}$ for all $(\xi,t)\in \bS^{n-1}\times [0,r]$ and $\phi \in (\phi_0-\delta, \phi_0+\delta)$. As $r>r_0$ and $\rho_w(\xi)-\lambda\in \cA_\theta^{-1}$ when $\xi\in \bS^{n-1}$ and $|\lambda|>r_0$, we deduce that  
$\rho_w(\xi)-\lambda$ is invertible for all $\xi \in \bS^{n-1}$ and all $\lambda$ in the open angular sector $\Theta:=\{|\arg \lambda -\phi_0|<\delta\}$. By homogeneity this implies that $\rho_w(\xi)-\lambda$ is invertible for all $\xi \in \R^n\setminus 0$ and $\lambda \in \Theta$. That is, $\Theta$ is contained in $\Theta(P)$. As $\Theta$ is an open set this implies that $\Theta(P)$ is a neighborhood of $\lambda_0$. Thus,  $\Theta(P)$ is a neighborhood of each of its points, and so this is an open set. 
The proof is complete. 
\end{proof} 

In what follows, setting $c:=\inf\{\norm{\rho_w(\xi)^{-1}}^{-1};\ |\xi|=1\}$ we define 
\begin{equation}
\Omega_c(P) = \{ (\xi,\lambda)\in(\Rn\setminus 0)\times\C^* ; \ \text{$\lambda\in\Theta(P)$ or $|\lambda|<c|\xi|^w$} \} .
\label{eq:Resolvent.OmegacP}
\end{equation}
In the notation of~(\ref{eq:Parameter.parameter-set-for-homogeneous-symbols}) this is just the open set $\Omega_c(\Theta(P))$ associated with the open cone $\Theta(P)$.  

\begin{lemma} \label{lem:Resolvent.principal-symbol-minus-lambda-smooth-symbol-with-parameter}
The following holds. 
\begin{enumerate}
 \item $\rho_w(\xi)-\lambda$ is invertible for all $(\xi,\lambda)\in \Omega_c(P)$. 
 
 \item $(\rho_w(\xi)-\lambda)^{-1}\in S_{-w}^{-1}(\Omega_c(P);\cA_\theta)$. 
\end{enumerate}
\end{lemma}
\begin{proof}
Set $\Omega_0(P)=(\R^n\setminus 0) \times \Theta(P)$. We have 
\begin{equation*}
 \Omega_c(P)=\Omega_0(P) \cup \{ (\xi,\lambda)\in(\Rn\setminus 0)\times\C^* ;\  |\lambda|<c|\xi|^w\}. 
\end{equation*}
The very definition of $\Theta(P)$ means that $\rho_w(\xi)-\lambda$ is invertible for all $(\xi,\lambda)\in \Omega_0(P)$. Let $(\xi,\lambda)\in(\Rn\setminus 0)\times\C^*$ be such that $|\lambda|<c|\xi|^w$. The definition of $c$ implies that $\|\rho_w(\eta)^{-1}\|\leq c^{-1} $ for all $\eta \in \bS^{n-1}$. In particular, the inequality holds for $\eta =|\xi|^{-1}\xi$. Thus,  
\begin{equation*}
\norm{\rho_w(\xi)^{-1}} =|\xi|^{-w} \norm{\rho_w\big(|\xi|^{-1}\xi\big)^{-1}} \leq |\xi|^{-w}c^{-1}<|\lambda|^{-1} . 
\end{equation*}
This implies that $\lambda^{-1}-\rho_w(\xi)^{-1}$ is invertible. As $\rho_w(\xi)-\lambda=\lambda\rho_w(\xi)(\lambda^{-1}-\rho_w(\xi)^{-1})$, we deduce that 
 $\rho_w(\xi)-\lambda$ is invertible when $0<|\lambda|<c|\xi|^w$. This proves the first part. 

The homogeneity of $\rho_w(\xi)$ implies that $(\rho_{w}(t\xi)-t^w\lambda)^{-1}=t^{-w}(\rho_w(\xi)-\lambda)^{-1}$ for all $(\xi,\lambda)\in \Omega_c(P)$ and $t>0$. 
Moreover, the differentiability of the inverse map $u \rightarrow u^{-1}$ of $\cA_\theta$ ensures us that $(\rho_w(\xi)-\lambda)^{-1}\in C^\infty(\Omega_c(P); \cA_\theta)$ (see~\cite[Lemma~11.2]{HLP:Part2}). Moreover, as $\partial_{\overline{\lambda}}(\rho_w(\xi)-\lambda)^{-1}=(\rho_w(\xi)-\lambda)^{-1}\partial_{\overline{\lambda}}(\lambda)(\rho_w(\xi)-\lambda)^{-1}=0$, we see that $(\rho_w(\xi)-\lambda)^{-1}$ is holomorphic with respect to $\lambda$, and hence $(\rho_w(\xi)-\lambda)^{-1}\in C^{\infty,\omega}(\Omega_c(P); \cA_\theta)$. 

To complete the proof it just remains to show that $(\rho_{w}(\xi)-\lambda)^{-1} \in C^{\infty,-1}(\Omega_0(P); \cA_\theta)$. Note that if $|\lambda|>\norm{\rho_w(\xi)}$, then  
\begin{equation} \label{eq:Resolvent.principal-symbol-minus-lambda-outside-spectral-radius}
 \norm{(\rho_{w}(\xi)-\lambda)^{-1}} = |\lambda|^{-1} \norm{(1-\lambda^{-1} \rho_{w}(\xi))^{-1}} \leq \frac{|\lambda|^{-1}}{1-|\lambda|^{-1}\norm{\rho_w(\xi)}}.
\end{equation}
Let $K\subset\Rn\setminus 0$ be compact. If $|\lambda|>2 \sup \{\norm{\rho_w(\xi)};\ \xi\in K\}$, then $\norm{\rho_w(\xi)}\leq \frac12 |\lambda|$ for all $\xi\in K$, and so by using~(\ref{eq:Resolvent.principal-symbol-minus-lambda-outside-spectral-radius}) we get $\norm{(\rho_{w}(\xi)-\lambda)^{-1}} \leq 2|\lambda|^{-1}$. It then follows that, for every pseudo-cone $\Theta'\subsubset \Theta(P)$, there is $C_{\Theta' K}>0$ such that 
\begin{equation} \label{eq:Resolvent.principal-symbol-minus-lambda}
  \norm{(\rho_{w}(\xi)-\lambda)^{-1}} \leq C_{\Theta' K} \big(1+|\lambda|\big)^{-1} \qquad \forall (\xi,\lambda)\in K\times \Theta'. 
\end{equation}
Given any multi-orders $\alpha$ and $\beta$ the multi-derivative $\delta^\alpha\partial_\xi^\beta(\rho_w(\xi)-\lambda)^{-1}$ is a linear combination of products of the form,
\begin{equation*}
 (\rho_w(\xi)-\lambda)^{-1}[\delta^{\alpha^1}\partial_\xi^{\beta^1}\rho_w(\xi)](\rho_w(\xi)-\lambda)^{-1} \cdots  
 [\delta^{\alpha^{k}}\partial_\xi^{\beta^{k}}\rho_w(\xi)](\rho_w(\xi)-\lambda)^{-1}. 
\end{equation*}
 Combining this with~(\ref{eq:Resolvent.principal-symbol-minus-lambda}) allows us to show that, for every pseudo-cone $\Theta'\subsubset \Theta(P)$, there is $C_{\Theta' K\alpha \beta}>0$ such that 
\begin{equation*}
  \norm{\delta^\alpha\partial_\xi^\beta(\rho_{w}(\xi)-\lambda)^{-1}} \leq C_{\Theta' K\alpha \beta} \big(1+|\lambda|\big)^{-1} \qquad \forall (\xi,\lambda)\in K\times \Theta'. 
\end{equation*}
 This shows $(\rho_{w}(\xi)-\lambda)^{-1} \in C^{\infty,-1}(\Omega_0(P); \cA_\theta)$. The proof is complete. 
\end{proof}

In what follows, given any $R>0$, we denote by  $\Lambda_R$ the open pseudo-cone defined by
\begin{equation*}
\Lambda_R = \Theta(P)\cup D(0,R), 
\end{equation*}
where $D(0,R)$ is the open disk of radius $R$ centered at the origin. 

\begin{theorem} \label{thm:Resolvent.parametrix-with-parameter}
Suppose that $P$ is elliptic with parameter. Then, for every $R>0$, the following holds. 
\begin{enumerate}
 \item $P-\lambda$ admits a parametrix $Q(\lambda)\in\Psi^{-w,-1}(\cA_\theta;\Lambda_R)$ in the sense that
\begin{equation} \label{eq:Resolvent.parametrix-with-parameter}
(P-\lambda)Q(\lambda) = Q(\lambda)(P-\lambda)=1  \quad \bmod \quad \Psi^{-\infty,0}(\cA_\theta;\Lambda_R) .
\end{equation}

\item Any parametrix $Q(\lambda)\in \Psi^{-w,-1}(\cA_\theta;\Lambda_R)$ as above has symbol $\sigma(\xi;\lambda)\sim \sum_{j\geq 0}  \sigma_{-w-j}(\xi;\lambda)$, where 
$ \sigma_{-w-j}(\xi;\lambda)\in S^{-1}_{-w-j}(\Omega_c(P);\cA_\theta)$ is given by
\begin{gather}
 \sigma_{-w}(\xi;\lambda) = \big(\rho_w(\xi)-\lambda\big)^{-1} , 
 \label{eq:Resolvent.parametrix-principal-symbol} \\
 \sigma_{-w-j}(\xi;\lambda) = -\sum_{\substack{k+l+|\alpha|=j \\ l<j}}\frac{1}{\alpha !}\big(\rho_w(\xi)-\lambda\big)^{-1} \partial_\xi^\alpha \rho_{w-k}(\xi) \delta^\alpha\sigma_{-w-l}(\xi;\lambda), \qquad j\geq 1 .
 \label{eq:Resolvent.parameterix-homogeneous-symbols}
\end{gather}
\end{enumerate}
\end{theorem}
\begin{proof}
We regard $P-\lambda$ as an element of $\Psi^{w,1}(\cA_\theta;\cA_\theta)$ (\emph{cf}.\ Example~\ref{ex:PsiDOs-parameter.usual-PsiDO-can-be-seen-as-PsiDO-with-parameter}). We claim that the formulas~(\ref{eq:Resolvent.parametrix-principal-symbol})--(\ref{eq:Resolvent.parameterix-homogeneous-symbols}) define maps $\sigma_{-w-j}(\xi;\lambda)\in S_{-w-j}^{-1}(\Omega_c(P);\cA_\theta)$ for $j=0,1,
\ldots$. We prove this by induction on $j$. For $j=0$ this is the contents of Lemma~\ref{lem:Resolvent.principal-symbol-minus-lambda-smooth-symbol-with-parameter}. Assume the result is true for $l<j$ with $j\geq 1$. Then $\sigma_{-w-j}(\xi;\lambda)$ is a linear combination of terms of the form, 
\begin{equation} \label{eq:Resolvent.homogeneous-parts-of-resolvent-components}
 \big(\rho_w(\xi)-\lambda\big)^{-1} \partial_\xi^\alpha \rho_{w-k}(\xi) \delta^\alpha\sigma_{-w-l}(\xi;\lambda),\qquad k+l+|\alpha|=j,\quad  l<j. 
\end{equation}
Here $(\rho_w(\xi)-\lambda)^{-1} \in S_{-w}^{-1}(\Omega_c(P);\cA_\theta)$. As mentioned in Example~\ref{ex:Parameter.homogeous-symbol-can-be-seen-as-homogeneous-symbol-with-parameter} we can regard $\partial_\xi^\alpha \rho_{w-k}(\xi)$ as an element of $S_{w-k-|\alpha|}^{0}(\Omega_c(P);\cA_\theta)\subset S_{w-k-|\alpha|}^{1}(\Omega_c(P);\cA_\theta)$. Moreover, by assumption 
$\sigma_{-w-l}(\xi;\lambda)\in S_{-w-l}^{-1}(\Omega_c(P);\cA_\theta)$, and so it follows from Remark~\ref{rem:Parameter.homogeneous-symbol-partial-derivative} that $ \delta^\alpha\sigma_{-w-l}(\xi;\lambda)$ is contained in $S_{-w-l}^{-1}(\Omega_c(P);\cA_\theta)$ as well. Therefore, by using Remark~\ref{rmk:Parameter.homogeneous-symbol-product} we see that each product of the form~(\ref{eq:Resolvent.homogeneous-parts-of-resolvent-components}) is contained in $S_{-w-j}^{-1}(\Omega_c(P);\cA_\theta)$, and hence  $\sigma_{-w-j}(\xi;\lambda)\in S_{-w-j}^{-1}(\Omega_c(P);\cA_\theta)$.  

By Proposition~\ref{prop:Parameter.Borel-for-classical-symbol} there is $\sigma(\xi;\lambda)\in S^{-w,-1}(\R^n\times \Lambda_R;\cA_\theta)$ such that $\sigma(\xi;\lambda)\sim  \sum_{j\geq 0}  \sigma_{-w-j}(\xi;\lambda)$. Set $Q(\lambda)=P_{\sigma}(\lambda)$. Then $Q(\lambda)\in \Psi^{-w,-1}(\cA_\theta; \Lambda_R)$. Moreover, by Proposition~\ref{prop:Parameter.composition-PsiDOs} the composite $(P-\lambda)Q(\lambda)$ is contained in $\Psi^{0,0}(\cA_\theta; \Lambda_R)$ and has symbol $[(\rho-\lambda)\sharp \sigma] (\xi;\lambda)$. Moreover, it follows from~(\ref{eq:PsiDOs-parameter.composition-symbol-homogeneous-parts}) and~(\ref{eq:Resolvent.parametrix-principal-symbol})--(\ref{eq:Resolvent.parameterix-homogeneous-symbols})
that $[(\rho-\lambda)\sharp \sigma] (\xi;\lambda)\sim  \sum_{j\geq 0} [(\rho-\lambda)\sharp \sigma]_{-j}(\xi;\lambda)$, where
\begin{equation*}
\big[(\rho-\lambda)\sharp \sigma\big]_{0}(\xi;\lambda)= \big(\rho_w(\xi)-\lambda\big) \sigma_{-w}(\xi;\lambda)=1,
 \end{equation*}
 and for $j\geq 1$ the homogeneous component $[(\rho-\lambda)\sharp \sigma]_{-j}(\xi; \lambda)$ is equal to
\begin{equation*}
 \big(\rho_w(\xi)-\lambda\big) \sigma_{-w-j}(\xi;\lambda) +\sum_{\substack{k+l+|\alpha|=j \\ l<j}}\frac{1}{\alpha !}\partial_\xi^\alpha \rho_{w-k}(\xi) \delta^\alpha\sigma_{-w-l}(\xi;\lambda)=0. 
\end{equation*}
This means that $(\rho-\lambda)\sharp \sigma (\xi;\lambda) \sim 1$, i.e., $(\rho-\lambda)\sharp \sigma (\xi;\lambda)-1\in \stS^{-\infty,0}(\R^n\times \Lambda_R; \cA_\theta)$. Thus, 
\begin{equation*}
 (P-\lambda) Q(\lambda)=P_{(\rho-\lambda)\sharp \sigma}(\lambda)=1\qquad  \bmod \Psi^{-\infty,0}(\cA_\theta; \Lambda_R). 
\end{equation*}

The invertibility of $\rho_w(\xi)-\lambda$ also enables us to construct $\tilde{\sigma}_{-w-j}(\xi;\lambda)\in S_{-w-j}^{-1}(\Omega_c(P);\cA_\theta)$, $j=0,1,\ldots$, such that in the notation of~(\ref{eq:PsiDOs-parameter.composition-symbol-homogeneous-parts}), for $j=0$, we have 
\begin{equation} \label{eq:Resolvent.left-parametrix-first-equation}
\big[\tilde{\sigma} \sharp (\rho-\lambda) \big]_{0}(\xi;\lambda)= \tilde{\sigma}_{-w}(\xi;\lambda)\big(\rho_w(\xi)-\lambda\big) =1,\\ 
\end{equation}
and  $[\tilde{\sigma} \sharp  (\rho-\lambda)]_{-j}(\xi; \lambda)$, $j\geq 1$, is equal to
\begin{equation} \label{eq:Resolvent.left-parametrix-equations}
 \tilde{\sigma}_{-w-j}(\xi;\lambda)  \big(\rho_w(\xi)-\lambda\big)  +\sum_{\substack{k+l+|\alpha|=j \\ k<j}}\frac{1}{\alpha !}
\partial_\xi^\alpha \tilde{\sigma}_{-w-k}(\xi;\lambda) \delta^\alpha\rho_{w-l}(\xi)=0. 
\end{equation}
By Proposition~\ref{prop:Parameter.Borel-for-classical-symbol} there is $\tilde{\sigma}(\xi;\lambda) \in S^{-w,-1}(\R^n\times \Lambda_R;\cA_\theta)$ such that $\tilde{\sigma}(\xi;\lambda)\sim \sum_{j\geq 0} \tilde{\sigma}_{-w-j}(\xi;\lambda)$. As above~(\ref{eq:Resolvent.left-parametrix-first-equation})--(\ref{eq:Resolvent.left-parametrix-equations}) implies that $\tilde{\sigma}\sharp (\rho-\lambda)-1$ is in $\stS^{-\infty,0}(\R^n\times \Lambda_R;\cA_\theta)$. 
Thus, if we set $\tilde{Q}(\lambda)=P_{\tilde{\sigma}}(\lambda)$, then $\tilde{Q}(\lambda)\in \Psi^{-w,-1}(\cA_\theta;\Lambda_R)$ and $\tilde{Q}(\lambda)(P-\lambda)=1$ modulo $\Psi^{-\infty,0}(\cA_\theta;\Lambda_R)$. 

Set $R(\lambda)= (P-\lambda) Q(\lambda)-1$ and $\tilde{R}(\lambda)=  \tilde{Q}(\lambda)(P-\lambda)-1$. Then we have 
\begin{gather*}
 \tilde{Q}(\lambda)(P-\lambda)Q(\lambda)= \tilde{Q}(\lambda)\big(1+R(\lambda)\big)=\tilde{Q}(\lambda)+\tilde{Q}(\lambda)R(\lambda),\\
\tilde{Q}(\lambda)(P-\lambda)Q(\lambda)= \big(1+\tilde{R}(\lambda)\big) Q(\lambda)= Q(\lambda) + \tilde{R}(\lambda) Q(\lambda).
\end{gather*}
As $\tilde{Q}(\lambda)R(\lambda)$ and $\tilde{R}(\lambda) Q(\lambda)$ are in $\Psi^{-\infty,-1}(\cA_\theta;\Lambda_R)$, we see that $Q(\lambda) -\tilde{Q}(\lambda)$ is contained in  $\Psi^{-\infty,-1}(\cA_\theta;\Lambda_R)$. Thus, 
\begin{equation*}
 Q(\lambda)(P-\lambda)= \tilde{Q}(\lambda)(P-\lambda)= 1 \quad \bmod \Psi^{-\infty,0}(\cA_\theta;\Lambda_R). 
\end{equation*}
This shows that $Q(\lambda)$ is a parametrix for $P-\lambda$ in the sense of~(\ref{eq:Resolvent.parametrix-with-parameter}).

The above considerations to compare $Q(\lambda)$ and $\tilde{Q}(\lambda)$ also shows that if $Q_1(\lambda)\in \Psi^{-w,-1}(\cA_\theta; \Lambda_R)$ is any other parametrix in the sense of~(\ref{eq:Resolvent.parametrix-with-parameter}), then $Q_1(\lambda)-Q(\lambda)$ is in $\Psi^{-\infty,-1}(\cA_\theta; \Lambda_R)$. Let $\sigma^{(1)}(\xi;\lambda)\in S^{-w,-1}(\R^n\times \Lambda_R;\cA_\theta)$ be the symbol of $Q_1(\lambda)$. Then Corollary~\ref{cor:PsiDOs-parameter.zero-operator-symbol} ensures us that $\sigma^{(1)}(\xi;\lambda)-\sigma(\xi;\lambda)$ is contained in $\stS^{-\infty,-1}(\R^n\times \Lambda_R; \cA_\theta)$, and so $\tilde{\sigma}(\xi;\lambda)\sim \sum \sigma_{-w-j}(\xi;\lambda)$, where the $\sigma_{-w-j}(\xi;\lambda)$ are given by~(\ref{eq:Resolvent.parametrix-principal-symbol})--(\ref{eq:Resolvent.parameterix-homogeneous-symbols}). This proves the 2nd part. The proof is complete. 
\end{proof}

\begin{remark}\label{rmk:Resolvent.Hol0} 
 The smoothing process of Proposition~\ref{prop:Parameter.Borel-for-classical-symbol} involves cutoffs by means of functions $\chi(\xi)\in C^\infty_c(\R^n)$. These are elements of $\stS^{-\infty,0}(\R^n\times \Lambda_R;\cA_\theta)$ which are by no mean unique. This creates an ambiguity with values in $\stS^{-\infty,-1}(\R^n\times \Lambda_R; \cA_\theta)$ in the construction of the symbol of the parametrix of $P-\lambda$. Note also that, in view of~(\ref{eq:Resolvent.parameterix-homogeneous-symbols}), the $\lambda$-decay of symbols $\sigma_{-w-j}(\xi;\lambda)$ need not decrease as $j$ becomes large, since we may have some non-zero contribution from terms like $(\rho_w(\xi)-\lambda)^{-1} \rho_{w-j+1}(\xi) \sigma_{-w-1}(\xi;\lambda)$ with $j$ arbitrary large. Because of all this we can't really expect obtaining better $\lambda$-decay for the remainder terms in~(\ref{eq:Resolvent.parametrix-with-parameter}). When we restrict the parameter set to $\Theta(P)$ we can improve the decay by adding some meromorphic singularities near $\lambda=0$ (see the proof of Theorem~\ref{thm:Resolvent.P-has-discrete-spectrum-resolvent-estimate} below). In some forthcoming work, we will explain that when $P$ is a differential operator and we also take the parameter set to $\Theta(P)$, then there is no need to use cut-off functions anymore and we then can modify the parametrix construction so as to have $\lambda$-decay of any order.
\end{remark}

\subsection{Localization of the spectrum and rays of minimal growth}
We shall now use Theorem~\ref{thm:Resolvent.parametrix-with-parameter} to look at the localization of the spectrum of $P$. 

\begin{theorem} \label{thm:Resolvent.P-has-discrete-spectrum-resolvent-estimate}
Suppose that $P$ is elliptic with parameter.
\begin{enumerate}
\item The spectrum of $P$ is an unbounded discrete subset of $\C$ consisting of eigenvalues with finite multiplicity.

\item For any cone  $\Theta'$ such that $\overline{\Theta'}\setminus \{0\} \subset \Theta(P)$ the following holds. 
\begin{enumerate}
 \item $\Theta'$ contains at most finitely many eigenvalues of $P$. 
 
 \item There are $r_0$ and $C>0$ such that 
\end{enumerate}
\begin{equation} \label{eq:Resolvent.resolvent-estimate}
\big\|(P-\lambda)^{-1} \big\| \leq C|\lambda|^{-1} \qquad \forall \lambda \in \Theta'\setminus D(0,r_0). 
\end{equation}
\end{enumerate}
\end{theorem}
\begin{proof}
By Theorem~\ref{thm:Resolvent.parametrix-with-parameter} there are $Q(\lambda)\in \Psi^{-w,-1}(\cA_\theta; \Theta(P))$ and $R(\lambda)\in \Psi^{-\infty,0}(\cA_\theta; \Theta(P))$ such that $(P-\lambda)Q(\lambda)+R(\lambda)=1$. Set $Q_1(\lambda)=Q(\lambda)-\lambda^{-1}R(\lambda)$ and $R_1(\lambda)=\lambda^{-1}PR(\lambda)$. Then 
$Q_1(\lambda)\in \Psi^{-w,-1}(\cA_\theta; \Theta(P))$ and $R_1(\lambda)\in \Psi^{-\infty,-1}(\cA_\theta; \Theta(P))$. Moreover, we have 
\begin{equation} \label{eq:Resolvent.P-minus-lambda-right-inverse-construction}
 (P-\lambda)Q_1(\lambda)= (P-\lambda)Q(\lambda)-\lambda^{-1}PR(\lambda)+R(\lambda)=1-R_1(\lambda). 
\end{equation}

Let $\Theta'$ be a cone  such that $\overline{\Theta'}\setminus \{0\} \subset \Theta(P)$. This implies that $\Theta'\setminus D(0,r)\subsubset \Theta(P)$ for any $r>0$. Moreover, as $R_1(\lambda)$ is in $\Psi^{-\infty,-1}(\cA_\theta; \Theta(P))$, Proposition~\ref{prop:PsiDOs-parameter.smoothing-operators-with-parameter-characterization} allows us to regard it as an element of $\Hol^{-1}(\Theta(P); \cL(\cH_\theta))$. 
Therefore, there is $C_{\Theta'r}>0$ such that $\|R_1(\lambda)\|\leq C_{\Theta'r}|\lambda|^{-1}$ for all $\lambda \in \Theta'\setminus D(0,r)$. It then follows that there is $r_0>0$ such that $\|R_1(\lambda)\|\leq \frac12$ for all $\lambda \in \Theta'\setminus D(0,r_0)$. 
This ensures that, for every $\lambda \in \Theta'\setminus D(0,r_0)$, the operator $1-R_1(\lambda)$ is invertible in $\cL(\cH_\theta)$ and 
$\|(1-R_1(\lambda))^{-1}\|\leq (1-\|R_1(\lambda)\|)^{-1}\leq 2$. 

As $Q_1(\lambda)\in \Psi^{-w,-1}(\cA_\theta; \Theta(P))$, it follows from Proposition~\ref{prop:PsiDOs-parameter.classical-PsiDO-Sobolev-mapping-properties} that 
\begin{equation*}
 Q_1(\lambda) \in \Hol^{-1}\big(\Theta(P);\cL(\cH_\theta)\big)\cap \Hol^{0}\big(\Theta(P);\cL(\cH_\theta, \cH_\theta^{(w)})\big) . 
\end{equation*}
In particular, the equalities~(\ref{eq:Resolvent.P-minus-lambda-right-inverse-construction}) hold on $\cH_\theta$. Let $\lambda \in \Theta'\setminus D(0,r_0)$. Then $Q_1(\lambda)(1-R_1(\lambda))^{-1}$ is in $\cL(\cH_\theta,\cH_\theta^{(w)})$, and by using~(\ref{eq:Resolvent.P-minus-lambda-right-inverse-construction}) we see that on $\cL(\cH_\theta)$ we have
\begin{equation*}
 (P-\lambda) Q_1(\lambda)\big(1-R_1(\lambda)\big)^{-1}= (1-R_1(\lambda))\big(1-R_1(\lambda)\big)^{-1}= 1.  
\end{equation*}
That is, $Q_1(\lambda)(1-R_1(\lambda))^{-1}$ is a right inverse of $P-\lambda$. 

We can similarly construct $Q_2(\lambda)\in \Psi^{-w,-1}(\cA_\theta; \Theta(P))$ and $R_2(\lambda)\in \Psi^{-\infty,-1}(\cA_\theta; \Theta(P))$ such that $Q_2(\lambda) (P-\lambda)=1-R_2(\lambda)$. We regard $R_2(\lambda)$ as an element of $\Hol^{-1}(\Theta(P); \cL(\cH_\theta^{(w)}))$. In the same way as above, by taking $r_0$ large enough $1-R_2(\lambda)$ becomes invertible in $\cL(\cH_\theta^{(w)})$ for all $\lambda \in \Theta'\setminus D(0,r_0)$. It then can be shown that $(1-R_2(\lambda))^{-1}Q_2(\lambda)\in \cL(\cH_\theta,\cH_\theta^{(w)})$ is a left inverse of $P-\lambda$ on its domain $\cH_\theta^{(w)}$. Therefore, it must agree with $Q_1(\lambda)(1-R_1(\lambda))^{-1}$, and so $Q_1(\lambda)(1-R_1(\lambda))^{-1}$ is a bounded two-sided inverse of $P-\lambda$. 

All this shows that $\Theta'\setminus D(0,r_0)$ is contained in $\C\setminus \Sp(P)$. This implies that $\Sp(P)\neq \C$, and so by Proposition~\ref{prop:Spectral.spectrum-P} the spectrum of $P$ is an unbounded discrete subset of $\C$ consisting of eigenvalues with finite multiplicity. In particular, there are at most finitely many eigenvalues of $P$ in the disk $D(0,r_0)$. As there are no eigenvalues in $\Theta'\setminus D(0,r_0)$, we then see that $\Theta'$ may contain at most finitely many of them. 

Finally, as $Q_1(\lambda) \in \Hol^{-1}(\Theta(P);\cL(\cH_\theta))$ and $\Theta'\setminus D(0,r_0)\subsubset \Theta(P)$ there is $C_{\Theta'r_0}>0$ such that $\|Q_1(\lambda)\|\leq C_{\Theta'r_0}|\lambda|^{-1}$ for all $\lambda \in \Theta'\setminus D(0,r_0)$. Furthermore, as shown above, if $\lambda \in \Theta'\setminus D(0,r_0)$
then $(P-\lambda)^{-1}= Q_1(\lambda)(1-R_1(\lambda))^{-1}$ and $\| (1-R_1(\lambda))^{-1}\|\leq 2$. Therefore, 
for all $\lambda \in \Theta'\setminus D(0,r_0)$, we have 
\begin{equation*}
 \norm{(P-\lambda)^{-1}}= \norm{Q_1(\lambda)(1-R_1(\lambda))^{-1}} \leq \norm{Q_1(\lambda)}\norm{(1-R_1(\lambda))^{-1}}\leq 2C_{\Theta'r_0} |\lambda|^{-1}. 
\end{equation*}
This gives~(\ref{eq:Resolvent.resolvent-estimate}). The proof is complete. 
\end{proof}

\begin{definition}
 A ray $L\subset \C^*$ is called a \emph{ray of minimal growth} for $P$ when the following two conditions are satisfied:
 \begin{enumerate}
\item[(i)] $L$ does not contain any eigenvalue of $P$. 

\item[(ii)] $\|(P-\lambda)^{-1}\|=\op{O}(|\lambda|^{-1})$ as $\lambda$ goes to $\infty$ along $L$.  
\end{enumerate}
\end{definition}

\begin{remark}
 The existence of ray of minimal growths is a crucial ingredient in the construction of complex powers and logarithms of elliptic operators (see, e.g., \cite{Se:PSPM67}). 
\end{remark}

\begin{example}
 If $P$ is selfadjoint, then every ray $L\subset \C\setminus \R$ is a ray of minimal growth.
\end{example}

As immediate consequences of Theorem~\ref{thm:Resolvent.P-has-discrete-spectrum-resolvent-estimate} we get the following results. 

\begin{corollary}
For any cone  $\Theta'$ such that $\overline{\Theta'}\setminus \{0\} \subset \Theta(P)$, all but finitely many rays contained in $\Theta'$ are rays of minimal growth for $P$. 
\end{corollary}

\begin{corollary} \label{cor:Resolvent.ray-with-no-eigenvalue-is-a-ray-of-minimal-growth}
 Any ray $L\subset \Theta(P)$ that does not contain any eigenvalue of $P$ is a ray of miminal growth. 
\end{corollary}

\subsection{Analysis of the resolvent} 
In what follows we set
\begin{equation*}
\check{\Theta}(P) = \Theta(P)\setminus \{ t\lambda; \ t>0, \ \lambda\in\Sp{P} \} .
\end{equation*}
That is,  $\check{\Theta}(P)$ is obtained from $\Theta(P)$ by deleting all the rays through eigenvalues of $P$. In particular, this is a cone. Moreover, in view of Corollary~\ref{cor:Resolvent.ray-with-no-eigenvalue-is-a-ray-of-minimal-growth}, every ray contained in $\check{\Theta}(P)$ is a ray of minimal growth. 

\begin{lemma}\label{lem:Resolvent.checkThetaP-open}
 $\check{\Theta}(P)$ is a non-empty open cone in $\C^*$. 
\end{lemma}
\begin{proof}
 Let $\lambda_0\in \Theta(P)$.  As $\Theta(P)$ is an open cone, there is an open cone $\Theta'$ containing $\lambda_0$ and such that $\overline{\Theta'}\setminus \{0\}\subset \Theta(P)$. Such a cone can be obtained for instance as the cone generated by a compact neighborhood of $\lambda_0$ in $\Theta(P)$. 
  Thanks to Theorem~\ref{thm:Resolvent.P-has-discrete-spectrum-resolvent-estimate} we know that $\Theta'$ contains at most finitely many eigenvalues of $P$. Let $L_1,\ldots, L_N$ be the rays contained in $\Theta'$ that contain eigenvalues of $P$. Set $\check{\Theta}':=\Theta' \setminus (L_1\cup \cdots \cup L_N)$. Then $\check{\Theta}'$ is a non-empty open cone which is contained in $\check{\Theta}(P)$. In particular, this implies that $\check{\Theta}(P)\neq \emptyset$. Moreover, if $\lambda_0\in \check{\Theta}(P)$, then $\check{\Theta}'$ is an open neighborhood of $\lambda_0$ contained in $\check{\Theta}(P)$. This shows that $\check{\Theta}(P)$ is a neighborhood of each of its point, and hence this is an open set. The proof is complete.  
\end{proof}

\begin{definition} \label{def:Resolvent.Agmon-pseudo-cone}
Set $R_0=\op{dist}(0,\Sp(P)\cap \C^*)$. The pseudo-cone $\Lambda(P)$ is defined by
\begin{equation*}
 \Lambda(P) = \left\{
\begin{array}{cl}
 \check{\Theta}(P)\cup D(0,R_0) & \text{if $0\not\in \Sp(P)$},\medskip \\
 \check{\Theta}(P)\cup \big[D(0,R_0)\setminus \{0\}\big]  & \text{if $0\in \Sp(P)$}.
\end{array}\right. 
\end{equation*}
\end{definition}

\begin{proposition} \label{prop:Resolvent.resolvent-gives-rise-to-bounded-operators}
$(P-\lambda)^{-1}$ is contained in $\Hol^{-1}(\Lambda(P); \cL(\cH_\theta))$. 
\end{proposition}
\begin{proof}
 As $\Lambda(P)\subset \C\setminus \Sp(P)$, we have a holomorphic map $\lambda \rightarrow P-\lambda$ from $\Lambda(P)$ to the invertible elements of $\cL(\cH_\theta^{(w)},\cH_\theta)$. By taking inverses we then get a holomorphic map from $\Lambda(P)$ to $\cL(\cH_\theta,\cH_\theta^{(w)})$, i.e.,
 $(P-\lambda)^{-1}\in \Hol(\Lambda(P); \cL(\cH_\theta))$. 
 
Let $\Lambda'$ be a pseudo-cone such that $\Lambda'\subsubset \Lambda(P)$. Denote by $\Theta'$ its conical component. Then $\overline{\Theta'}\setminus \{0\}$ is contained in $\check{\Theta}(P)\subset \Theta(P)$, and so by Theorem~\ref{thm:Resolvent.P-has-discrete-spectrum-resolvent-estimate} the estimate~(\ref{eq:Resolvent.resolvent-estimate}) holds on $\Theta(P)\setminus D(0,r_0)$ for $r_0$ large enough. Note that $\Lambda' \setminus [\Theta(P)\setminus D(0,r_0)]$ is pre-compact in $\Lambda(P)$, since this is a bounded set whose closure is contained in $\Lambda(P)$. It then follows there is $C_{\Lambda'}>0$ such that
\begin{equation*}
 \big\| (P-\lambda)^{-1}\big \| \leq C_{\Lambda'}(1+|\lambda|)^{-1} \qquad \forall \lambda \in \Lambda'. 
\end{equation*}
This shows that $(P-\lambda)^{-1}\in \Hol^{-1}(\Lambda(P); \cL(\cH_\theta))$. The proof is complete. 
\end{proof}

We are now in a position to prove the main result of this paper. 

\begin{theorem} \label{thm:Resolvent.resolvent-is-psido-with-parameter}
Suppose that $P$ is elliptic with parameter. Then the following holds. 
\begin{enumerate}
 \item The resolvent $(P-\lambda)^{-1}$ is contained in $\Psi^{-w,-1}(\cA_\theta;\Lambda(P))$. 

\item $(P-\lambda)^{-1}$ has symbol $\sigma(\xi;\lambda)\sim \sum \sigma_{-w-j}(\xi;\lambda)$, where $\sigma_{-w-j}(\xi;\lambda)$ is given by~(\ref{eq:Resolvent.parametrix-principal-symbol})--(\ref{eq:Resolvent.parameterix-homogeneous-symbols}). 
\end{enumerate}
\end{theorem}
\begin{proof}
By construction $\Lambda(P)\subset \Theta(P)\cup D(0,R_0)$. Therefore, by Theorem~\ref{thm:Resolvent.parametrix-with-parameter} there is $Q(\lambda)$ in $\Psi^{-w,-1}(\cA_\theta; \Lambda(P))$ such that
\begin{equation} \label{eq:Resolvent.parametrix-with-parameter-R1-and-R2}
 Q(\lambda)(P-\lambda)=1-R_1(\lambda) \qquad \text{and} \qquad (P-\lambda)Q(\lambda)=1-R_2(\lambda),
\end{equation}
where $R_1(\lambda)$ and $R_2(\lambda)$ are in $\Psi^{-\infty,0}(\cA_\theta;\Lambda(P))$. Multiplying the first equality by $(P-\lambda)^{-1}$ gives $Q(\lambda)= (1-R_1(\lambda))(P-\lambda)^{-1}$, i.e., 
\begin{equation} \label{eq:Resolvent.inverse-of-P-lambda-written-by-parametrix-and-R_1}
(P-\lambda)^{-1} = Q(\lambda) + R_1(\lambda)(P-\lambda)^{-1} .
\end{equation}
Proposition~\ref{prop:PsiDOs-parameter.PsiDO-gives-rise-to-family-of-continuous-operators-on-cAtheta} and Proposition~\ref{prop:PsiDOs-parameter.smoothing-operators-with-parameter-characterization} ensure us that $Q(\lambda)$ and $R_1(\lambda)$ are in $\Hol^{-1}(\Lambda(P); \cL(\cA_\theta))$ and $\Hol^0(\Lambda(P); \cL(\cA_\theta',\cA_\theta))$, respectively. Moreover, thanks to Proposition~\ref{prop:Resolvent.resolvent-gives-rise-to-bounded-operators} we also know that $(P-\lambda)^{-1}$ is contained in $\Hol^{-1}(\Lambda(P); \cL(\cH_\theta))$, and so it can be regarded as an element of $\Hol^{-1}(\Lambda(P);  \cL(\cA_\theta,\cA_\theta'))$. Therefore, the composition $R_1(\lambda)(P-\lambda)^{-1}$ gives rise to an element of  $\Hol^{-1}(\Lambda(P); \cL(\cA_\theta))$. Combining all this with~(\ref{eq:Resolvent.inverse-of-P-lambda-written-by-parametrix-and-R_1}) we deduce that $(P-\lambda)^{-1}$ is contained in $\Hol^{-1}(\Lambda(P); \cL(\cA_\theta))$.

Multiplying by $(P-\lambda)^{-1}$ the 2nd equality in~(\ref{eq:Resolvent.parametrix-with-parameter-R1-and-R2}) gives $Q(\lambda)=(P-\lambda)^{-1} (1-R_2(\lambda))$. That is, 
\begin{equation} \label{eq:Resolvent.inverse-of-P-lambda-written-by-parametrix-and-R_2}
(P-\lambda)^{-1} = Q(\lambda) + (P-\lambda)^{-1}R_2(\lambda) .
\end{equation}
As $(P-\lambda)^{-1}\in\Hol^{-1}(\Lambda(P); \cL(\cA_\theta))$ and $R_2(\lambda)$ is in $\Hol^0(\Lambda(P); \cL(\cA_\theta',\cA_\theta))$ by Proposition~\ref{prop:PsiDOs-parameter.smoothing-operators-with-parameter-characterization} the composition $(P-\lambda)^{-1}R_2(\lambda) $ is in $\Hol^{-1}(\Lambda(P); \cL(\cA_\theta',\cA_\theta))$. Therefore, by Proposition~\ref{prop:PsiDOs-parameter.smoothing-operators-with-parameter-characterization} this is an element of $\Psi^{-\infty,-1}(\cA_\theta; \Lambda(P))$. Combining this with~(\ref{eq:Resolvent.inverse-of-P-lambda-written-by-parametrix-and-R_2}) then shows that $(P-\lambda)^{-1}$ is contained in $\Psi^{-w,-1}(\cA_\theta;\Lambda(P))$. This proves the 1st part. The 2nd part follows from the 2nd part of Theorem~\ref{thm:Resolvent.parametrix-with-parameter} and the fact that $(P-\lambda)^{-1}$ is a parametrix. The proof is complete. 
\end{proof}

\begin{remark}
 Suppose that $0\in \Sp(P)$. In this case $0\not \in \Lambda(P)$, and so Theorem~\ref{thm:Resolvent.resolvent-is-psido-with-parameter} asserts that $(P-\lambda)^{-1}$ is a $\Hol^{-1}$-family of \psidos\ off $\lambda=0$. By using analytic Fredholm theory it can be shown that at $\lambda=0$ we actually have a meromorphic singularity in $\cL(\cH_\theta)$ with finite rank poles (see, e.g., \cite[Appendix~C]{DZ:scattering}). It can be actually shown that the poles are smoothing operators and by removing the meromorphic singularity we actually get a $\Hol^{-1}$-family of \psidos\ near $\lambda=0$ (see~\cite{LP:Powers}). 
\end{remark}

We mention a few consequences of Theorem~\ref{thm:Resolvent.resolvent-is-psido-with-parameter}. First, we have the following refinement of Proposition~\ref{prop:PsiDOs-parameter.classical-PsiDO-Sobolev-mapping-properties}. 

\begin{proposition}\label{prop:Resolvent.Sobolev} 
 For every $s\in \R$, we have
\begin{equation*}
 (P-\lambda)^{-1} \in \bigcap_{0\leq a \leq 1} \Hol^{-1+a}\left(\Lambda(P); \cL\big(\cH_\theta^{(s)}, \cH_\theta^{(s+aw)}\big)\right). 
\end{equation*}
\end{proposition}
\begin{proof}
 We know by Theorem~\ref{thm:Resolvent.resolvent-is-psido-with-parameter} that $(P-\lambda)^{-1} \in\Psi^{-w,-1}(\cA_\theta; \Lambda(P))$. Therefore, by using Proposition~\ref{prop:PsiDOs-parameter.classical-PsiDO-Sobolev-mapping-properties} we see that, for every $s\in \R$, we have 
\begin{equation*}
 (P-\lambda)^{-1} \in \bigcap_{-1\leq d' \leq 0} \Hol^{d'}\left(\Lambda(P); \cL\big(\cH_\theta^{(s-w+w|d'|)}, \cH_\theta^{(s)}\big)\right).
\end{equation*}
Substituting $-1+a$ for $d'$ and $s+aw$ for $s$ with $a\in [0,1]$ gives the result. 
\end{proof}

In another direction we have the following Schatten class properties of the resolvent. 
\begin{proposition}\label{prop:Resolvent.Schatten}
 Set $p=nw^{-1}$. Then 
\begin{equation*}
 (P-\lambda)^{-1} \in \bigcap_{q\geq p} \Hol^{-1+pq^{-1}}\Big(\Lambda(P);\cL^{(q,\infty)}\Big). 
\end{equation*}
\end{proposition}
\begin{proof}
 We know by Theorem~\ref{thm:Resolvent.resolvent-is-psido-with-parameter} that $(P-\lambda)^{-1} \in\Psi^{-w,-1}(\cA_\theta; \Lambda(P))$. We thus are in the setting of Proposition~\ref{prop:PsiDOs-parameter.Schatten-classical}  with $d=-1$ and $m=-w$. Here $mw^{-1}=-1=d$, and so~(\ref{eq:PsiDOs-parameter.classical2}) gives
\begin{equation*}
 (P-\lambda)^{-1} \in \bigcap_{q\geq p} \Hol^{d(q)}\left(\Lambda(P);\cL^{(q,\infty)}\right).
\end{equation*}
 where $d(q)=(-w+nq^{-1})w^{-1}=-1+(nw^{-1})q^{-1}=-1+pq^{-1}$. This proves the result. 
\end{proof}

One approach to derive spectral asymptotics is to use resolvent-type expansions. This often leads us to estimate traces of operators of the form, 
\begin{equation*}
 Q(A_0,\ldots, A_N)(\lambda)= A_0(P-\lambda)^{-1}A_1 \cdots A_{N-1}(P-\lambda)^{-1}A_N, \qquad  A_j\in \Psi^{a_j}(\cA_\theta). 
\end{equation*}
For instance, when $A_1=\cdots A_N=1$ we get the operator $A_0(P-\lambda)^{-N}$. Note that by Proposition~\ref{prop:Parameter.composition-PsiDOs} and Theorem~\ref{thm:Resolvent.resolvent-is-psido-with-parameter}  $Q(A_0,\ldots, A_N)(\lambda)$ is contained in $\Psi^{m,-N}(\cA_\theta; \Lambda(P))$ where $m=a-Nw$ and we have set $a=a_0+\cdots +a_N$. In particular, we obtain a family of trace-class operators when $a-Nw<-n$. 

\begin{proposition}\label{prop:Resolvent.trace-Q}
Let $A_j\in \Psi^{a_j}(\cA_\theta)$, $a_j\in \R$, $j=0,\ldots, N$. Set $a=a_0+\cdots +a_N$, and assume that $-Nw+a<-n$.  
\begin{enumerate}
 \item If $a<-n$, then $Q(A_0,\ldots, A_N)(\lambda)\in \Hol^{-N}(\Lambda(P);\cL^1)$ and, for any pseudo-cone $\Lambda \subsubset \Lambda(P)$, there is $C_{\Lambda}>0$ such that 
          \begin{equation*}
             \Big| \Tra\big[Q(A_0,\ldots, A_N)(\lambda)\big] \Big|\leq  C_{\Lambda} (1+|\lambda|)^{-N} \qquad \forall \lambda \in \Lambda.  
         \end{equation*}

\item If $a\geq -n$ and $\delta>(a+n)w^{-1}$, then $Q(A_0,\ldots, A_N)(\lambda)\in \Hol^{-N+\delta}(\Lambda(P);\cL^1)$ and, for any pseudo-cone $\Lambda \subsubset \Lambda(P)$, there is $C_{\Lambda\delta}>0$ such that 
          \begin{equation*}
             \Big| \Tra\big[Q(A_0,\ldots, A_N)(\lambda)\big] \Big|\leq  C_{\Lambda \delta} (1+|\lambda|)^{-N+\delta} \qquad \forall \lambda \in \Lambda.  
             \end{equation*}
\end{enumerate}
\end{proposition}
\begin{proof}
As mentioned above $Q(A_0,\ldots, A_N)(\lambda)\in \Psi^{m,-N}(\cA_\theta;\Lambda(P))$ with $m=a-Nw$. We thus are in the framework of Proposition~\ref{prop:PsiDOs-parameter.trace-class} with $d=-N$. Here $(m+n)w^{-1}=(a-Nw+n)w^{-1}=-N+(a+n)w^{-1}$, and so the condition $d>(m+n)w^{-1}$ becomes $-N>-N+(a+n)w^{-1}$, i.e., $a<-n$. The result then follows from Proposition~\ref{prop:PsiDOs-parameter.trace-class}. 
\end{proof}


\section{Complex Powers of Positive Elliptic Operators}\label{sec:square-roots}
In this section, we use the results of the previous section to show that the complex powers of positive elliptic \psidos\ are \psidos. We refer to~\cite{LP:Powers} for a more comprehensive account on complex powers of elliptic \psidos\ on noncommutative tori. 

In what follows, we let $P\in \Psi^{w}(\cA_\theta)$ be an elliptic \psido\ with symbol $\rho(\xi)\sim \sum_{j\geq 0} \rho_{w-j}(\xi)$. We make the following positivity assumptions: 
\begin{itemize}
 \item[(P1)]  The principal symbol $\rho_{w}(\xi)$ is positive, i.e., $\rho_{w}(\xi)$ is selfadjoint and has positive spectrum for all $\xi\neq 0$. 
 
 \item[(P2)] $P$ is formally selfadjoint and has non-negative spectrum. 
\end{itemize}
Note that we do not assume $P$ to be a differential operator. The ellipticity of $P$ and (P2) ensure us that $P$ with domain $\cH_\theta^{(w)}$ is a selfadjoint Fredholm operator on $\cH_\theta$ with non-negative (discrete) spectrum. In addition, (P1) and (P2) imply that $\check{\Theta}(P)=\Theta(P)=\C\setminus [0,\infty)$, and so we have 
\begin{equation*}
 \Lambda(P)=\C\setminus [r_0,\infty), \qquad \text{where}\ r_0:=\dist\big(0,\Sp(P)\setminus 0\big)=\|P^{-1}\|^{-1}. 
\end{equation*}

Given $z\in \C$ and $\lambda\in \C\setminus (-\infty,0]$, we put $\lambda^z=e^{z \log \lambda}$, where $\lambda \rightarrow \log \lambda$ is the continuous determination of the logarithm on $\C\setminus (-\infty,0]$ with values in the stripe $\{|\Re \lambda|<\pi\}$. In particular, the power functions $\lambda \rightarrow\lambda^z$ are holomorphic functions on $\C\setminus (-\infty,0]$. Moreover, they agree with the standard power functions on $(0,\infty)$. 

For $z\in \C$ we define the operator $P^z$ by standard Borel functional calculus, with the convention that $P^z=0$ on $\ker P$ (i.e.,  
 $\lambda^z=0$ for $\lambda=0$). Alternatively, if $(e_\ell)_{\ell \geq 0}$ is any orthonormal eigenbasis of $P$ with $Pe_\ell=\lambda_\ell e_\ell$, then we have 
\begin{equation*}
 P^ze_\ell = \lambda_\ell^z e_\ell \qquad \forall \ell\geq 0. 
\end{equation*}
For $\Re z \leq 0$ we obtain a bounded operator on $\cH_\theta$. For $\Re z>0$ we get a closed unbounded operator with domain $\cH^{(w\Re z)}_\theta$. In particular, the domain of $P^z$ always contains $\cA_\theta$. We also have the group property, 
\begin{equation}
 P^{z_1+z_2}u=P^{z_1} P^{z_2}u \qquad \forall u\in \cA_\theta \quad  \forall z_i\in \C.
 \label{eq:Powers.group} 
\end{equation}

In addition, for $\Re z<0$ we have the integral formula, 
\begin{equation*}
 P^z = \frac1{2i\pi}\int_\Gamma \lambda^z (P-\lambda)^{-1}d\lambda, 
\end{equation*}
where the integral converges in $\cL(\cH_\theta)$ and $\Gamma$ is a downward-oriented keyhole contour of the form $\Gamma =\partial \Lambda(r)$, with
\begin{equation*}
 \Lambda(r):= \left\{\lambda \in \C^*; \ \Re \lambda\leq 0 \ \text{or} \ |\lambda|\leq r\right\}, \qquad 0<r<r_0. 
\end{equation*}
Note that $\Lambda(r)$ is a pseudo-cone and $\Lambda(r)\subsubset \Lambda(P)$.

\begin{figure}[h]
\centering{\def\svgwidth{0.35\columnwidth}
\begingroup%
  \makeatletter%
  \providecommand\color[2][]{%
    \errmessage{(Inkscape) Color is used for the text in Inkscape, but the package 'color.sty' is not loaded}%
    \renewcommand\color[2][]{}%
  }%
  \providecommand\transparent[1]{%
    \errmessage{(Inkscape) Transparency is used (non-zero) for the text in Inkscape, but the package 'transparent.sty' is not loaded}%
    \renewcommand\transparent[1]{}%
  }%
  \providecommand\rotatebox[2]{#2}%
  \newcommand*\fsize{\dimexpr\f@size pt\relax}%
  \newcommand*\lineheight[1]{\fontsize{\fsize}{#1\fsize}\selectfont}%
  \ifx\svgwidth\undefined%
    \setlength{\unitlength}{283.46456693bp}%
    \ifx\svgscale\undefined%
      \relax%
    \else%
      \setlength{\unitlength}{\unitlength * \real{\svgscale}}%
    \fi%
  \else%
    \setlength{\unitlength}{\svgwidth}%
  \fi%
  \global\let\svgwidth\undefined%
  \global\let\svgscale\undefined%
  \makeatother%
  \begin{picture}(1,1)%
    \lineheight{1}%
    \setlength\tabcolsep{0pt}%
    \put(0.55088001,0.54775318){\color[rgb]{0,0,0}\makebox(0,0)[lt]{\lineheight{1.25}\smash{\begin{tabular}[t]{l}$r$\end{tabular}}}}%
    \put(0,0){\includegraphics[width=\unitlength,page=1]{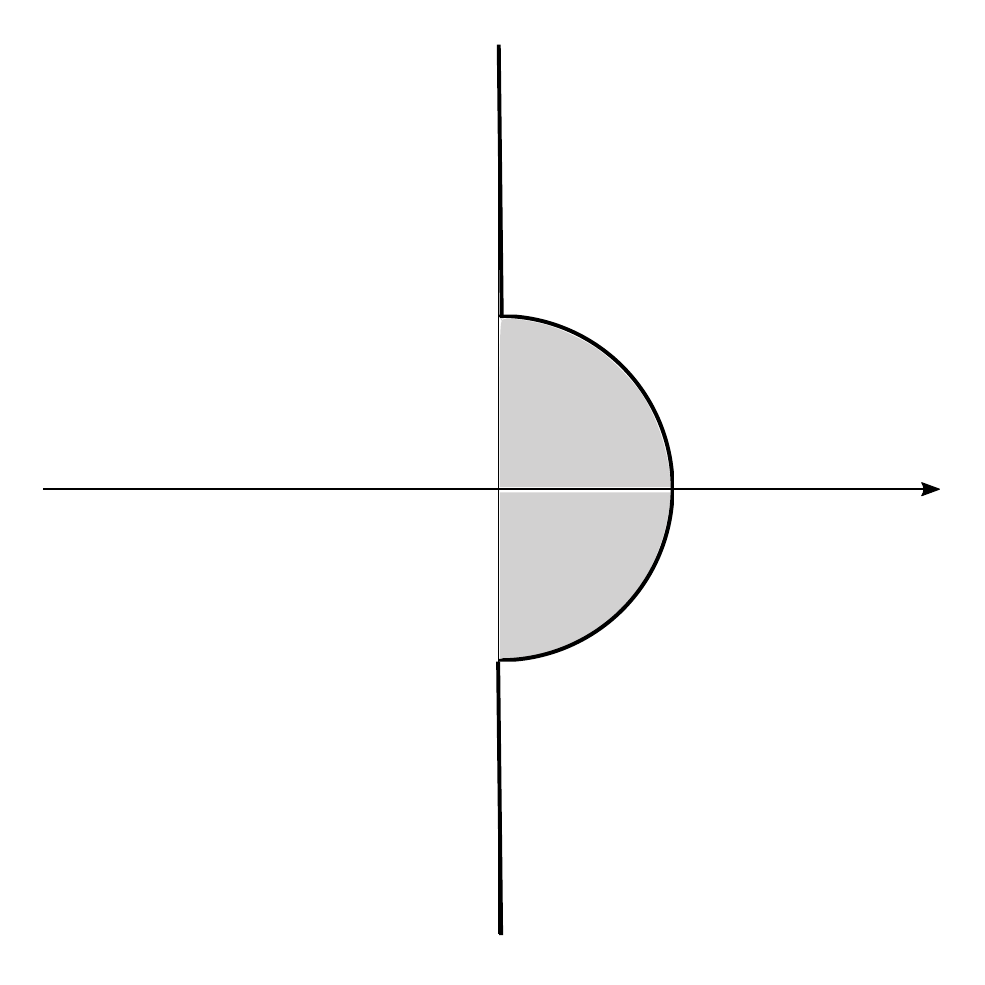}}%
    \put(0.14627678,0.60841667){\color[rgb]{0,0,0}\makebox(0,0)[lt]{\lineheight{1.25}\smash{\begin{tabular}[t]{l}$\Lambda(r)$\end{tabular}}}}%
    \put(0,0){\includegraphics[width=\unitlength,page=2]{lambdar2.pdf}}%
    \put(0.52614292,0.89983631){\color[rgb]{0,0,0}\makebox(0,0)[lt]{\lineheight{1.25}\smash{\begin{tabular}[t]{l}$\Gamma$\end{tabular}}}}%
  \end{picture}%
\endgroup%
}
\caption{Pseudo-Cone $\Lambda(r)$ and Contour $\Gamma=\partial\Lambda(r)$}
\end{figure} 

The main aim of this section is to show that the powers $P^z$ are \psidos. To reach this end we need a couple of lemmas. 

\begin{lemma}\label{lem:Powers.integration-standard-symbols}
Suppose that $\Lambda$ is an open pseudo-cone containing $\Gamma$. Let $\sigma(\xi; \lambda)\in \stS^{m,d}(\R^n\times \Lambda;\cA_\theta)$ with $m\in \R$ and $d<-1$. 
\begin{enumerate}
 \item[(i)] The integral $\rho(\xi)=\int_\Gamma \sigma(\xi;\lambda)d\lambda$ converges in $\stS^{m}(\R^n;\cA_\theta)$. 
 
 \item[(ii)] We have
                 \begin{equation*}
                      \int_\Gamma  P_{\sigma}(\lambda) d\lambda= P_{\rho},
                \end{equation*}
                where the integral converges in $\cL(\cA_\theta)$. 
\end{enumerate}
\end{lemma}
\begin{proof}
The first part follows from the 1st part of Proposition~\ref{prop:AppendixA.Hold} and the fact that $\stS^{m}(\R^n;\cA_\theta)$ is a Fr\'echet space. The 2nd part follows from the 2nd part of Proposition~\ref{prop:AppendixA.Hold} and the continuity of the linear map $\stS^m(\R^n;\cA_\theta)\ni \rho \rightarrow P_\rho \in \cL(\cA_\theta)$. This completes the proof. 
\end{proof}

Define 
\begin{equation*}
 c=\inf_{|\xi|=1} \left\|\rho_w(\xi)^{-1}\right\|^{-1} \quad \text{and} \quad c'=\sup_{|\xi|=1} \left\|\rho_w(\xi)\right\|. 
\end{equation*}
The ellipticity of $P$ ensures us that $c>0$. Moreover, as $\rho_w(\xi)$ has positive spectrum, its spectrum is contained in the interval 
$[\|\rho_w(\xi)^{-1}\|^{-1}, \|\rho_w(\xi)\|]$ for every $\xi\neq 0$. As by homogeneity $\rho_w(\xi)=|\xi|^w\rho_w(|\xi|^{-1}\xi)$, we deduce that
\begin{equation}
 \Sp\left( \rho_w(\xi)\right) \subset \left[c|\xi|^w,c'|\xi|^w\right]\qquad \forall \xi\in \R^n\setminus 0. 
 \label{eq:Powers.spectrum-rhow}
\end{equation}
In what follows we define $\Omega_c(P)$ as in~(\ref{eq:Resolvent.OmegacP}). Namely, as $\Theta(P)=\C\setminus [0,\infty)$, we have
\begin{equation*}
 \Omega_c(P)=\left\{ (\xi,\lambda)\in (\R^n\setminus 0)\times \C; \ \lambda\in \C\setminus [c|\xi|^w,\infty)\right\}. 
\end{equation*}

\begin{lemma}\label{lem:Powers.integration-homogeneous-symbols}
 Let $z\in \C$, $\Re z <0$, and $\sigma(\xi; \lambda) \in S_{m}^{-1}(\Omega_c(P); \cA_\theta)$, $m\in \R$. In addition, let $\Gamma$ be a contour of the form $\Gamma=\partial \Lambda(r)$ with $0<r<c$. Then the integral 
\begin{equation*}
 \rho(\xi):= \int_{|\xi|^w\Gamma} \lambda^z \sigma(\xi;\lambda)d\lambda, \qquad \xi \neq 0, 
\end{equation*}
 converges in $\cA_\theta$ and defines a homogeneous symbol in $S_{m+w(z+1)}(\R^n;\cA_\theta)$.
\end{lemma}
\begin{proof}
As $\sigma(\xi;\lambda)$ is in $S^{-1}_m(\Omega_c(P);\cA_\theta)$ this is a $C^{\infty,\omega}$-map from $\Omega_c(P)$ to $\cA_\theta$. Set 
\begin{equation*}
 \Theta=\bigg\{\lambda\in \C^*; \frac{\pi}4<|\arg \lambda|<\frac{3\pi}{4}\bigg\}.
\end{equation*}
We get an open cone containing $i\R\setminus 0$ such that $\overline{\Theta}\setminus 0$ is contained in $\C\setminus [0,\infty)=\Theta(P)$. Thus, by the very definition of the space $S_{m}^{-1}(\Omega_c(P);\cA_\theta)$, for any compact $K\subset \R^n\setminus 0$ and multi-orders $\alpha$ and $\beta$, there is $C_{K\alpha\beta}>0$ such that 
\begin{equation}
 \big\| \delta^\alpha \partial_\xi^\beta \sigma(\xi;\lambda) \big\| \leq C_{K\alpha \beta} \big(1+|\lambda|\big)^{-1} \qquad \forall (\xi,\lambda)\in K\times \Theta.
 \label{eq:Powers.estimates-sigma-homogeneous}  
\end{equation}

For $s>0$ we also set
\begin{equation*}
 \Lambda'(s) := \Theta \cup \bigg\{\lambda \in \C^*; |\arg \lambda|\leq \frac{\pi}{4}\ \text{and}\ |\lambda|<cs^w\bigg\}.  
\end{equation*}
Thus, $\Lambda'(s)$ is an open pseudo-cone with conical part  $\Theta$. Note that $\lambda^z \in \Hol^{\Re z}(\Lambda(s))$. We also observe that the contour  $|\xi|^w\Gamma$ is contained in $\Lambda'(s)$ for $s\geq |\xi|$. Furthermore, if $s\leq |\xi|$ and $\lambda \in \C^*$ is such that $|\lambda|<cs^w$, then $|\lambda|<c|\xi|^w$, and so $(\xi,\lambda)\in \Omega_c(P)$. Thus, 
\begin{equation*}
 \{\xi\} \times \Lambda'(s) \subset \Omega_c(P) \qquad \text{for all $s\in (0,|\xi|]$}. 
\end{equation*}

\begin{figure}[h]
\centering{\def\svgwidth{0.35\columnwidth}
\begingroup%
  \makeatletter%
  \providecommand\color[2][]{%
    \errmessage{(Inkscape) Color is used for the text in Inkscape, but the package 'color.sty' is not loaded}%
    \renewcommand\color[2][]{}%
  }%
  \providecommand\transparent[1]{%
    \errmessage{(Inkscape) Transparency is used (non-zero) for the text in Inkscape, but the package 'transparent.sty' is not loaded}%
    \renewcommand\transparent[1]{}%
  }%
  \providecommand\rotatebox[2]{#2}%
  \newcommand*\fsize{\dimexpr\f@size pt\relax}%
  \newcommand*\lineheight[1]{\fontsize{\fsize}{#1\fsize}\selectfont}%
  \ifx\svgwidth\undefined%
    \setlength{\unitlength}{283.46456693bp}%
    \ifx\svgscale\undefined%
      \relax%
    \else%
      \setlength{\unitlength}{\unitlength * \real{\svgscale}}%
    \fi%
  \else%
    \setlength{\unitlength}{\svgwidth}%
  \fi%
  \global\let\svgwidth\undefined%
  \global\let\svgscale\undefined%
  \makeatother%
  \begin{picture}(1,1)%
    \lineheight{1}%
    \setlength\tabcolsep{0pt}%
    \put(0,0){\includegraphics[width=\unitlength,page=1]{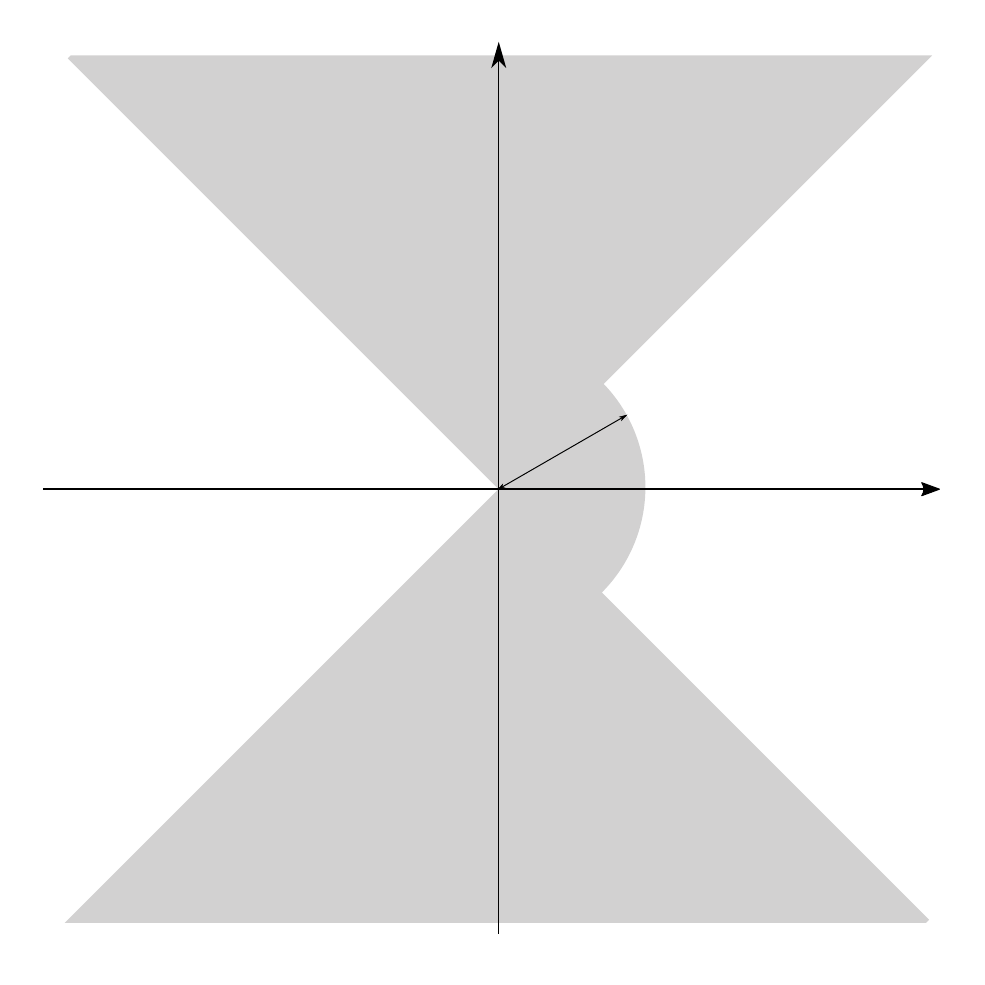}}%
    \put(0.55378361,0.55313503){\color[rgb]{0,0,0}\makebox(0,0)[lt]{\lineheight{1.25}\smash{\begin{tabular}[t]{l}$cs^w$\end{tabular}}}}%
    \put(0,0){\includegraphics[width=\unitlength,page=2]{lambdaprimes.pdf}}%
    \put(0.60854172,0.83747025){\color[rgb]{0,0,0}\makebox(0,0)[lt]{\lineheight{1.25}\smash{\begin{tabular}[t]{l}$\Lambda'(s)$\end{tabular}}}}%
  \end{picture}%
\endgroup%
}
\caption{Pseudo-Cone $\Lambda'(s)$}
\end{figure}

Let $\xi\in \R^n\setminus 0$. As $\{\xi\}\times \Lambda'(|\xi|)\subset \Omega_c(P)$ the map $\lambda \rightarrow \sigma(\xi;\lambda)$ is  holomorphic map from $\Lambda'(|\xi|)$ to $\cA_\theta$. Combining this with the estimates~(\ref{eq:Powers.estimates-sigma-homogeneous}) for  $K=\{\xi\}$  and $\beta=0$ we see that this is a $\Hol^{-1}(\Lambda(|\xi|))$-map, and so the map $ \lambda \rightarrow \lambda^z\sigma(\xi;\lambda)$ is in $\Hol^{-1+\Re z}(\Lambda'(|\xi|);\cA_\theta)$. Recall that $|\xi|^w \Gamma \subset \Lambda'(|\xi|)$ and $\cA_\theta$ is a Fr\'echet space. Therefore, by the 1st~part of Proposition~\ref{prop:AppendixA.Hold}, for every $\xi\neq 0$, the integral $\rho(z;\xi):=  \int_{|\xi|^w\Gamma} \lambda^z \sigma(\xi;\lambda)d\lambda$ converges in $\cA_\theta$. Furthermore, by the 3rd~part of Proposition~\ref{prop:AppendixA.Hold} we may replace 
the contour $|\xi|^w\Gamma$ by $s^w\Gamma$ with any $s\in(0,|\xi|]$.

Let $R>0$ and set $U_{R}=\{\xi\in \R^n; \ |\xi|>R\}$. If $\xi\in U_{R}$, then $R^w<|\xi|^w$, and so we have 
\begin{equation*}
 \rho(\xi) =  \int_{R^w\Gamma} \lambda^z \sigma(\xi;\lambda)d\lambda \qquad \forall \xi \in U_{R}.
\end{equation*}
In addition, if $0<|\lambda|<cR^w$, then $|\lambda|<c|\xi|^w$, and so $(\xi,\lambda)\in \Omega_c(P)$. Therefore, we see that $U_R\times \Lambda'(R)\subset \Omega_c(P)$. 
Combining this with the estimates~(\ref{eq:Powers.estimates-sigma-homogeneous}) shows that $\sigma(\xi;\lambda)\in C^{\infty,-1}(U_R\times \Lambda'(R);\cA_\theta)$, and hence $\lambda^z\sigma(\xi;\lambda)\in C^{\infty,-1+\Re z}(U_R\times \Lambda'(R);\cA_\theta)$. As $C^\infty(U_{R};\cA_\theta)$ is a Fr\'echet space (see, e.g., \cite[Proposition~C.23]{HLP:Part1}), we may apply 1st part of Proposition~\ref{prop:AppendixA.Hold} to see that the integral $\int_{R^w\Gamma} \lambda^z \sigma(\xi;\lambda)d\lambda$ converges in $C^\infty(U_{R}; \cA_\theta)$. It then follows that $\rho(\xi)$ is a $C^\infty$-map from $U_{R}$ to $\cA_\theta$ for every $R>0$, and hence it lies in $C^\infty(\R^n\setminus 0; \cA_\theta)$. 

Let us show that $\rho(\xi)$ is homogeneous. Let $\xi\in \R^n\setminus 0$ and $t>0$. By Proposition~\ref{prop:Integrals.scaling-contour} we have
\begin{equation*}
  \rho(t\xi)= \int_{t^w|\xi|^w\Gamma} \lambda^z \sigma(t\xi;\lambda)d\lambda =  \int_{|\xi|^w\Gamma}t^{wz+w} \lambda^z \sigma(t\xi;t^w\lambda)d\lambda. 
\end{equation*}
As $\sigma(t\xi;t^w\lambda)=t^m \sigma(\xi;\lambda)$ we obtain
\begin{equation*}
  \rho(t\xi)=  t^{wz+w+m}\int_{|\xi|^w\Gamma}\lambda^z \sigma(\xi;\lambda)d\lambda= t^{m+w(z+1)} \rho(\xi). 
\end{equation*}
This shows that $\rho(\xi)$ is homogeneous of degree $m+w(z+1)$, and so $\rho(\xi)\in S_{m+w(z+1)}(\R^n; \cA_\theta)$. The proof is complete. 
\end{proof}

It follows from Theorem~\ref{thm:Resolvent.resolvent-is-psido-with-parameter} that $(P-\lambda)^{-1}$ is in $\Psi^{-w,-1}(\cA_\theta; \Lambda(P))$ and has symbol $\sigma(\xi; \lambda) \sim \sum_{j\geq 0} \sigma_{-w-j}(\xi;\lambda)$, where 
\begin{gather}
 \sigma_{-w}(\xi;\lambda)= \big(\rho_{w}(\xi)-\lambda\big)^{-1},
 \label{eq:Powers.symbol-resolvent1}\\
 \sigma_{-w-j}(\xi;\lambda) = -\sum_{\substack{k+l+|\alpha|=j \\ l<j}}\frac{1}{\alpha !}\big(\rho_w(\xi)-\lambda\big)^{-1} \partial_\xi^\alpha \rho_{w-k}(\xi) \delta^\alpha\sigma_{-w-l}(\xi;\lambda), \qquad j\geq 1 .
  \label{eq:Powers.symbol-resolvent2}
\end{gather}
For every $j\geq 0$ the symbol $\sigma_{-w-j}(\xi;\lambda)$ is in $S_{-w-j}^{-1}(\Omega_c(P); \cA_\theta)$, where $\Omega_c(P)$ is defined as above. 

We are now in a position to state and prove the main result of this section. 

\begin{theorem}\label{thm:Powers.powers}
Assume that \textup{(P1)}  and \textup{(P2)} hold. For every $z\in \C$, the operator $P^z$ is in $\Psi^{wz}(\cA_\theta)$ and has symbol $\rho(z;\xi)\sim \sum_{j\geq 0} \rho_{wz-j}(z;\xi)$, with 
\begin{equation}
 \rho_{wz-j}(z;\xi) :=\frac{1}{2i\pi} \int_{\gamma_{\xi}} \lambda^{z} \sigma_{-w-j}(\xi; \lambda)d\lambda,\qquad \xi\neq 0, 
  \label{eq:Powers.homogeneous-symbols}
\end{equation}
where $\sigma_{-w-j}(\xi;\lambda)$ is given by~(\ref{eq:Powers.symbol-resolvent1})--(\ref{eq:Powers.symbol-resolvent2}) and $\gamma_{\xi}$ is any clockwise-oriented Jordan curve in $\C\setminus (-\infty,0]$ that encloses $\Sp(\rho_w(\xi))$. In particular, the principal symbol of $P^{z}$ is $(\rho_{w}(\xi))^{z}$.  
\end{theorem}
\begin{proof}
 Let $z\in \C$, $\Re z<0$. As mentioned above $(P-\lambda)^{-1}=P_{\sigma}(\lambda)$ with $\sigma(\xi; \lambda) \in S^{-w,-1}(\R^n\times \Lambda(P); \cA_\theta)$ such that $\sigma(\xi;\lambda)\sim \sum \sigma_{-w-j}(\xi;\lambda)$, where $\sigma_{-w-j}(\xi;\lambda)$ is given by~(\ref{eq:Powers.symbol-resolvent1})--(\ref{eq:Powers.symbol-resolvent2}). In particular, by using Proposition~\ref{symbols:inclusion-classical-standard} we see that $\sigma(\xi; \lambda)\in \stS^{0,-1}(\R^n\times \Lambda(P); \cA_\theta)$.   
It then follows from Lemma~\ref{lem:Powers.integration-standard-symbols} that the integral $\rho(z;\xi)=(2i\pi)^{-1}\int_\Gamma \lambda^{z} \sigma(\xi;\lambda) d\lambda$ converges in $\stS^{0}(\R^n;\cA_\theta)$, and on $\cA_\theta$ we have 
\begin{equation*}
 P^{z}= \frac{1}{2i\pi} \int_\Gamma \lambda^{z} \frac{d\lambda}{P-\lambda}= \frac1{2i\pi} \int_\Gamma \lambda^{z} P_\sigma(\lambda)d\lambda= P_{\rho(z;\cdot)}, 
\end{equation*}
where the 2nd integral converges in $\cL(\cA_\theta)$. 

Without any loss of generality we may assume that $\Gamma=\partial\Lambda(r)$ with $0<r<\min(r_0,c)$. Given any $j\geq 0$, as $\sigma_{-w-j}(\xi;\lambda)$ is contained in $S_{-w-j}^{-1}(\Omega_c(P); \cA_\theta)$, it follows from Lemma~\ref{lem:Powers.integration-homogeneous-symbols} that we define a homogeneous symbol in $S_{m(j)}(\R^n;\cA_\theta)$ with $m(j)=(-w-j)+w(z+1)=wz-j$ by letting
\begin{equation}
 \rho_{wz-j}(z;\xi):= \frac{1}{2i\pi}\int_{|\xi|^w\Gamma} \lambda^z \sigma_{-w-j}(\xi;\lambda) d\lambda, \qquad \xi \neq 0.
 \label{eq:Powers.rhowj-formula-xiwGamma} 
\end{equation}

Let $\xi\in \R^n\setminus 0$. Recall that $\sigma_{-w}(\lambda;\xi)=(\rho_w(\xi)-\lambda)^{-1}$, and so $\lambda \rightarrow \sigma_{-w}(\xi;\lambda)$ is a holomorphic map from $\C\setminus \Sp(\rho_w(\xi))$ to $\cA_\theta$. Using~(\ref{eq:Powers.symbol-resolvent2}) and arguing by induction then shows that, for every $j\geq 0$, the map $\lambda \rightarrow \sigma_{-w-j}(\xi;\lambda)$ is holomorphic on $\C\setminus \Sp(\rho_w(\xi))$. 
We also know from~(\ref{eq:Powers.spectrum-rhow}) that $\Sp(\rho_w(\xi))\subset [c|\xi|^w,c'|\xi|^w]$. As $|\xi|^w\Gamma \subset \C\setminus (r|\xi|^w,\infty)\subset\C \setminus [c|\xi|^w,\infty)$, we see that the contour $|\xi|^w\Gamma$ is contained in $\C\setminus \Sp(\rho_w(\xi))$. Therefore, by the 3rd part of Proposition~\ref{prop:AppendixA.Hold}, 
in the formula~(\ref{eq:Powers.rhowj-formula-xiwGamma}) we may replace the contour $|\xi|^w\Gamma$ by any clockwise-oriented Jordan curve $\gamma_{\xi}$ in $\C\setminus (-\infty,0]$ whose interior contains $\Sp(\rho_w(\xi))$. That is, we have
\begin{equation}
 \rho_{wz-j}(z;\xi)=  \frac{1}{2i\pi}\int_{\gamma_{\xi}} \lambda^z \sigma_{-w-j}(\xi;\lambda) d\lambda, \qquad j\geq 0.
 \label{eq:Powers.rhowz-integral-gammaxi} 
\end{equation}
In particular, for $j=0$ we get 
\begin{equation*}
 \rho_{wz}(z;\xi)= \frac{1}{2i\pi}\int_{\gamma_{\xi}} \lambda^z \big(\rho_{w}(\xi)-\lambda\big)^{-1} d\lambda= \rho_w(\xi)^z.
 \end{equation*}

Let us show that $\rho(z;\xi) \sim \sum_{j\geq 0} \rho_{wz-j}(z;\xi)$.  By Remark~\ref{rmk:Parameter.classical-symbol-asymptotics-equivalent-conditions} the asymptotic expansion 
$\sigma(\xi;\lambda)\sim \sum_{j\geq 0} \sigma_{-w-j}(\xi;\lambda)$ implies that, for all $N\geq 0$ and $J\geq N+w$, there is $r_{NJ}(\xi; \lambda)$ in $\stS^{-w-N,-1}(\R^n\times \Lambda(P); \cA_\theta)$ such that 
\begin{equation*}
 \sigma(\xi;\lambda) = \sum_{j<J}\big(1-\chi(\xi)\big)\sigma_{-w-j}(\xi;\lambda) + r_{NJ}(\xi; \lambda), 
\end{equation*}
where $\chi(\xi)\in C^\infty_c(\R^n)$ is as in Lemma~\ref{lem:Parameter.homogeneous-symbol-estimate}. In our setting this means that $\chi(\xi)=1$ for $|\xi|<(c^{-1}R)^{\frac1w}$ for some $R\geq r_0$. Moreover, it follows from Lemma~\ref{lem:Parameter.homogeneous-symbol-estimate} that $(1-\chi(\xi))\sigma_{-w-j}(\xi)$ is contained in $\stS^{-j,-1}(\R^n\times \Lambda(P);\cA_\theta)$. In any case, for all $\xi \in \R^n$, we have 
\begin{align*}
 \rho(z;\xi) &= \frac{1}{2i\pi} \int_\Gamma \lambda^{z} \sigma(\xi;\lambda) d\lambda\\ 
 & = \sum_{j<J}\frac{1}{2i\pi} \int_\Gamma \lambda^{z}  \big(1-\chi(\xi)\big)\sigma_{-w-j}(\xi;\lambda) d\lambda +  
\frac{1}{2i\pi} \int_\Gamma \lambda^{z}   r_{NJ}(\xi; \lambda)d\lambda.
 \end{align*}
 
 It follows from Lemma~\ref{lem:Powers.integration-standard-symbols} that the integral $\int_\Gamma \lambda^{z}   r_{NJ}(\xi; \lambda)d\lambda$ converges in $\stS^{-w-N}(\R^n; \cA_\theta)$. In addition, let $j\geq 0$ and $\xi\in \R^n$. We claim that
 \begin{equation}
 \frac{1}{2i\pi} \int_\Gamma \lambda^{z}  \big(1-\chi(\xi)\big)\sigma_{-w-j}(\xi;\lambda) d\lambda = \left(1-\chi(\xi)\right)\rho_{wz-j}(z;\xi).
 \label{eq:Powers.cut-off-symbols-integral} 
\end{equation}
 If $|\xi|\leq (c^{-1}R)^{\frac1w}$, then both sides of the equation are zero, and so the equation holds. Suppose now that $|\xi|>(c^{-1}R)^{\frac1w}$. This implies that 
 $c|\xi|^w>R\geq r_0>r$. As $\Sp(\rho_w(\xi))\subset [c|\xi|^w,c'|\xi|^w]$ we see that $\Gamma$ is contained in $\C\setminus \Sp(\rho_w(\xi))$. Therefore, in the same way as above, in the integral in~(\ref{eq:Powers.cut-off-symbols-integral}) we may replace the contour $\Gamma$ by any Jordan curve $\gamma_{\xi}$ as in~(\ref{eq:Powers.rhowz-integral-gammaxi}). Thus, 
\begin{align*}
  \frac{1}{2i\pi} \int_\Gamma \lambda^{z}  \big(1-\chi(\xi)\big)\sigma_{-w-j}(\xi;\lambda) d\lambda & = 
  \frac{1}{2i\pi} \int_{\gamma_{\xi}}  \lambda^{z}  \big(1-\chi(\xi)\big)\sigma_{-w-j}(\xi;\lambda) d\lambda\\
    & = \big(1-\chi(\xi)\big) \cdot \frac{1}{2i\pi} \int_{\gamma_{\xi}}  \lambda^{z}  \sigma_{-w-j}(\xi;\lambda) d\lambda\\
 & = \left(1-\chi(\xi)\right)\rho_{wz-j}(z;\xi). 
\end{align*}
This confirms our claim. 

It follows from all this that, for all $N\geq 1$ and $J\geq N+w$, we have 
\begin{equation*}
 \rho(z;\xi) =  \sum_{j<J}  \left(1-\chi(\xi)\right)\rho_{wz-j}(z;\xi) \quad \bmod \stS^{-w-N}(\R^n;\cA_\theta). 
\end{equation*}
In the same way as in the proof of Lemma~\ref{lem:Parameter.standard-symbol-asymptotics-qualitative-equivalence} this implies that 
$\rho(z;\xi) \sim \sum_{j \geq 0}  (1-\chi(\xi)) \rho_{wz-j}(z;\xi)$ in the sense that $\rho(z;\xi)-\sum_{j <N}  (1-\chi(\xi)) \rho_{wz-j}(z;\xi) \in \stS^{w\Re z-N}(\R^n;\cA_\theta)$ for all $N\geq 0$. 
It then follows that $\rho(z;\xi) \sim \sum_{j \geq 0}\rho_{wz-j}(z;\xi)$ in the sense of~(\ref{eq:Symbols.classical-estimates}) (\emph{cf}.~\cite[Remark~3.21]{HLP:Part1}). In particular, this shows that $\rho(z;\xi)\in S^{wz}(\R^n;\cA_\theta)$, and so the power $P^z=P_{\rho(z;\cdot)}$ is an operator in $\Psi^{wz}(\cA_\theta)$. This proves the theorem when $\Re z<0$. 

Suppose now that $\Re z\geq 0$, and let $m\in \N$ be such that $\Re z -m<0$. The group property~(\ref{eq:Powers.group}) implies that $P^z=P^mP^{z-m}$ on $\cA_\theta$. Here $P^m \in \Psi^{mw}(\cA_\theta)$, and, as  $\Re z -m<0$, the first part of the proof shows that $P^{z-m}\in \Psi^{w(z-m)}(\cA_\theta)$. It then follows from Proposition~\ref{prop:Composition.composition-PsiDOs} that $P^{z}\in \Psi^{wz}(\cA_\theta)$.

It remains to show that the homogeneous symbols $\rho_{wz-j}(z;\xi)$ of $P^z$ are given by the formula~(\ref{eq:Powers.homogeneous-symbols}). By the first part of the proof the formula holds true when $\Re z<0$.  Let $\xi\in \R^n\setminus 0$. We observe that the integrals $(2i\pi)^{-1}\int_{\gamma_{\xi}} \lambda^z \sigma_{-w-j}(\xi;\lambda)d\lambda$ give rise to holomorphic maps from $\C$ to $\cA_\theta$. 
In particular, this implies that $z\rightarrow \rho_{wz-j}(z;\xi)$ is a holomorphic map from the half-space $\{\Re z<0\}$ to $\cA_\theta$ for every~$j\geq 0$. 
Furthermore,  the equality $P^z=P^{m}P^{z-m}$ and Proposition~\ref{prop:Composition.composition-PsiDOs} show that, for any $m\in \N$ and $\Re z <m$, we have 
\begin{equation*}
 \rho_{wz-j}(z;\xi) = \sum_{k+l+|\alpha|=j} \frac{1}{\alpha!} \partial_\xi^\alpha \rho_{mw-k}(m;\xi)\delta^\alpha \rho_{w(z-m)-l}(z-m;\xi).
\end{equation*}
It then follows that the map $z\rightarrow  \rho_{wz-j}(z;\xi)$ is holomorphic on every half-space $\{\Re z<m\}$, $m\in \N$, and hence it is holomorphic on all $\C$. 
Therefore, if we fix $\xi \in \R^n\setminus 0$, then both sides of~(\ref{eq:Powers.homogeneous-symbols}) are holomorphic maps on $\C$. As they agree on the half-space $\{\Re z<0\}$ they agree everywhere. The proof is complete. 
\end{proof}

\begin{remark}
 It was mentioned without proof in~\cite{FGK:MPAG17} that complex powers of positive elliptic \psidos\ on noncommutative 2-tori are \psidos. Therefore, Theorem~\ref{thm:Powers.powers} confirms the claim of~\cite{FGK:MPAG17}.
\end{remark}

We end the section with the following consequence of Theorem~\ref{thm:Powers.powers}. 

\begin{theorem}\label{thm:Powers.powers-absolute-value}
 Let $P\in \Psi^{w}(\cA_\theta)$ be elliptic and have principal symbol $\rho_w(\xi)$. 
\begin{enumerate}
 \item The absolute value $|P|:=\sqrt{P^*P}$ is in $\Psi^{w}(\cA_\theta)$ and has principal symbol $|\rho_w(\xi)|$. 
 
 \item For every $z\in \C$, the operator $|P|^z$ is in $\Psi^{wz}(\cA_\theta)$ and has principal symbol $|\rho_w(\xi)|^z$. 
\end{enumerate}
\end{theorem}
\begin{proof}
By Proposition~\ref{prop:Composition.composition-PsiDOs} and Proposition~\ref{prop:Adjoints.rhostar-formal-adjoint} the operator $P^*P$ is in $\Psi^{2w}(\cA_\theta)$ and has principal symbol $\rho_{w}(\xi)^*\rho_w(\xi)$. 
In particular, its principal symbol is positive and invertible, i.e., $P^*P$ is elliptic and it satisfies (P1). It also satisfies (P2), since $P^*P$ is formally selfadjoint and $\acoup{P^*Pu}{u}= \acoup{Pu}{Pu}\geq 0$ for all $u \in \cA_\theta$, which implies that the eigenvalues of $P^*P$ must be non-negative. 
Therefore, by Theorem~\ref{thm:Powers.powers} the absolute value $|P|:=\sqrt{P^*P}$ is in $\Psi^{w}(\cA_\theta)$ and has principal symbol $\sqrt{\rho_w(\xi)^*\rho_w(\xi)}=|\rho_w(\xi)|$. 
Note that $|P|$ satisfies (P1) and (P2) as well, and so by Theorem~\ref{thm:Powers.powers} again, for every $z\in \C$, the operator $|P|^z$ is in $\Psi^{wz}(\cA_\theta)$ and has principal symbol $|\rho_w(\xi)|^z$. The proof is complete. 
\end{proof}

\appendix

\section{Improper Integrals with Values in Locally Convex Spaces} \label{sec:Improper-Integrals} 
In this appendix we gather a few facts about convergence of improper integrals with values in some locally convex space. 

In what follows we let $\sE$ be a locally convex space. Given any continuous map $f:[a,b]\rightarrow \sE$ its Riemann integral 
$\int_a^b f(\lambda)d\lambda$ is defined in the same way as the Riemann integral of scalar-valued functions (see, e.g., \cite{FJ:JDE68, Ha:BAMS82}; see also~\cite[Appendix~B]{HLP:Part1}). 

Given a continuous map $f:[a,\infty) \rightarrow \sE$ we say that the integral $\int_a^\infty f(\lambda) d\lambda$ \emph{converges in} $\sE$ when $\lim_{b\rightarrow \infty}\int_a^bf(\lambda)d\lambda$ exists. In this case we set
\begin{equation*}
 \int_a^\infty f(\lambda) d\lambda = \lim_{b\rightarrow \infty}\int_a^bf(\lambda)d\lambda. 
\end{equation*}

\begin{proposition}\label{prop:Integration.mapping}
 Suppose that $T:\sE\rightarrow \sF$ is a continuous linear map with values in some locally convex space $\cF$. Let $f:[a,\infty) \rightarrow \sE$ be a continuous map whose integral $\int_a^\infty f(\lambda)d\lambda$ converges in $\sE$. Then the integral $\int_a^\infty T[f(\lambda)]d\lambda$ converges in $\sF$, and we have 
\begin{equation*}
  \int_a^\infty T\big[f(\lambda)\big] d\lambda = T \Big( \int_a^\infty f(\lambda)d\lambda\Big). 
\end{equation*}
\end{proposition}
\begin{proof}
 The composition $T[f(\lambda)]$ is a continuous map from $[a,\infty)$ to $\sF$. Moreover, we have $\int_a^b T[f(\lambda)]d\lambda = T(\int_a^b f(\lambda)d\lambda)$ for $0<a<b$ (see, e.g.,~\cite[Proposition~B.5]{HLP:Part1}). Thus, in $\sF$ we have 
\begin{equation*}
 \lim_{b\rightarrow \infty} \int_a^b T\big[f(\lambda)\big] d\lambda =  \lim_{b\rightarrow \infty} T \Big( \int_a^b f(\lambda)d\lambda\Big) = 
 T \Big( \int_a^\infty f(\lambda)d\lambda\Big). 
\end{equation*}
This proves the result. 
 \end{proof}

\begin{remark}
 In the special case $\cE=\C$ the above results shows that, for every $\varphi\in \sE'$, the integral $\int_0^\infty \varphi[f(\lambda)] d\lambda$ is convergent, and we have 
 \begin{equation}
 \int_a^\infty \varphi\big[f(\lambda)\big] d\lambda = \varphi \Big( \int_a^\infty f(\lambda)d\lambda\Big). 
 \label{eq:Integrals.linear-forms}
\end{equation}
Note that this property uniquely determines the value of $ \int_a^\infty f(\lambda)d\lambda$. 
\end{remark}

\begin{proposition}\label{prop:Integrals.Change-variable}
 Let $f:[a,\infty)\rightarrow \sE$ be a continuous map such that $\int_a^\infty f(\lambda)d\lambda$ converges in $\sE$. Let $\phi:[a',\infty)\rightarrow [a,
 \infty)$ be a $C^1$-diffeomorphism. Then $\int_{a'}^\infty \phi'(\lambda) (f\circ \phi)(\lambda) d\lambda$ converges in $\sE$, and we have 
 \begin{equation*}
 \int_{a'}^\infty \phi'(\lambda) (f\circ \phi)(\lambda) d\lambda = \int_a^\infty f(\lambda)d\lambda. 
\end{equation*}
\end{proposition}
\begin{proof}
 The assumption that $\phi$ is a diffeomorphism ensures us that $\phi$ is increasing, $\phi(a')= a$,  and $
 \lim_{\lambda \rightarrow \infty}\phi(\lambda)=\infty $. Therefore, by using \cite[Proposition~B.7]{HLP:Part1} we get 
\begin{equation*}
 \int_{a'}^b \phi'(\lambda) (f\circ \phi)(\lambda) d\lambda = \int_{a}^{\phi(b)} f(\lambda)d\lambda \longrightarrow  \int_{a}^{\infty}f(\lambda)d\lambda \qquad 
 \text{as $b \rightarrow \infty$}. 
\end{equation*}
This proves the result. 
\end{proof}

It is convenient to interpret the above notion of convergence in term of Lebesgue integrations of maps with values in locally convex spaces. A natural setting for this notion of integration is the setting of quasi-complete locally convex Suslin spaces (see~\cite{Th:TAMS75}; see also~\cite[Appendix~B]{HLP:Part1}). This encompasses a large class of locally convex spaces, including separable Fr\'echet spaces such as $\cA_\theta$. However, in the precise setting of this paper, we can bypass this as follows. 

In what follows, given any Borel set $Y\subset \C$, we say that a measurable map $f:Y\rightarrow \sE$  is \emph{absolutely integrable} when
\begin{equation}
 \int_Y \mathfrak{p}\big[f(\lambda)\big] d\lambda <\infty \qquad \text{for every continuous semi-norm $\mathfrak{p}$ on $\sE$}. 
 \label{eq:Integrals.integrability-condition}
\end{equation}

Recall that a locally convex space is quasi-complete when every bounded Cauchy net is convergent. For instance, Fr\'echet spaces and their weak duals, as well as inductive limits of Fr\'echet spaces are quasi-complete (see, e.g., \cite{Tr:AP67}). 

\begin{proposition}
Assume that $\sE$ is quasi-complete. Then, for every absolutely integrable continuous map $f:[a,\infty)\rightarrow \sE$, the integral $\int_a^\infty f(\lambda) d\lambda$ converges in $\sE$. 
\end{proposition}
\begin{proof}
Let $f:[a,\infty)\rightarrow \sE$ be an absolutely integrable continuous map.  Let $\fp$ be a continuous semi-norm on $\sE$. Given any $b\geq a$ we have
\begin{equation*}
 \fp\Big[ \int_a^{b} f(\lambda)d\lambda\Big] \leq   \int_a^{b}\fp\big[ f(\lambda) \big]d\lambda \leq   \int_a^{\infty}\fp\big[ f(\lambda) \big]d\lambda <\infty. 
\end{equation*}
Similarly, given any $c\geq 0$, we have 
\begin{equation*}
 \fp\Big[ \int_a^{b+c} f(\lambda)d\lambda-  \int_a^{b} f(\lambda)d\lambda\Big] = \fp\Big[ \int_b^{b+c} f(\lambda)d\lambda\Big] \leq  
  \int_b^{b+c}\fp\big[ f(\lambda) \big]d\lambda \leq  \int_b^{\infty}\fp\big[ f(\lambda) \big]d\lambda. 
\end{equation*}
The absolute-integrability condition~(\ref{eq:Integrals.integrability-condition}) ensures us that $ \int_b^{\infty}\fp\big[ f(\lambda) \big]d\lambda \rightarrow 0$ as $b\rightarrow \infty$. Therefore, we see that, for every continuous semi-norm $\fp$ on $\sE$, we have 
\begin{equation*}
  \sup_{c\geq 0}\fp\Big[ \int_a^{b+c} f(\lambda)d\lambda-  \int_a^{b} f(\lambda)d\lambda\Big]\longrightarrow 0 \qquad \text{as $b\rightarrow \infty$}. 
\end{equation*}
This shows that $\{ \int_a^{b} f(\lambda)d\lambda; b\geq a\}$ is a bounded Cauchy net in $\sE$. As $\cE$ is quasi-complete, it then follows that $\lim_{b\rightarrow \infty}  \int_a^{b} f(\lambda)d\lambda$ exists in $\sE$. That is, the integral $\int_a^\infty f(\lambda) d\lambda$ converges in $\sE$. The proof is complete. 
\end{proof}

As an immediate consequence we have the following dominated convergence result. 

\begin{proposition}\label{prop:Integrals.dominated-convergence} 
 Suppose that $\sE$ is quasi-complete. Let $f:[a,\infty) \rightarrow \sE$ be a continuous map such that, for every continuous semi-norm  $\fp$ on $\sE$, there is  a function $g_\fp(\lambda) \in L^1[a,\infty)$ such that 
\begin{equation*}
 \fp\big[ f(\lambda)\big] \leq g_\fp(\lambda) \qquad \forall \lambda \geq a. 
\end{equation*}
Then $f(\lambda)$ is absolutely integrable, and the integral $\int_a^\infty f(\lambda) d\lambda$ converges in $\sE$. 
\end{proposition}

The above considerations extend \emph{verbatim} to integrals along rays $L_\phi(a)=e^{i\phi}[a,\infty)$ with $a\geq 0$ and $0\leq \phi< 2\pi$. If $f: L_\phi(a)\rightarrow \sE$ is a continuous map, then the integral $\int_{L_\phi(a)} f(\lambda)d\lambda$ converges in $\sE$ when $\lim_{b\rightarrow \infty} \int_a^b f(e^{i\phi}\lambda) d(e^{i\phi}\lambda)$ exists in $\sE$. In this case, and if $L_\phi(a)$ is outward-oriented (i.e., it is directed toward $\infty$), then we set 
\begin{equation*}
 \int_{L_\phi(a)} f(\lambda)d\lambda = \lim_{b\rightarrow \infty} \int_a^b f(e^{i\phi}\lambda) d(e^{i\phi}\lambda). 
\end{equation*}
There is a change of sign when $L_\phi(a)$ is inward-oriented. 

More generally, we can consider \emph{keyhole contours}, by which we mean contours of the form, 
\begin{equation*}
 \Gamma = L_{\phi_1}(a) \cup C_{\phi_1,\phi_2}(a) \cup L_{\phi_2}(a), \qquad a\geq 0, \quad \phi_1>\phi_2 \geq \phi_1-2\pi,
\end{equation*}
where $C_{\phi_1,\phi_2}(a)$ is the arc of circle $\{ae^{i\phi}; \ \phi_1\geq \phi \geq \phi_2\}$. The orientation of $\Gamma$ is chosen so that it agrees with the clockwise-orientation on  $C_{\phi_1,\phi_2}(a)$. Thus, $L_{\phi_1}(a)$ is inward-oriented, whereas $L_{\phi_2}(a)$ is outward-oriented. For instance when $\phi_1=\pi$, $\phi_2=0$ and $a=0$ the contour $\Gamma$ is just the real line $(-\infty,\infty)$.

\begin{figure}[h]
\begin{minipage}{0.35\linewidth}
\centering{\def\svgwidth{\columnwidth}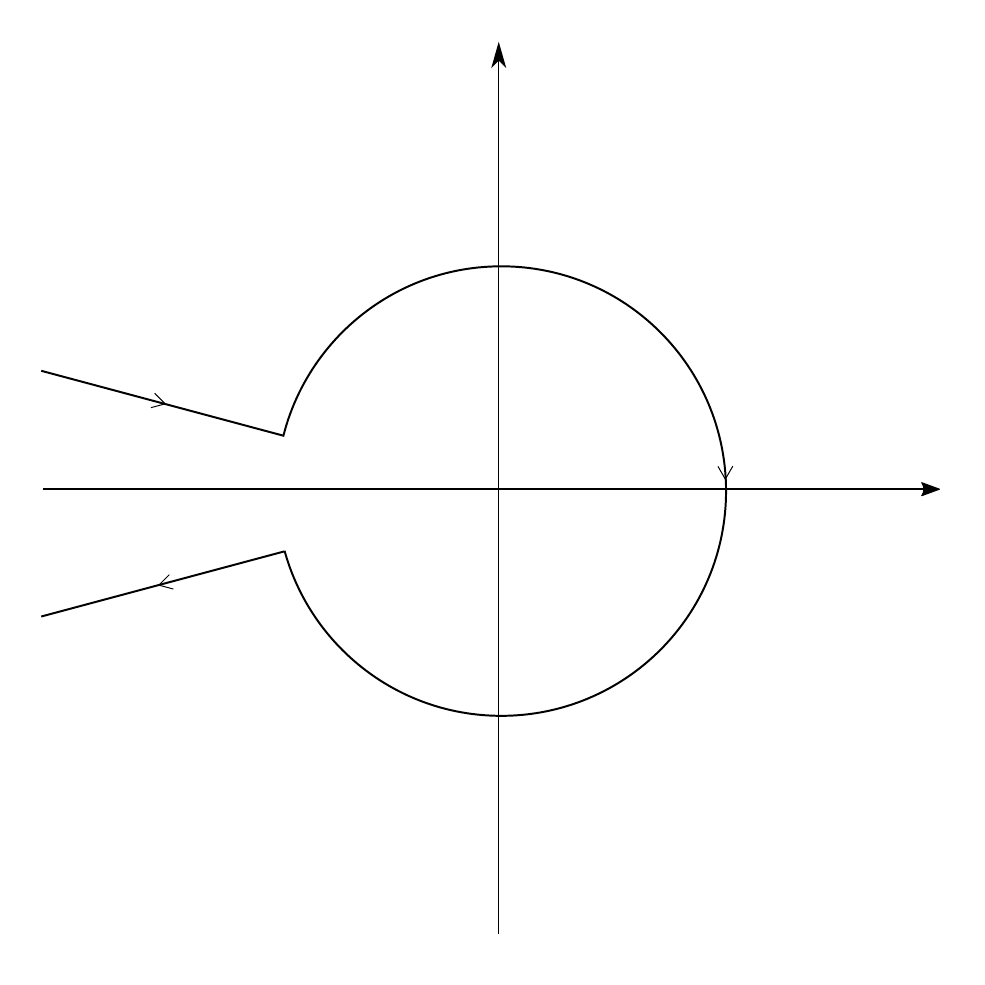}
\end{minipage}
\begin{minipage}{0.35\linewidth}
\centering{\def\svgwidth{\columnwidth}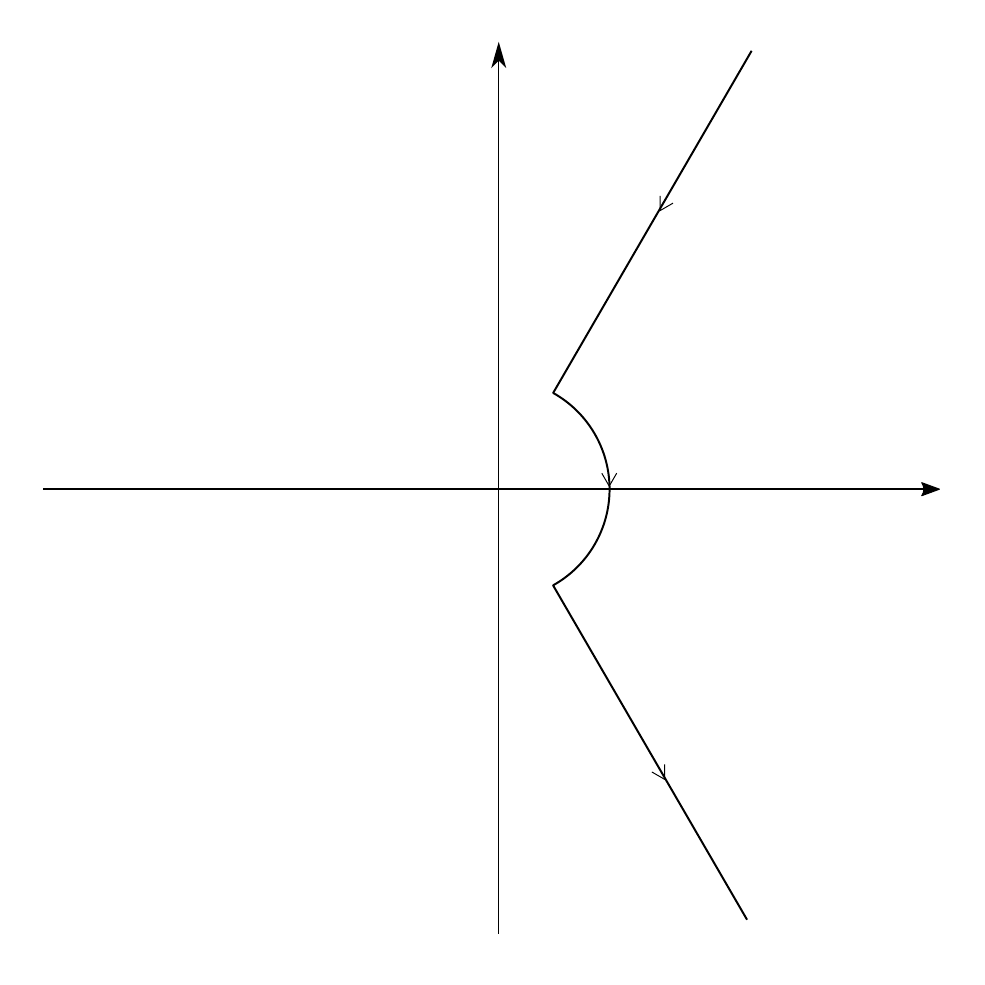}
\end{minipage}
\caption{Examples of Keyhole Contours}
\end{figure} 

Given a continuous map $f:\Gamma \rightarrow \sE$ we say that the contour integral $\int_\Gamma f(\lambda) d\lambda$ converges in $\sE$ when 
the integrals $\int_{L_{\phi_1}(a)} f(\lambda)d\lambda$ and  $\int_{L_{\phi_2}(a)} f(\lambda)d\lambda$ both converge in $\sE$. In this case we set
\begin{equation*}
 \int_\Gamma f(\lambda)d\lambda = \int_{L_{\phi_1}(a)} f(\lambda)d\lambda + \int_{C_{\phi_1,\phi_2}(a)} f(\lambda)d\lambda + \int_{L_{\phi_2}(a)} f(\lambda)d\lambda,
\end{equation*}
where the middle integral in the r.h.s.~is defined as a Riemann integral.

At the exception of Proposition~\ref{prop:Integrals.Change-variable} all the previous results of this section hold \emph{verbatim} for this kind of improper integrals. We also have the following version of Proposition~\ref{prop:Integrals.Change-variable}. 

\begin{proposition}\label{prop:Integrals.scaling-contour}
 Suppose that $\Gamma$ is a keyhole contour, and let $f:\Gamma\rightarrow \sE$ be a continuous map whose integral $\int_\Gamma f(\lambda)d\lambda$ converges in $\cE$. Then, for every $t>0$, the integral $\int_{t\Gamma} f(t^{-1}\lambda)d\lambda$ converges in $\sE$, and we have
\begin{equation*}
  \int_{t\Gamma} f(t^{-1}\lambda)d\lambda = t\int_{\Gamma} f(\lambda)d\lambda. 
\end{equation*}
\end{proposition}

We are mostly interested in the case where $f(\lambda)$ is an $\Hol^d(\Lambda)$-map, where $\Lambda$ is some open pseudo-cone containing the contour $\Gamma$. Note that keyhole contours themselves are pseudo-cones. 

\begin{proposition}\label{prop:AppendixA.Hold} 
 Assume that $\sE$ is quasi-complete.  Let $f(\lambda)\in \Hol^d(\Lambda;\sE)$, $d<-1$, where $\Lambda$ is some open pseudo-cone. In addition, let $\Gamma$ be a keyhole contour contained in $\Lambda$. 
 \begin{enumerate}
 \item The contour integral $\int_\Gamma f(\lambda) d\lambda$ converges in $\sE$. 
 
 \item Let $T:\sE\rightarrow \sF$ be a continuous linear map with values in some other locally convex space $\cF$. Then, the integral $\int_\Gamma T[f(\lambda)] d\lambda$ converges in $\sF$, and we have 
 \begin{equation*}
  \int_\Gamma T\big[f(\lambda)\big] d\lambda = T \Big( \int_\Gamma f(\lambda)d\lambda\Big). 
\end{equation*}

\item Let $\Gamma'\subset \Lambda$ be either a keyhole contour or a clockwise-oriented Jordan curve, and assume 
there is a domain $\Omega\subset \Lambda$ such that $\partial \Omega = \Gamma \cup \Gamma'$. Then, we have 
\begin{equation*}
 \int_{\Gamma}  f(\lambda)d\lambda =  \int_{\Gamma'}  f(\lambda)d\lambda. 
\end{equation*}
\end{enumerate}
\end{proposition}
\begin{proof}
 As $\Gamma$ is a pseudo-cone~$\subsubset \Lambda$, for every continuous semi-norm $\fp$ on $\sE$, there is $C_{\Gamma \fp}>0$ such that 
\begin{equation*}
 \fp\big[ f(\lambda)\big] \leq \left(1+|\lambda|\right)^d \qquad \forall \lambda \in \Gamma.   
\end{equation*}
 As $d<-1$ and $\sE$ is quasi-complete, it follows from Proposition~\ref{prop:Integrals.dominated-convergence} that $f(\lambda)$ is absolutely integrable and the integral $\int_\Gamma f(\lambda)d\lambda$ converges in $\sE$. This proves the first part. Combining it with Proposition~\ref{prop:Integration.mapping} gives the 2nd part. 
 
 It remains to prove the 3rd part. Let $\Gamma'\subset \Lambda$ be either a keyhole contour or a clockwise-oriented Jordan curve, and assume   
there is a domain $\Omega\subset \Lambda$ such that $\partial \Omega = \Gamma \cup \Gamma'$. In addition, let $\varphi \in \sE'$. The composition $\varphi \circ f(\lambda)$ is contained in $\Hol^d(\Lambda)$, and so this is an integrable holomorphic function on $\Omega$. Therefore, if we orient $\partial \Omega$ suitably, then we get
\begin{equation*}
 0 = \int_{\partial \Omega} \varphi\left[f(\lambda)\right] d\lambda = \int_{\Gamma}\varphi\left[f(\lambda)\right] - \int_{\Gamma'} \varphi\left[f(\lambda)\right]d\lambda. 
\end{equation*}
 Combining this with~(\ref{eq:Integrals.linear-forms}) we obtain
 \begin{equation*}
  \varphi \Big( \int_{\Gamma} f(\lambda)d\lambda\Big) =   \varphi \Big( \int_{\Gamma'} f(\lambda)d\lambda\Big) \qquad \forall \varphi \in \sE'. 
\end{equation*}
 As $\cE'$ separates the point of $\sE$ by Hahn-Banach theorem, we see that $ \int_{\Gamma} f(\lambda)d\lambda=  \int_{\Gamma'} f(\lambda)d\lambda$. This proves the 3rd part. The proof is complete. 
\end{proof}

\section{Proof of Proposition~\ref{prop:topo-smoothing}} \label{app:topo-smoothing}
In this Appendix, we include a proof of Proposition~\ref{prop:topo-smoothing}. Recall that the topology of $\cL(\cA_\theta',\cA_\theta)$ is generated by the semi-norms,
\begin{equation*} 
R\longrightarrow \sup_{u\in \sB}\| \delta^\alpha(Ru) \| , \qquad \sB\subset\cA_\theta'\text{ bounded} , \quad \alpha\in \N_0^n.
\end{equation*}
We have the following description of bounded sets in $\cA_\theta'$. 

\begin{lemma} \label{lem:Sobolev.inductive-limit-regularity}
A subset $\sB\subset \cA_\theta'$ is bounded in $\cA_\theta'$ if and only if there $s\in \R$ such that $\sB$ is contained and bounded in $\cH_\theta^{(s)}$. 
\end{lemma}
\begin{proof}
This a a special case of a general result for inductive limits of compact inclusions of Banach spaces, but we shall give a proof for reader's convenience.
By Proposition~\ref{prop:Sobolev.Hs-inclusion-cAtheta} the strong topology of $\cA_\theta'$ agrees with the inductive limit of the $\cH_\theta^{(s)}$-topologies. Equivalently, this is the strongest locally convex topology with respect to which the inclusion of $\cH_\theta^{(s)}$ into $\cA_\theta'$ is continuous for every $s\in \R$. Thus, a basis of neighborhoods of the origin in $\cA_\theta'$ consists of all convex balanced sets $\cU$ such that $\cU\cap \cH_\theta^{(s)}$ is a neighborhood of the origin in $\cH_\theta^{(s)}$ for every $s\in \R$. 

Bearing this in mind, a subset $\sB\subset \cA_\theta'$ is bounded when, given any neigborhood $\sU$ of the origin in $\cA_\theta'$, there is $t>0$ such that $t\sB \subset \sU$. As the inclusion of  $\cH_\theta^{(s)}$, $s\in \R$, into $\cA_\theta'$ is continuous, it is immediate that any bounded set of $\cH^{(s)}_\theta$ is bounded in $\cA_\theta'$. 

Conversely, let $\sB$ be a bounded set of $\cA_\theta'$. With a view toward contradiction assume that $\sB$ is not a bounded set in any $\cH_\theta^{(s)}$. Given $s\in \R$ and $r>0$ let us denote by $B^{(s)}(r)$ the open ball in $\cH_\theta^{(s)}$ of radius $r$ about the origin. The above assumption means that, for every $s\in \R$, we cannot find $r>0$ such that $\sB \subset B^{(s)}(r)$. In particular, there is $u^{(0)}\in \sB$ such that $u^{(0)}\not\in B^{(0)}(1)$, i.e., $\sum |u_k^{(0)}|^2>1$, and so there is an integer $N_0$ such that $\sum_{|k|\leq N_0}  |u_k^{(0)}|^2>1$. Likewise, there is $u^{(1)}\in \sB$ such that $\sum \brak{k}^{-2} |u_k^{(1)}|^2>2 \sum_{|k|\leq N_0}  |u_k^{(0)}|^2$, and so there is also an integer $N_1$ such that $\sum_{|k|\leq N_1}  \brak{k}^{-2}|u_k^{(1)}|^2> 2 \sum_{|k|\leq N_0}  |u_k^{(0)}|^2$. Repeating this argument allows us to construct sequences $(N_\ell)_{\ell \geq 0}\subset \N_0$ and $(u^{(\ell)})_{\ell \geq 0}\subset \sB$ such that 
\begin{equation} \label{eq:Appendix.sequence-construction}
       \sum_{|k|\leq N_{\ell+1}}  \brak{k}^{-2(\ell+1)}\big|u_k^{(\ell+1)}\big|^2> 2  \sum_{|k|\leq N_{\ell}}  \brak{k}^{-2\ell}\big|u_k^{(\ell)}\big|^2 \qquad \text{for all $\ell\geq 0$}. 
\end{equation}

For $\ell\geq 0$ set $\mu_\ell = 2^{-\ell}  \sum_{|k|\leq N_{\ell}}  \brak{k}^{-2\ell}|u_k^{(\ell)}|^2$, and define 
\begin{equation*}
 \sU= \bigcap_{\ell \geq 0} \biggl\{u \in \cA_\theta'; \  \sum_{|k|\leq N_{\ell}}  \brak{k}^{-2\ell}|u_{k}|^2<\mu_\ell \biggr\}. 
\end{equation*}
Note that~(\ref{eq:Appendix.sequence-construction}) ensures us that $(\mu_\ell)_{\ell\geq 0}$ is an increasing sequence.  

\begin{claim*}
 $\sU$ is a neighborhood of the origin in $\cA_\theta'$.  
\end{claim*}
\begin{proof}[Proof of the Claim]
 For $\ell\geq 0$ set $p_\ell(u)=  (\sum_{|k|\leq N_{\ell}}  \brak{k}^{-2\ell}|u_{k}|^2)^{\frac12}$, $u\in \cA_\theta'$. This defines semi-norms on $\cA_\theta'$. In particular, we see that $\sU$ is a convex balanced set. Therefore, in order to prove the claim it is enough to show that $\sU\cap \cH_\theta^{(s)}$ is a neighborhood of the origin in $\cH_\theta^{(s)}$ for every $s\in \R$. In fact, as the inclusion of $\cH_\theta^{(s)}$ into $\cH_\theta^{(s')}$ is continuous for $s>s'$, it is enough to show this for $s=-\ell_0$ with $\ell_0\in \N$. 
 
 Given $\ell_0\in \N$, let $\epsilon \in (0,\sqrt{\mu_{\ell_0}})$. We observe that if $u\in B^{(-\ell_0)}(\epsilon)$ and $\ell\geq \ell_0$, then 
\begin{equation*}
\sum_{|k|\leq N_\ell} \brak{k}^{-2\ell} |u_k|^2 \leq \sum_{k\in \Z^n} \brak{k}^{-2\ell_0} |u_k|^2<\epsilon^2<\mu_{\ell_0}\leq \mu_\ell.   
\end{equation*}
This shows that $B^{(-\ell_0)}(\epsilon)\subset \{u\in \cA_\theta'; \ p_\ell(u)<\sqrt{\mu_\ell}\}$ for $\ell\geq \ell_0$. Thus, 
\begin{equation*}
 \cU\cap B^{(-\ell_0)}(\epsilon) = \bigcap_{\ell < \ell_0}\left\{ u \in B^{(-\ell_0)}(\epsilon); \ p_\ell(u)<\sqrt{\mu_\ell}\right\}. 
\end{equation*}
As the semi-norms $p_\ell$ are continuous on $\cH_\theta^{(-\ell_0)}$, it then follows that $ \cU\cap B^{(-\ell_0)}(\epsilon)$ is an open set of $\cH_\theta^{(-\ell_0)}$, and so $\sU\cap \cH_\theta^{(-\ell_0)}$ is a neighborhood of the origin in $\cH_\theta^{(-\ell_0)}$ for every $\ell_0\in \N$. The proof is complete. 
\end{proof}

The above claim leads us to a contradiction as follows. As $\sB$ is bounded in $\cA_\theta'$ and $\sU$ is a neighborhood of the origin in $\cA_\theta'$, there is $t>0$ such that $t \sB\subset \sU$. In particular, $tu^{(\ell)}\in \sU$ for all $\ell \geq 0$. Now, choose $\ell$ so that $2^{-\ell}\leq t^2$. Then we have
\begin{equation*}
 \sum_{|k|\leq N_\ell}\brak{k}^{-2\ell} t^2|u_k^{(\ell)}|^2  \geq \sum_{|k|\leq N_\ell} 2^{-\ell}  \brak{k}^{-2\ell}  |u_k^{(\ell)}|^2 =\mu_\ell.
\end{equation*}
Thus,  $tu^{(\ell)}$ cannot be contained in $\sU$. This is a contradiction. Therefore, if $\sB$ is a bounded set of $\cA_\theta'$, then it must be a bounded set of $\cH_\theta^{(s)}$ for some $s\in \R$. This proves the result. 
\end{proof}

We are now in a position to prove Proposition~\ref{prop:topo-smoothing}. 

\begin{proof}[Proof of Proposition~\ref{prop:topo-smoothing}]
Given $s,t\in \R$,  the continuity of the inclusions $\cH_\theta^{(s)}\subset\cA_\theta'$ and $\cA_\theta\subset\cH_\theta^{(t)}$ implies that the natural embedding of $\cL(\cA_\theta',\cA_\theta)$ into  $\cL(\cH_\theta^{(s)},\cH_\theta^{(t)})$ is continuous, and so $\| \cdot \|_{s,t}$ is a continuous semi-norm on 
$\cL(\cA_\theta',\cA_\theta)$. 

Moreover, we know by Proposition~\ref{prop:Sobolev.Hs-inclusion-cAtheta} that the topology of $\cA_\theta$ is generated by the Sobolev norms $\|\cdot \|_{t}$, $t\in \R$. Therefore, the topology of 
$\cL(\cA_\theta',\cA_\theta)$ is generated by the semi-norms. 
\begin{equation*} 
R\longrightarrow \sup_{u\in \sB}\| Ru \|_t , \qquad \sB\subset\cA_\theta'\ \text{bounded} , \quad t\in \R.
\end{equation*}
Bearing this in mind, let $t\in \R$ and let $\sB$ be a bounded subset of $\cA_\theta'$.  By Lemma~\ref{lem:Sobolev.inductive-limit-regularity} there is $s\in \R$ such that $\sB$ is a bounded subset of $\cH_\theta^{(s)}$, i.e., there is $\epsilon>0$ such that $\sB\subset \{u\in \cH_\theta^{(s)};\ \|u\|_s \leq \epsilon\}$. Thus, for all $R\in \cL(\cA_\theta',\cA_\theta)$, we have 
\begin{equation*}
 \sup_{u \in \sB} \|Ru\|_t \leq \sup_{\substack{u \in \cH_\theta^{(s)}\\ \|u\|_s\leq \epsilon}} \|Ru\|_{t} 
 \leq \epsilon \sup_{\substack{u \in \cH_\theta^{(s)}\\ \|u\|_s\leq 1}} \|Ru\|_{t}= \epsilon \|R\|_{s,t}. 
\end{equation*}
All this ensures us that the norms $\| \cdot \|_{s,t}$, $s,t\in \R$, generate the topology of $\cL(\cA_\theta',\cA_\theta)$. The proof is complete. 
\end{proof}

\end{document}